\documentclass[10pt]{amsart}

\usepackage{amsmath}
\usepackage{amsthm}
\usepackage{amsopn}
\usepackage{amssymb}
\usepackage[all]{xy}
\usepackage{lscape,xcolor}
\usepackage{graphicx}
\usepackage{hyperref}
\usepackage{float}
\usepackage{mathtools}
\usepackage{enumerate}
\usepackage{soul}
\usepackage{stmaryrd}
\usepackage{verbatim}
\usepackage{rotating}
\usepackage{tikz-cd}
\usepackage{subcaption}
\usepackage{tcolorbox}
\usepackage{multicol}

\usepackage[normalem]{ulem}
\newcommand{\stkout}[1]{\ifmmode\text{\sout{\ensuremath{#1}}}\else\sout{#1}\fi}

\definecolor{darkspringgreen}{rgb}{0.09, 0.45, 0.27}

\allowdisplaybreaks[1]

\newcommand{\mmod}{\! \sslash \!}

\newcommand{\mc}[1]{\mathcal{#1}}

\newcommand{\mb}[1]{\mathbb{#1}}
\newcommand{\mr}[1]{\mathrm{#1}}

\newcommand{\abs}[1]{\lvert #1 \rvert}

\newcommand{\bra}[1]{\langle #1 \rangle}
\newcommand{\br}[1]{\overline{#1}}

\newcommand{\td}[1]{\widetilde{#1}}

\newcommand{\ZZ}{\mathbb{Z}}

\newcommand{\CC}{\mathbb{C}}
\newcommand{\QQ}{\mathbb{Q}}
\newcommand{\WW}{\mathbb{W}}
\newcommand{\FF}{\mathbb{F}}
\newcommand{\GG}{\mathbb{G}}
\newcommand{\MS}{\mathbb{S}}

\newcommand{\TMF}{\mathrm{TMF}}
\newcommand{\Tmf}{\mathrm{Tmf}}
\newcommand{\tmf}{\mathrm{tmf}}

\newcommand{\bo}{\mathrm{bo}}

\def \HF2{\mr{H}\FF_2}

\newcommand{\BP}{\mr{BP}}

\newcommand{\tb}{\bar{t}}
\newcommand{\bh}{{h}}
\newcommand{\si}{\sigma}
\newcommand{\sitd}{\widetilde{\sigma}}
\newcommand{\sibr}{\overline{\sigma}}
\newcommand{\Sirav}{S}

\def \AA0{\br{A \mmod A(0)}_*}
\def \AA2{A\mmod A(2)_*}

\hypersetup{%
  bookmarksnumbered=true,%
  bookmarks=true,%
  colorlinks=true,%
  linkcolor=blue,%
  citecolor=blue,%
  filecolor=blue,%
  menucolor=blue,%
  pagecolor=blue,%
  urlcolor=blue,%
  pdfnewwindow=true,%
  pdfstartview=FitBH}

\def\@url#1{{\tt\def~{\lower3.5pt\hbox{\char'176}}\def\_{\char'137}#1}}

\def\makeautorefname#1#2{\expandafter\def\csname#1autorefname\endcsname{#2}}
\makeautorefname{equation}{Equation}%
\makeautorefname{footnote}{footnote}%
\makeautorefname{item}{item}%
\makeautorefname{figure}{Figure}%
\makeautorefname{table}{Table}%
\makeautorefname{part}{Part}%
\makeautorefname{appendix}{Appendix}%
\makeautorefname{chapter}{Chapter}%
\makeautorefname{section}{Section}%
\makeautorefname{subsection}{Section}%
\makeautorefname{subsubsection}{Section}%
\makeautorefname{paragraph}{Paragraph}%
\makeautorefname{subparagraph}{Paragraph}%
\makeautorefname{theorem}{Theorem}%
\makeautorefname{theo}{Theorem}%
\makeautorefname{thm}{Theorem}%
\makeautorefname{addendum}{Addendum}%
\makeautorefname{addend}{Addendum}%
\makeautorefname{add}{Addendum}%
\makeautorefname{maintheorem}{Main theorem}%
\makeautorefname{mainthm}{Main theorem}%
\makeautorefname{corollary}{Corollary}%
\makeautorefname{claim}{Claim}%
\makeautorefname{corol}{Corollary}%
\makeautorefname{coro}{Corollary}%
\makeautorefname{cor}{Corollary}%
\makeautorefname{lemma}{Lemma}%
\makeautorefname{lemm}{Lemma}%
\makeautorefname{lem}{Lemma}%
\makeautorefname{sublemma}{Sublemma}%
\makeautorefname{sublem}{Sublemma}%
\makeautorefname{subl}{Sublemma}%
\makeautorefname{proposition}{Proposition}%
\makeautorefname{proposit}{Proposition}%
\makeautorefname{propos}{Proposition}%
\makeautorefname{propo}{Proposition}%
\makeautorefname{prop}{Proposition}%
\makeautorefname{property}{Property}
\makeautorefname{proper}{Property}
\makeautorefname{scholium}{Scholium}%
\makeautorefname{step}{Step}%
\makeautorefname{conjecture}{Conjecture}%
\makeautorefname{conject}{Conjecture}%
\makeautorefname{conj}{Conjecture}%
\makeautorefname{question}{Question}
\makeautorefname{questn}{Question}
\makeautorefname{quest}{Question}
\makeautorefname{ques}{Question}
\makeautorefname{qn}{Question}
\makeautorefname{definition}{Definition}%
\makeautorefname{defin}{Definition}%
\makeautorefname{defi}{Definition}%
\makeautorefname{def}{Definition}%
\makeautorefname{defn}{Definition}%
\makeautorefname{dfn}{Definition}%
\makeautorefname{notation}{Notation}
\makeautorefname{nota}{Notation}
\makeautorefname{notn}{Notation}
\makeautorefname{remark}{Remark}%
\makeautorefname{rema}{Remark}%
\makeautorefname{rem}{Remark}%
\makeautorefname{rmk}{Remark}%
\makeautorefname{rk}{Remark}%
\makeautorefname{remarks}{Remarks}%
\makeautorefname{rems}{Remarks}%
\makeautorefname{rmks}{Remarks}%
\makeautorefname{rks}{Remarks}%
\makeautorefname{example}{Example}%
\makeautorefname{examp}{Example}%
\makeautorefname{exmp}{Example}%
\makeautorefname{exmps}{Examples}%
\makeautorefname{exam}{Example}%
\makeautorefname{exa}{Example}%
\makeautorefname{algorithm}{Algorith}%
\makeautorefname{algo}{Algorith}%
\makeautorefname{alg}{Algorith}%
\makeautorefname{axiom}{Axiom}%
\makeautorefname{axi}{Axiom}%
\makeautorefname{ax}{Axiom}%
\makeautorefname{case}{Case}%
\makeautorefname{claim}{Claim}%
\makeautorefname{clm}{Claim}%
\makeautorefname{assumption}{Assumption}%
\makeautorefname{assumpt}{Assumption}%
\makeautorefname{conclusion}{Conclusion}%
\makeautorefname{concl}{Conclusion}%
\makeautorefname{conc}{Conclusion}%
\makeautorefname{condition}{Condition}%
\makeautorefname{condit}{Condition}%
\makeautorefname{cond}{Condition}%
\makeautorefname{construction}{Construction}%
\makeautorefname{construct}{Construction}%
\makeautorefname{const}{Construction}%
\makeautorefname{cons}{Construction}%
\makeautorefname{criterion}{Criterion}%
\makeautorefname{criter}{Criterion}%
\makeautorefname{crit}{Criterion}%
\makeautorefname{exercise}{Exercise}%
\makeautorefname{exer}{Exercise}%
\makeautorefname{exe}{Exercise}%
\makeautorefname{problem}{Problem}%
\makeautorefname{problm}{Problem}%
\makeautorefname{probm}{Problem}%
\makeautorefname{prob}{Problem}%
\makeautorefname{solution}{Solution}%
\makeautorefname{soln}{Solution}%
\makeautorefname{sol}{Solution}%
\makeautorefname{summary}{Summary}%
\makeautorefname{summ}{Summary}%
\makeautorefname{sum}{Summary}%
\makeautorefname{operation}{Operation}%
\makeautorefname{oper}{Operation}%
\makeautorefname{observation}{Observation}%
\makeautorefname{observn}{Observation}%
\makeautorefname{obser}{Observation}%
\makeautorefname{obs}{Observation}%
\makeautorefname{ob}{Observation}%
\makeautorefname{convention}{Convention}%
\makeautorefname{convent}{Convention}%
\makeautorefname{conv}{Convention}%
\makeautorefname{cvn}{Convention}%
\makeautorefname{warning}{Warning}%
\makeautorefname{warn}{Warning}%
\makeautorefname{note}{Note}%
\makeautorefname{fact}{Fact}%
\makeautorefname{conjtel}{Conjecture}%
\makeautorefname{gue}{Guess}%
\makeautorefname{goal}{Goal}%
\makeautorefname{figure}{Figure}%

  \makeatletter
                   \let\c@lemma\c@theorem
                  \makeatother

 \newtheorem{thm}[equation]{Theorem}
 \newtheorem{cor}[equation]{Corollary}
 \newtheorem{lem}[equation]{Lemma}
 \newtheorem{prop}[equation]{Proposition}

              \newtheorem{war}[equation]{Warning}

               \theoremstyle{definition}
 \newtheorem{defn}[equation]{Definition}
 \newtheorem{ex}[equation]{Example}
 
 \newtheorem{rmk}[equation]{Remark}

 \newtheorem{conjecture}[equation]{Conjecture}
 
 \newtheorem*{thm*}{Theorem}
 \newtheorem*{cor*}{Corollary}
 \newtheorem*{lem*}{Lemma}
 \newtheorem*{prop*}{Proposition}
  \newtheorem*{not*}{Notation}
    \newtheorem*{guess*}{Guess}

\newtheorem*{defn*}{Definition}
\newtheorem*{ex*}{Example}
\newtheorem*{exs*}{Examples}
\newtheorem*{rmk*}{Remark}
\newtheorem*{claim*}{Claim}

\newtheoremstyle{named}{}{}{\itshape}{}{\bfseries}{.}{.5em}{\thmnote{#3}}
\theoremstyle{named}
\newtheorem*{namedtheorem}{Theorem}

\numberwithin{equation}{subsection}
\numberwithin{figure}{section}

\makeatletter
\let\c@lem=\c@thm
\let\c@cor=\c@thm
\let\c@prop=\c@thm
\let\c@lem=\c@thm
\let\c@ex=\c@thm
\let\c@exs=\c@thm
\let\c@obs=\c@thm
\let\c@rmk=\c@thm
\let\c@perthm=\c@thm
\let\c@conjtel=\c@thm
\let\c@exmps=\c@thm
\let\c@rem=\c@thm
\let\c@question=\c@thm
\let\c@warn=\c@thm
\let\c@claim=\c@thm
\let\c@quest=\c@thm
\let\c@notation=\c@thm
\let\c@note=\c@thm
\let\c@conjtel=\c@thm
\let\c@gue=\c@thm
\let\c@goal=\c@thm
\makeatother

\DeclareMathOperator{\Ext}{Ext}
\DeclareMathOperator{\Hom}{Hom}
\DeclareMathOperator{\aut}{Aut}
\DeclareMathOperator{\im}{im}

\DeclareMathOperator{\Map}{Map}

\DeclareMathOperator{\Aut}{Aut}

\DeclareMathOperator{\Gal}{Gal}

\DeclareMathOperator{\sq}{Sq}

\newcommand{\s}{\wedge}

\newcommand{\E}[2]{\prescript{#1}{#2}{E}}

\usepackage[linecolor=black]{todonotes}
\usepackage{blindtext}
\usepackage{tcolorbox}

\newcommand{\smsh}{\wedge}

\definecolor{limegreen}{rgb}{0.2, 0.8, 0.2}
\definecolor{darkmagenta}{rgb}{0.55, 0.0, 0.55}
\definecolor{lavenderrose}{rgb}{0.91, 0.33, 0.5}
\definecolor{goldenpoppy}{rgb}{0.99, 0.76, 0.0}
\definecolor{seagreen}{rgb}{0.11, 0.35, 0.02}
\definecolor{maroon}{RGB}{128,0,0}
\definecolor{darkviolet}{RGB}{148,0,211}
\definecolor{twelve}{RGB}{100,100,170}
\definecolor{thirteen}{RGB}{100,150,50}
\definecolor{fourteen}{RGB}{200,0,0}
\definecolor{fifteen}{RGB}{0,200,0}
\definecolor{sixteen}{RGB}{0,0,200}
\definecolor{seventeen}{RGB}{200,0,200}
\definecolor{eighteen}{RGB}{0,200,200}

\def\ev#1#2#3{
{\color{#1}(#2,#3:#3)^{ev}}
}

\title{The telescope conjecture at height 2 and the tmf resolution}

\author{A.~Beaudry}\address{University of Colorado at Boulder}\email{agnes.beaudry@colorado.edu}
\author{M.~Behrens}\address{University of Notre Dame}\email{mbehren1@nd.edu}
\author{P.~Bhattacharya}\address{University of Notre Dame}\email{pbhattac@nd.edu}
\author{D.~Culver}\address{Max Plank Institute for Mathematics }\email{dculver@mpim-bonn.mpg.de}
\author{Z.~Xu}\address{University of California, San Diego}\email{xuzhouli@ucsd.edu}

\thanks{This material is based upon work supported by the National Science Foundation under Grants No. DMS-1050466/1611786/2005476, DMS-1725563/1906227, and DMS-1810638.}

\begin{document}

\begin{abstract}
Mahowald proved the height 1 telescope conjecture at the prime 2 as an application of his seminal work on bo-resolutions.  In this paper we study the height 2 telescope conjecture at the prime 2 through the lens of $\tmf$-resolutions. To this end we compute the structure of the $\tmf$-resolution for a specific type $2$ complex $Z$.  We find that, analogous to the height 1 case, the $E_1$-page of the $\tmf$-resolution possesses a decomposition into a $v_2$-periodic summand, and an Eilenberg-MacLane summand which consists of bounded $v_2$-torsion.  However, unlike the height 1 case, the $E_2$-page of the $\tmf$-resolution exhibits unbounded $v_2$-torsion.  We compare this to the work of Mahowald-Ravenel-Shick, and discuss how the validity of the telescope conjecture is connected to the fate of this unbounded $v_2$-torsion: either the unbounded $v_2$-torsion kills itself off in the spectral sequence, and the telescope conjecture is true, or it persists to form \emph{$v_2$-parabolas} and the telescope conjecture is false.  We also study how to use the $\tmf$-resolution to effectively give low dimensional computations of the homotopy groups of $Z$.  These computations allow us to prove a conjecture of the second author and Egger: the $E(2)$-local Adams-Novikov spectral sequence for $Z$ collapses.
\end{abstract}

\maketitle	

\tableofcontents


\section{Introduction}\label{sec:intro}

\subsection*{The telescope conjecture}

Fix a prime $p$ and let $X$ be a finite spectrum.  The perspective of chromatic homotopy theory is to understand $X_{(p)}$ through the study of its \emph{chromatic tower} \cite[Sec.~7.5]{Ravenelorange} \cite{BarthelBeaudry}
$$  \cdots \rightarrow X_{E(n)} \rightarrow X_{E(n-1)} \rightarrow \cdots \rightarrow X_{E(0)} = X_\QQ. $$
Here, $X_{E(n)}$ denotes the Bousfield localization of $X$ with respect to the Johnson-Wilson spectrum $E(n)$ with
$$ \pi_*E(n) = \ZZ_{(p)}[v_1, \ldots, v_n, v_n^{-1}] $$
where $\abs{v_n} = 2(p^n-1)$.  
The chromatic convergence theorem of Hopkins and Ravenel \cite{HopkinsRavenel} states that $X_{(p)}$ is recovered as the inverse limit of the tower.  Thus the $E(n)$-localizations interpolate between the rationalization and the $p$-localization of $X$.  The \emph{monochromatic layers} of the chromatic tower are defined to be the fibers
$$ M_nX \rightarrow X_{E(n)} \rightarrow X_{E(n-1)}. $$
Applying $\pi_*$ to the chromatic tower yields the \emph{chromatic spectral sequence}
$$ \E{css}{}_1^{n,*}(X) = \pi_* M_nX \Rightarrow \pi_*X. $$ 
The efficacy of the chromatic approach is established by Morava's change of rings theorem \cite{Morava}, which states that the Adams-Novikov spectral sequence for $M_nX$ takes the form 
$$ \E{anss}{}_2^{s,t}(M_nX) = H^s_c(\GG_n, (E_n)_t M_nX) \Rightarrow \pi_{t-s}M_nX, $$
where $E_n$ is the height $n$ Morava E-theory spectrum and $\GG_n$ is the height $n$ Morava stabilizer group.  For a given height $n$, $\E{anss}{}^{*,*}_2(M_nX)$ (and in fact the entire Adams-Novikov spectral sequence) is in principle completely computable. 

In reality, the complexity of these computations increases significantly as a function of $n$, and therefore these computations have only been carried out successfully for small values of $n$.  It is thus desirable to have a means of directly relating the homotopy groups of each of the monochromatic layers $M_nX$ to the homotopy groups of $X$ itself, without having to resort to needing to compute the entire chromatic spectral sequence.

There is a variant of the chromatic tower which does have this property.  Let $X^f_{E(n)}$ denote the finite $E(n)$-localization, obtained by killing only finite $E(n)$-acyclic spectra (instead of all $E(n)$ acyclic spectra).  The finite localizations also form a tower, with finite monochromatic layers defined to be the fibers
$$ M_n^fX \rightarrow X^f_{E(n)} \rightarrow X^f_{E(n-1)}. $$
The advantage of this variant of the chromatic tower is that the elements of the homotopy groups of these finite monochromatic layers have a concrete relationship to the homotopy groups of $X$ itself: elements of $\pi_*M_n X$ correspond to $v_n$-periodic families in $\pi_*X$ \cite[Sec.~2.5]{Ravenelorange}.   

In \cite{1984}, Ravenel proposed the following Panglossian conjecture.

\begin{namedtheorem}[Telescope Conjecture]
For any spectrum $X$, prime $p$, and height $n$, the natural map
\[X^f_{E(n)} \to X_{E(n)}\] 
is an equivalence.
\end{namedtheorem}

The Hopkins-Smith thick subcategory theorem \cite{HopkinsSmith} implies that the telescope conjecture is true if and only if it is true for a single type $n$ spectrum (a $p$-local finite spectrum which is $E(n-1)$-acyclic, but not $E(n)$-acyclic).  In this case where $X$ is type $n \ge 1$, the Hopkins-Smith periodicity theorem \cite{HopkinsSmith} implies there is an asymptotically unique \emph{$v_n$-self map}, that is, an $E(n)$-self equivalence
$$ v: \Sigma^N X \rightarrow X $$
where $N > 0$.
The telescope of $X$ is defined
as the homotopy colimit 
\[ \widehat{X} := X \xrightarrow{v} \Sigma^{-N} X  \xrightarrow{v} \Sigma^{-2N} X  \rightarrow \ldots \]
and we have
$$ X^f_{E(n)} \simeq \widehat{X}. $$
Thus, for any $p$-local spectrum $X$ of type $n$, the natural map
\begin{equation}\label{eq:telescopemap}
\widehat{X} \rightarrow X_{E(n)}
\end{equation}
is an equivalence if and only if the $p$-primary height $n$ telescope conjecture is true.

\subsection*{The height 1 case}

The telescope conjecture was in large part motivated by the case of height $n = 1$, where the conjecture was already proven by Mahowald for $p = 2$ \cite{Mahowaldbo}, \cite[Thm.~1.2]{Mahowald}, and Miller for $p > 2$ \cite{Miller}.  In both of these cases, the proof is computational, in the sense that the authors compute the homotopy groups of the source and target of (\ref{eq:telescopemap}) and show that the map is an isomorphism on these homotopy groups.
The methods used in each of these cases, though, are somewhat different.

In the $p > 2$ case of \cite{Miller}, Miller considered the mod $p$ Moore spectrum $M(p)$, which is type $1$, with $v_1$-self map 
\begin{equation}\label{eq:v1^1Mp}
 v: \Sigma^{2(p-1)}M(p) \rightarrow M(p)
\end{equation}
having the property that it is given by multiplication by $v_1$ in $E(1)$-homology.  We will call such a self-map a $v_1^1$-self map.  Miller computes the localized Adams spectral sequence
$$ v_1^{-1}\E{ass}{}_2^{s,t} = v_1^{-1}\Ext^{s,t}_{A_*}(\FF_p, H_* M(p)) \Rightarrow \pi_{t-s}\widehat{M(p)} $$
where $A_*$ denotes the $p$-primary dual Steenrod algebra, and $H_*$ denotes mod $p$ homology.  To do this, he completely computes the $E_2$-page, and then gives a delicate lifting argument which computes the $d_2$ Adams differentials from the $d_1$ Adams-Novikov differentials.  He then shows the localized Adams spectral sequence collapses at $E_3$ to the known values of $\pi_*M(p)_{E(1)}$.

In the case of $p = 2$, the situation is more complicated as the mod $2$ Moore spectrum only has a \emph{$v_1^4$-self map}
$$ v_1^4: \Sigma^{8} M(2) \rightarrow M(2), $$
having the property that it is given by multiplication by $v_1^4$ on $E(1)$-homology.  
For this reason, in \cite{Mahowald}, Mahowald considers the 2-local type 1-spectrum
$$ Y := M(2) \wedge C(\eta) $$ 
where $C(\eta)$ denotes the cofiber of the Hopf map $\eta : S^1 \rightarrow S^0$.  In contrast with the case of the mod 2 Moore spectrum, the spectrum $Y$ possesses a $v_1^1$-self map:
$$ v_1^1: \Sigma^2 Y \rightarrow Y. $$
Mahowald analyzed the $\bo$-based Adams spectral sequence (aka the ``$\bo$-resolution'') for $Y$:
$$ \E{\bo}{}_1^{s,t}(Y) = \pi_t \bo^{\wedge s+1} \wedge Y \Rightarrow \pi_{t-s} Y. $$
Here $\bo$ denotes the connective real K-theory spectrum.
This spectral sequence is significantly simplified by the fact that we have an equivalence
$$ \bo \wedge Y \simeq k(1), $$
where $k(1)$ denotes the height $1$ connective Morava K-theory spectrum.  Unfortunately, the $v_1$-localized $\bo$-resolution converges to the $E(1)$-local homotopy groups of $Y$ (rather than those of $\widehat{Y}$):
$$ v_1^{-1}\E{bo}{}^{s,t}_1(Y) \Rightarrow \pi_{t-s}Y_{E(1)}. $$
Nevertheless, Mahowald was able to deduce the height $1$ telescope conjecture at the prime $2$ by establishing the following key results: 
\begin{description}
\item[Collapse theorem] The $v_1$-localized $\bo$-resolution for $Y$ collapses at its $E_2$-page.
\item[Bounded torsion theorem] If $x \in \E{\bo}{}^{*,*}_2$ is $v_1$-torsion, then $v_1^2 x = 0$.
\item[Vanishing line theorem] There is a $c$ so that $\E{\bo}{}^{s,t}_2(Y) = 0$ for $s > \frac{t-s}{5} + c$. 
\end{description}  
The idea is to use these key results to prove the map
\begin{equation}\label{eq:Ytelmap}
 \pi_*\widehat{Y} \rightarrow \pi_*Y_{E(1)}
\end{equation}
is surjective and injective.  
The map (\ref{eq:Ytelmap}) is surjective because if $y \in \pi_*Y_{E(1)}$ is detected by an element $y' \in v_1^{-1}\E{\bo}{}_2^{*,*}(Y)$ in the $v_1$-localized $\bo$-resolution, then the bounded torsion theorem implies that the targets of the differentials supported by the family $v_1^{2i}y'$ in the unlocalized $\bo$-resolution lie above a line of slope $1/4$, and thus will eventually surpass the $1/5$ vanishing line.  Hence for $i \gg 0$, the element $v_1^{2i}y'$ detects a $v_1$-periodic family mapping to that of $y$ under (\ref{eq:Ytelmap}).  The map (\ref{eq:Ytelmap}) is injective because the collapse theorem implies that any element $x \in \pi_*Y$ which maps to zero in $\pi_*Y_{E(1)}$ must be detected by a $v_1$-torsion element of $\E{\bo}{}_2^{*,*}(Y)$, and the bounded torsion theorem then implies that the family of elements  $v_1^{2i}x$ are detected in the $\bo$ resolution above a line of slope $1/4$, and thus will eventually surpass the $1/5$ vanishing line.  Hence $x$ must be $v_1$-torsion.

\subsection*{Attempts to disprove the telescope conjecture}

Less than a decade after his 1984 paper, Ravenel's optimistic beliefs concerning the telescope conjecture took a decidedly Orwellian turn.  In \cite{Raveneltelescope}, Ravenel studied the height $2$ telescope conjecture at primes $p \ge 5$ by considering the analog of Miller's argument for the Smith-Toda complex $V(1)$ (the cofiber of (\ref{eq:v1^1Mp})).  The Adams-Novikov spectral sequence
$$ \E{anss}{}^{*,*}_2(V(1)_{E(2)}) = H^*_c(\GG_2, (E_2)_*V(1)) \Rightarrow \pi_* V(1)_{E(2)} $$
collapses for dimensional reasons.
Ravenel computed the $E_2$-term of the localized Adams spectral sequence
$$ v_2^{-1}\E{ass}{}^{*,*}_2(V(1)) = v_2^{-1}\Ext_{A_*}(\FF_p, H_*V(1)) \Rightarrow \pi_*\widehat{V(1)}, $$
and found that the Adams-Novikov differentials lifted to differentials of unbounded length in the localized Adams spectral sequence.  He then observed that a power operation argument gave rise to Toda-type differentials which 
preceeded the lifted Adams-Novikov differentials, potentially causing $\pi_*\widehat{V(1)}$ to differ from $\pi_*V(1)_{E(2)}$.  Although he initially thought he had a counterexample to the telescope conjecture, it eventually became clear that it was impossible to rule out the possibility that a bizarre pattern of other differentials might subsequently ``fix'' the havoc caused by these Toda-type differentials, allowing the telescope conjecture to hold.  Mahowald, Ravenel, and Shick summarized the uncertain state of affairs in \cite{MahowaldRavenelShick}.


\subsection*{Main results}

\emph{The purpose of this paper is to carry out the height $2$ analog of Mahowald's analysis of the height $1$ telescope conjecture at the prime $2$.}  

To this end we replace the $\bo$-resolution of Mahowald with the $\tmf$-based Adams spectral sequence (aka the $\tmf$-resolution), 
$$ \E{\tmf}{}_1^{s,t}(X) = \pi_t (\tmf^{\wedge {s+1}} \wedge X) \Rightarrow \pi_{t-s}X $$
where $\tmf$ denotes the spectrum of \emph{connective topological modular forms} \cite{tmf}. The role that was played by Mahowald's spectrum $Y$ will now be reprised by $Z$, a $2$-local finite spectrum of type $2$ constructed by the third author and Egger \cite{bhateggerZ}, with the distinguished property that it posesses a $v_2^1$-self map
$$ v_2^1: \Sigma^6 Z \rightarrow Z, $$
and that there is an equivalence 
\begin{equation}\label{eq:tmfsmashZ}
\tmf \wedge Z \simeq k(2). 
\end{equation}
Here $k(2)$ is the height 2 connective Morava $K$-theory spectrum. 

We find that, similar to the height $1$ case, the $E_1$ term of the $\tmf$-resolution for $Z$ fits into a short exact sequence
\begin{equation}\label{eq:sesintro}
 0 \to V^{*,*}(Z) \to \E{\tmf}{}^{*,*}_1(Z) \to \mc{C}^{*,*}(Z) \to 0 
 \end{equation} 
 where the groups $\mc{C}^{*,*}(Z)$ are $v_2$-torsion free and completely computable and the groups $V^{*,*}(Z)$ are $v_2^1$-torsion and essentially incomputable.  We will refer to $\mc{C}^{*,*}(Z)$ as the \emph{good complex} and $V^{*,*}(Z)$ as the \emph{evil complex}.

We will show that the good complex $\mc{C}^{*,*}(Z)$ is an explicit connective subcomplex of the cobar complex for computing the $E_2$-term of the Adams-Novikov spectral sequence 
\begin{equation}\label{eq:anssZE(2)}
 \E{anss}{}_2^{*,*}(Z_{E(2)}) = H^*_c(\GG_2; (E_2)_*Z) \Rightarrow \pi_* Z_{E(2)}
 \end{equation}
and the localized $\tmf$ resolution
$$ v_2^{-1}\E{\tmf}{}_2^{*,*} = v_2^{-1}H^{*,*}(\mc{C}(Z)) \Rightarrow \pi_* Z_{E(2)}. $$
is isomorphic to the spectral sequence (\ref{eq:anssZE(2)}).\footnote{In general, for a bigraded cochain complex $C^{*,*}$, we shall denote its cohomology by $H^{*,*}(C)$.}  The groups 
$$ \E{anss}{}_2^{*,*}(Z_{E(2)}) $$
were computed by the third author and Egger \cite{bhateggerK2Z}.  It turned out that, unlike the situation for large primes, the spectral sequence (\ref{eq:anssZE(2)}) cannot be shown to collapse at its $E_2$-page simply for dimensional reasons, but the third author and Egger conjectured that it does collapse.  
One major result of this paper is a proof of this conjecture.

\begin{namedtheorem}[Collapse Theorem (Theorem~\ref{thm:k2locss})]
The $v_2$-localized $\tmf$-resolution for $Z$ collapses at its $E_2$-page.
\end{namedtheorem}

The height $2$ story begins to diverge in the context of the bounded torsion theorem.  We construct an analog of the May filtration on the good complex $\mc{C}^{*,*}(Z)$, and we will refer to the associated spectral sequence 
$$ \E{MR}{}_1^{*,*,*} = H^{*,*}(E^0_*\mc{C}(Z)) \Rightarrow H^{*,*}(\mc{C}(Z)) $$
as the \emph{May-Ravenel spectral sequence}.  We will completely compute the $E_1$-term of the May-Ravenel spectral sequence, and will observe the following:

\begin{namedtheorem}[Unbounded torsion theorem (Theorem \ref{thm:MRE1})]
The May-Ravenel $E_1$-page has 
unbounded $v_2$-torsion: there are elements which are $v_2^i$-torsion for $i$ arbitrarily large.
\end{namedtheorem}

Unfortunately, we are unable to deduce the same unbounded torsion statement for $H^{*,*}(\mc{C}(Z))$ (which is equivalent to unbounded torsion in  $\E{\tmf}{}^{*,*}_2(Z)$) because we do not know if the May-Ravenel spectral sequence collapses at $E_1$, and we do not know if there are hidden $v_2$ extensions in this spectral sequence.  Nevertheless, the computation of the unbounded torsion allows us to understand exactly how the telescope conjecture could fail at height $2$.

We now turn our attention to the final component of Mahowald's work on $\bo$-resolutions: the vanishing line.  Clearly, a vanishing line for the cohomology of the good complex $H^{*,*}(\mc{C}(Z))$ can be read off of our computation of the May-Ravenel $E_1$-term.  In order to lift this vanishing line to one for $\E{\tmf}{}^{*,*}_2(Z)$, we need a vanishing line for the cohomology of the evil complex $H^{*,*}(V(Z))$.  

In \cite{BBBCX} we developed a technique (the agathakakological\footnote{{\bf agathokakological} (ag-uh-thuh-kak-uh-LAHJ-uh-kuhl) 
\emph{adjective:} Made up of both good and evil.}
spectral sequence) for computing the cohomology of the evil complex associated to the $\bo$-resolution by relating it to $\Ext_{A_*}$ and the good complex.  We will construct an agathokakological spectral sequence in our present setting of the $\tmf$-resolution.  This will allow us to use a vanishing line in $\Ext_{A_*}$ to establish a vanishing line for the cohomology of the evil complex $H^{*,*}(V(Z))$, thus establishing a vanishing line for $\E{\tmf}{}^{*,*}_2(Z)$.

\begin{namedtheorem}[Vanishing line theorem (Theorem~\ref{thm:vanishingline})]
In the $\tmf$-resolution for $Z$, we have $\E{\tmf}{}_2^{s,t}(Z) = 0$ for
$$ s > \frac{t-s+12}{11}. $$
\end{namedtheorem}

The slope of this line cannot be improved at $E_2$; $1/11$ is the slope of the non-nilpotent element
$$ g_2 \in \Ext^{4, 48}_{A_*}(\FF_2, \FF_2), $$
and it turns out $g_2$ lifts to $\E{\tmf}{}_2^{4,48}(Z)$.  We conjecture that for some $r > 2$, the $E_r$-page has a slope $1/13$ vanishing line (Conj.~\ref{conj:vanishingline}).

The agathokakological spectral sequence allows us to combine our computations of the cohomology of the good complex $H^{*,*}(\mc{C}(Z))$ with low dimensional computer computations of $\Ext_{A_*}$ to obtain low dimensional computations of the $\tmf$-resolution $E_2$-page $\E{\tmf}{}^{*,*}_2(Z)$.   Using this technique, we compute the $\tmf$-resolution of $Z$ through the $40$-stem.  This is not just an academic exercise --- rather it is the means by which we prove the collapse theorem.  In this range we are able to locate unlocalized elements which map to the generators of the $E_2$-term of the Adams-Novikov spectral sequence for $Z_{E(2)}$. By observing that the corresponding unlocalized elements are permanent cycles in the $\tmf$-resolution, we deduce that their images in the Adams-Novikov spectral sequence for $Z_{E(2)}$ are permanent cycles.

The unbounded torsion theorem allows us to identify the possible ways the map
$$ \pi_* \widehat{Z} \rightarrow \pi_* Z_{E(2)} $$
can fail to be an isomorphism.  The last section of this paper is a detailed discussion giving a precise conjecture for what $\pi_*\widehat{Z}$ is (the \emph{parabola conjecture}), and how this conjectural answer differs from $\pi_*Z_{E(2)}$.  The parabola conjecture is essentially an adaptation of the conjectures of Mahowald, Ravenel, and Shick \cite{Ravenelparabola}, \cite{MahowaldRavenelShick} to our context.

\subsection*{Future directions}

It is probably clear to the reader that the authors hoped that adapting Mahowald's approach to the $2$-primary height $1$ telescope conjecture to the height $2$ context would yield new information that would lead to a computational proof or disproof of the telescope conjecture at chromatic height $2$.  The results of this paper are as such inconclusive, and the telescope conjecture remains one of the great unlocked mysteries of the subject.

The seasoned expert will recognize, however, that if this was the authors' only goal, then we would have been better off studying the $BP\bra{2}$-resolution of the Smith-Toda complex $V(1)$ for primes $p \ge 5$.  Indeed, that would have simplified many parts of this paper.

However, the authors had other motivations for undertaking this particular endeavor at the prime $2$.  We wanted to complete the computation of $\pi_*Z_{E(2)}$ initiated by the third author and Egger in \cite{bhateggerK2Z}.  Not only does our analysis show that the structure of the homotopy groups of $Z_{E(2)}$ mirrors the structure of the homotopy groups of $V(1)_{E(2)}$ at primes $p \ge 5$, despite the fact that the $E(2)$-local Adams-Novikov spectral sequence is no longer sparse,  it also represents the first non-trivial complete computation of the homotopy groups of any $E(2)$-local finite complex at the prime $2$.

The prime $2$ represents the last computational frontier for chromatic height $2$, where computations are elaborate but straightforward for primes $p \ge 5$ (see, for example, \cite{SY}), and downright difficult, but possible, the prime $p = 3$ (see, for example, \cite{GHMR}).  In fact, the duality resolution of \cite{GHMR} is a minimal $\tmf$-resolution of the sphere in the $K(2)$-local stable homotopy category.


Besides the fact that the $2$-torsion in $\tmf$ is an order of magnitude more complicated than the $3$-torsion, there is a fundamental unsolved difficulty at $p = 2$: Bobkova and Goerss have successfully constructed a duality resolution for the maximal cyclotomic extension of the $K(2)$-local sphere \cite{BobkovaGoerss}, but constructing resolutions of the $K(2)$-local sphere itself is much more subtle. Our analysis links the $\tmf$-resolution explicitly to the Morava stabilizer group through the good complex.  We are hopeful that this will allow us to one day use the $\tmf$-resolution to help us understand finite resolutions for the $K(2)$-local sphere itself.

We also plan to develop the $\tmf$-resolution as a valuable tool for low dimensional $2$-primary computations of stable homotopy groups.  Our use of the $\tmf$-resolution to compute the first $40$ stems of $Z$ required very little effort --- the computation could probably be pushed to much higher degrees if we had a good reason to do so.  In the case of the sphere, there is such a motivation: the Kervaire invariant one problem in dimension $126$ \cite{HHR}.  The work of Isaksen, Wang, and the fifth author \cite{IsaksenWangXu} shows that complex motivic homotopy theory can be used to effectively compute the $2$-primary Adams spectral sequence for the sphere, and they have used their machinery to carry out this computation up to the 90 stem.  It is unclear whether their techniques alone will suffice to get up to dimension 126.  The $\tmf$-resolution could provide a valuable tool for analyzing Adams differentials between $v_2$-periodic elements in $\Ext_{A_*}$.  The computations of this paper provide the starting point for the analysis of the $\tmf$-resolution of the sphere.

\subsection*{Conventions}

We will use the following notation throughout this paper.
\begin{itemize}
\item ASS = classical Adams spectral sequence.
\item $\tmf$-ASS = the $\tmf$-based ASS (aka the $\tmf$-resolution).
\item ANSS = Adams-Novikov spectral sequence (aka the $BP$-based ASS).
\item AKSS = agathokakological spectral sequence.
\item $H_*(-)/H^*(-)$ denotes homology/cohomology with $\FF_2$-coefficients.
\item $H$ denotes the mod 2 Eilenberg-MacLane spectrum.
\item $A$ denotes the mod $2$ Steenrod algebra, and $A_*$ is its dual.
\end{itemize}

For $X$ any $2$-complete spectrum, we shall let
$$ \E{ass}{}_2^{s,t}(X) = \Ext^{s,t}_{A_*}(\FF_2, H_*X) \Rightarrow \pi_{t-s}X $$
denote its ASS.  Assuming this spectral sequence converges, we shall say an element of $\pi_*X$ has \emph{Adams filtration} $s$ if it is detected in the ASS by a class in $\E{ass}{}_2^{s,*}(X)$.

Finally, in \cite{bhateggerZ}, the third author and Egger show that there is a class of spectra $\td{\mc{Z}}$, each of whose cohomology is isomorphic as $A(2)$-modules, and each of which admits a $v_2^1$-self map.  For concreteness, the spectrum we call $Z$ in this paper is always taken to be a particular fixed member of this class for which the cofiber of its $v_2$-self map has cohomology as described in Appendix~\ref{apx:data}. 

\subsection*{Organization of the paper}

In Section~\ref{sec:background}, we recall some basic facts about the spectrum $\tmf$, its cohomology, and its relationship to Morava $E$-theory.  We will also review some facts about the spectrum $Z$.

%

In Section~\ref{sec:goodevildecomp} we begin our analysis of the $\tmf$-ASS $\{\E{\tmf}{}^{n,t}_r(Z)\}$.  The $E_1$-term is given by
$$ \E{\tmf}{}^{n,t}_1(Z) = \pi_t (\tmf^{\wedge n+1} \wedge Z). $$
We will compute this $E_1$-term using the Adams spectral sequences
$$ \E{ass}{}^{s,t}_2(\tmf^{\wedge n+1} \wedge Z) \Rightarrow \pi_{t-s}(\tmf^{\wedge n+1}\wedge Z). $$
We will explain how to use Margolis homology to compute the $E_2$-terms of these Adams spectral sequences, and we show these Adams spectral sequences collapse to give a short exact sequence of chain complexes (\ref{eq:sesintro}) (the \emph{good/evil} decomposition).
The goal is to use the short exact sequence (\ref{eq:sesintro}) to compute $\E{\tmf}{}_2^{*,*}$ from $H^{*,*}(\mc{C}(Z))$ and $H^{*,*}(V(Z))$. It will turn out that $H^{*,*}(V(Z))$ is computable despite the incomputability of $V^{*,*}(Z)$ itself. 

In Section~\ref{sec:MSG}, we both recall the structure of the Morava stabilizer group and Morava stabilizer algebra associated to the Honda height $2$ formal group, and relate these to the corresponding groups and algebras for the formal group coming from the unique supersingular elliptic curve $C$ in characteristic $2$.  We compute the action of the group of automorphisms $\Aut(C)$ on the Morava $E$-homology of the complex $Z$.

In Section~\ref{sec:gooddiffs}, we compute the differentials in the good complex $\mc{C}^{*,*}(Z)$.  This is accomplished by showing that the good complex is actually isomorphic to the cobar complex of an explicit sub-Hopf algebra $\sitd(2)$ of a quotient of the Morava stabilizer algebra.

At this point, the number of different Hopf algebras important for our purposes has become significant, so we give a list in Table~\ref{fig:listofhopfs} to help the reader keep track.

In Section~\ref{sec:good} we embark on the computation of 
$$ H^{*,*}(\mc{C}(Z)) \cong \Ext^{*,*}_{\sitd(2)}(k(2)_*, k(2)_*). $$ 
The cohomology of the Morava stabilizer algebra was computed by Ravenel \cite{Ravenel_coh} using a modification of the May spectral sequence which we will call the \emph{May-Ravenel spectral sequence}.  We adapt the May-Ravenel spectral sequence to compute the cohomology of $\sitd(2)$.
We completely compute the $E_1$-term of this spectral sequence (Theorem~\ref{thm:MRE1}), thus proving the unbounded torsion theorem.

Having dealt with the good complex, in Section~\ref{sec:agatho} we turn to the problem of computing the cohomology of the evil complex.  Following the techniques introduced in \cite{BBBCX}, 
we introduce a refinement of the $\tmf$-ASS called the \emph{topological agathokakological spectral sequence} (topological AKSS)
$$ H^{*,*}(\mc{C}(Z)) \oplus H^{*,*}(V(Z)) \Rightarrow \pi_*Z. $$
We also introduce an algebraic version, the \emph{algebraic agathokakological spectral sequence} (algebraic AKSS) 
$$ H^{*,*,*}(\mc{C}_{alg}(Z)) \oplus H^{*,*}(V(Z)) \Rightarrow \E{ass}{}^{*,*}_2(Z). $$
We then prove the \emph{dichotomy principle} (Theorem~\ref{thm:dichotomy}), which relates evil terms in the algebraic AKSS to $v_2$-torsion in $\E{ass}{}_2^{*,*}(Z)$.
We therefore are able to recover $H^{*,*}(V(Z))$ from $H^{*,*,*}(\mc{C}_{alg}(Z))$ (which we completely compute) and $\E{ass}{}^{*,*}_2(Z)$ (which we compute using Bruner's Ext software \cite{Bruner}).  

In Section~\ref{sec:stemwise}, we perform low dimensional computations of the $\tmf$-ASS (or equivalently, the topological AKSS) for $Z$ in the range $t-n < 40$.  This proceeds by first analyzing $v_2$-periodicity in $\E{ass}{}^{*,*}_2(Z)$ by analyzing $\E{ass}{}^{*,*}_2(A_2)$, where $A_2$ is the cofiber
$$ \Sigma^6 Z \xrightarrow{v_2} Z \to A_2 $$
whose cohomology is isomorphic to the subalgebra $A(2) \subset A$ as an $A(2)$-module.  Appendix~\ref{apx:data} contains the Bruner module definition data used to compute the relevant Ext charts.
We then compute the algebraic AKSS in our range.  From this we extract $H^{*,*}(V(Z))$, which we input into the topological AKSS, and compute through this same range.  We end this section with a comparison to the computations of Bhattacharya-Egger of the Adams-Novikov spectral sequence (ANSS) for $Z_{E(2)}$, and prove the collapse theorem by mapping the $\tmf$-ASS to the $K(2)$-local ANSS (Theorem~\ref{thm:k2locss}).

In Section~\ref{sec:telescope} we discuss how the analog of Mahowald's approach to the $2$-primary height $1$ telescope conjecture using the $\bo$-resolution for $Y$ fails in the context of the $\tmf$-resolution for $Z$.  Namely, assuming there are no additional differentials or extensions in the May-Ravenel spectral sequence, and assuming a certain pattern of $d_3$-differentials, we show that $\E{\tmf}{}_4$ decomposes into a direct sum of three pieces: 
\begin{enumerate}
\item a summand which is $v_2$-torsion free, and is isomorphic to  $\pi_*Z_{E(2)}$ after $v_2$ inversion, 
\item a summand which consists entirely of bounded $v_2^2$-torsion, and
\item a summand which consists of unbounded $v_2$-torsion, and assembles via a conjectural sequence of hidden extensions, into an uncountable collection of \emph{$v_2$-parabolas}.\footnote{We call them $v_2$-parabolas because they lie on (sideways) parabolas in the $(t-n,n)$-plane.}
\end{enumerate}
We explain how our work in previous sections proves the vanishing line theorem (the slope $1/11$ vanishing line for $\E{\tmf}{}_2^{*,*}(Z)$). We explain why one might expect to be able to improve this to a slope $1/13$ vanishing line, which would preclude infinite families of hidden extensions among the terms in summand (2) from assembling to give $v_2$-families in $\pi_*\widehat{Z}$. 
We then describe the analogs of conjectures of Mahowald-Ravenel-Shick \cite{MahowaldRavenelShick} which describe a hypothetical picture (\emph{the parabola conjecture}) of $\pi_* \widehat{Z}$ which is assembled from a portion of the classes in summands (1) and (3) above, and in particular is unequal to $\pi_*Z_{E(2)}$.  However, just as in \cite{MahowaldRavenelShick}, it is totally possible for a bizarre pattern of differentials between $v_2$-parabolas to occur to make the telescope conjecture true.  

\subsection*{Acknowledgments}

The authors benefited greatly from conversations with Phil Egger, Paul Goerss, Mike Hopkins, Doug Ravenel, and Tomer Schlank. The authors also owe much gratitude to the referee, for their incredibly thoughtful insights into how to improve what is a very technical paper, and for providing the statement and proof of Proposition~\ref{prop:k(n)}.


\section{Background}\label{sec:background}

\subsection{Morava $K$-theory and $E$-theory}

Recall \cite[Part II]{Adams} that a homotopy commutative ring spectrum is said to be \emph{complex orientable} if the map on reduced $E$-cohomology
$$ \td{E}^*(\CC P^\infty) \rightarrow \td{E}^{*}(\CC P^1) $$
is surjective.  The cohomology $\td{E}^*(\CC P^1)$ is free of rank $1$ as an $E_*$-module, and a lift 
$$ x \in \td{E}^*(\CC P^\infty) $$ 
of a generator of $\td{E}^*(\CC P^1)$ is called a complex orientation.  We then have
$$ E^*(\CC P^\infty) = E^*[[x]]. $$
The $H$-space structure on $\CC P^\infty$ gives rise to a formal group law over $E^*$.
In the case where the spectrum $E$ is even periodic, ($\pi_{odd} E = 0$ and $\pi_2 E$ contains a unit) we can take our complex orientation to lie in $\td{E}^0(\CC P^\infty)$, and the resulting formal group law $\mc{F}_E$ can actually be defined over the ring $E_0$.

A complex orientation of a ring spectrum $E$ is equivalent to the structure of a map of ring spectra
$$ MU \rightarrow E, $$
where $MU$ is the complex cobordism spectrum.
For a prime $p$, the $p$-localization of $MU$ splits as a wedge of suspensions of the Brown-Peterson specrum $BP$, with
$$ BP_* \cong \ZZ_{(p)}[v_1, v_2, v_3, \cdots] $$
with $\abs{v_i} = 2(p^i - 1)$.  The Wilson spectrum $BP\bra{n}$ can be constructed as the regular quotient of $BP$ given by \cite{Strickland}
$$ BP\bra{n} = BP/(v_{n+1}, v_{n+2}, \ldots ). $$
However, these ring spectra depend on the choices of the generators $v_i$, and as such there are many different \emph{forms} of $BP\bra{n}$. 
Associated to any such choice is the associated Johnson-Wilson spectrum
$$ E(n) := BP\bra{n}[v_n^{-1}] $$
and the associated Morava $K$-theory spectrum is the regular quotient
$$ K(n) = E(n)/(p, v_1, \ldots v_{n-1}) $$
with
$$ \pi_*K(n) = \FF_p[v_n^{\pm 1}]. $$
The connective Morava $K$-theory $k(n)$ is the connective cover of $K(n)$.

The localization functors $(-)_{E(n)}$ and $(-)_{K(n)}$ are independent of the form of $E(n)$ and $K(n)$, and we have \cite{1984}
$$ (-)_{E(n)} = (-)_{K(0) \vee \cdots \vee K(n)}. $$
In particular, if $X$ is a type $n$ spectrum, then we have
$$ X_{E(n)} \simeq X_{K(n)}. $$

The height $n$ Morava $E$-theory spectrum $E_n$ \cite{BarthelBeaudry} is a $K(n)$-local even periodic variant of the Johnson-Wilson spectrum $E(n)$.  Like $E(n)$, there are many forms of $E_n$, one for each height $n$ formal group law $\mc{F}$ over a perfect field $\FF$ of characteristic $p$.  The formal group law associated to $E_n$ is the Lubin-Tate universal deformation of $\mc{F}$, and we have
$$ \pi_*E_n = \WW(\FF)[[u_1, \ldots, u_{n-1}]][u^{\pm 1}] $$
where $\WW(\FF)$ denotes the Witt ring of $\FF$, $\abs{u_i} = 0$, and $\abs{u} = -2$.
Goerss, Hopkins, and Miller showed that $E_n$ admits a homotopically unique $E_\infty$-structure, and admits a natural action of the \emph{Morava stabilizer group} $\Aut(\mc{F})$.
If $\mc{F}$ is obtained from a formal group law over $\FF_p$ via base change, then there is a natural action of $\Gal(\FF/\FF_p)$ on $\Aut(\mc{F})$, and the natural action of $\Aut(\mc{F})$ on $E_n$ extends to an action of the associated 
\emph{extended Morava stabilizer group}
$$ \GG_n := \Aut(\mc{F}) \rtimes \Gal(\FF/\FF_p). $$
Note that $\GG_n$ implicitly depends both on the formal group $\mc{F}$, and the field $\FF$.

Morava $E$-theory gives rise to an associated variant of Morava K-theory, which is defined to be the spectrum given by the regular quotient.
$$ K_n := E_n/(p, u_1, \ldots, u_{n-1}) $$
so we have
$$ \pi_*K_n = \FF[u^{\pm 1}]. $$
Again, different formal group laws $\mc{F}$ and different fields of definition $\FF$ can give rise to different forms of $K_n$.

\subsection{Topological modular forms}

We give a brief overview of some facts about the spectrum of connective topological modular forms $\tmf$.  A more complete introduction may be found in \cite{Behrenstmf}, \cite{tmf}.  

An \emph{elliptic cohomology theory} consists of a triple
$$ (E, C, \alpha) $$
where $E$ is a complex orientable even periodic ring spectrum, $C$ is an elliptic curve over $E_0$, and $\alpha$ is an isomorphism
$$ \alpha: \widehat{C} \xrightarrow{\cong} F_E $$
between the formal group law $\widehat{C}$ associated of $C$ and the formal group law of $E$.

Goerss, Hopkins, and Miller constructed a sheaf of $E_\infty$-ring spectra $\mc{O}^{top}$ on the \'etale site of the moduli stack of elliptic curves $\mc{M}_{ell}$, with the property that the spectrum of sections 
$$ E_C := \mc{O}^{top}(\mr{spec}(R) \xrightarrow{C} \mc{M}_{ell}) $$
associated to an affine etale open classifying an elliptic curve $C/R$ is an elliptic cohomology theory for the elliptic curve $C$.

The Goerss-Hopkins-Miller sheaf is actually defined over the Deligne-Mumford compactification $\br{\mc{M}}_{ell}$ of the moduli stack $\mc{M}_{ell}$ of elliptic curves.
The spectrum of \emph{non-connective topological modular forms} is defined to be the spectrum of global sections of this sheaf
$$ \Tmf := \mc{O}^{top}(\br{\mc{M}}_{ell}). $$
There is a natural map from the homotopy groups of $\Tmf$ to the ring of integral modular forms for $SL_2(\ZZ)$
\begin{equation}\label{eq:MFmap}
\pi_{2*}\Tmf \to \mr{MF}_*(SL_2(\ZZ)) = \ZZ[c_4, c_6, \Delta]/(c_4^3-c_6^2 = 1728\Delta).
\end{equation}
Here $c_4$ and $c_6$ denote normalizations of the Eisenstein series of weight $4$ and $6$, respectively, and $\Delta$ denotes the discriminant of weight $12$.  The map (\ref{eq:MFmap}) is a rational isomorphism, but is not an isomorphism integrally.  Nevertheless the modular forms $c_4$ and $\Delta^{24}$ are in the image, we shall let $c_4 \in \pi_{8}\Tmf$ and $\Delta^{24} \in \pi_{576}\Tmf$ denote lifts of these modular forms to $\pi_*\Tmf$.

The spectrum of \emph{connective} topological modular forms is defined to be the connective cover of this spectrum
$$ \tmf := \tau_{\ge 0}\Tmf. $$
The spectrum of \emph{periodic topological modular forms} is defined to be the global sections of the sheaf $\mc{O}^{top}$ over the non-singular locus
$$ \TMF := \mc{O}^{top}(\mc{M}_{ell}). $$ 
We have
$$ \TMF = \tmf[\Delta^{-24}] $$
where $\Delta^{24} \in \pi_{576}\tmf$.

Inverting $\Delta^{24}$ has the effect of inverting some power of $v_2$ for every prime $p$, and as such, there is a close relationship between $\TMF$ and $\tmf_{K(2)}$.  There is an equivalence \cite[Prop.~6.6.14]{Behrenstmf}
$$ \tmf_{K(2)} \simeq \TMF^\wedge_{(p,c_4)}. $$

Up to isomorphism, there is a unique supersingular elliptic curve $C$ over $\FF_4$.  The elliptic curve $C$ admits a Weierstrass presentation \cite[III.1]{Silverman}
\begin{equation}\label{eq:Ceq}
y^2 + y = x^3. 
\end{equation}
Let $\widehat{C}$ denote the associated height $2$-formal group over $\FF_4$.
The automorphisms of $C$ induce automorphisms of $\widehat{C}$, giving rise to an inclusion
$$ \Aut(C) \hookrightarrow \GG_2. $$
The $2$-primary $K(2)$-localization of $\tmf$ is then given by \cite[Sec.~5]{Behrensbuilding}
\begin{equation}\label{eq:K2tmf}
 \tmf_{K(2)} \simeq E_2^{h\Aut(C) \rtimes \Gal} 
 \end{equation}
 where $\Gal = \Gal(\FF_4/\FF_2)$.
The form of connective Morava $K$-theory in the equivalence 
$$ \tmf \wedge Z \simeq k(2) $$
of (\ref{eq:tmfsmashZ}) is the form associated to the formal group $\widehat{C}$, regarded as a formal group over $\FF_2$.

Associated to the congruence subgroups $\Gamma_1(n) \le SL_2(\ZZ)$, Hill and Lawson constructed variants $\Tmf_1(n)$ of $\Tmf$ associated to the compactified moduli stacks $\br{\mc{M}}_1(n)$ of elliptic curves with $\Gamma_1(n)$ structure \cite{HillLawson}.  Lawson and Naumann \cite{LawsonNaumann} proved that the connective cover $\tmf_1(3)$ of $\Tmf_1(3)$ gives a form of $BP\bra{2}$ at the prime 2: 
\begin{equation}\label{eq:BP2tmf13}
 \tmf_1(3)_{(2)} \simeq \BP\bra{2} 
\end{equation}
We have
$$ \tmf_1(3)_{K(2)} \simeq E_2^{hC_3 \rtimes \Gal}. $$
Associated to the log-\'etale map
$$ \br{\mc{M}}_1(3) \rightarrow \mc{M}_{ell} $$
given by forgetting $\Gamma_1(3)$-structures, there is a map
$$ \tmf \rightarrow \tmf_1(3) $$
and hence a map
\begin{equation}\label{eq:tmfBP2map}
 \tmf \rightarrow BP\bra{2}.
 \end{equation}
The $K(2)$-localization of this map is given by the canonical inclusion
$$ E_2^{h\Aut(C) \rtimes \Gal} \rightarrow E_2^{hC_3 \rtimes \Gal}. $$

\subsection{Subalgebras and subquotients of the Steenrod algebra}\label{sec:steenrod}
Let $A$ denote the mod $2$ Steenrod algebra and let $A_*$ be its dual. The algebra $A_*$ is a polynomial algebra on the Milnor generators $\xi_i$ of degree $2^i-1$. Letting $\zeta_i = \br{\xi}_i$ denote the conjugates, $A_*$ can also be expressed as
\[A_* = \FF_2[\zeta_1, \zeta_2, \zeta_3, \ldots].\]
The coproduct on $A_*$ is given by
\[\psi(\zeta_k) = \sum_{i+j = k}\zeta_i \otimes \zeta_j^{2^i}.\] 
The elements $\zeta_i$ are dual to the elements $Q_{i-1} \in A$.  The elements $Q_n$ are primitive, satisfy $Q_n^2=0$, and generate an exterior subalgebra
$$ E[Q_0, Q_1, Q_2, \cdots ] \subseteq A. $$ 
Let $A(n)$ be the subalgebra generated by $\sq^1, \ldots, \sq^{2^n}$.

For a $B$ subalgebra of $A$, we will be interested in $A$-modules of the form 
\[A\mmod B := A \otimes_{B} \mathbb{F}_2,\] 
since we have \cite[2.1, 4.1, 4.2]{Ravenel}, \cite{Mathew}
\begin{equation}\label{eq:homologyquotients}
\begin{split}
H^*\bo &\cong A\mmod A(1), \\ 
H^*\tmf &\cong A\mmod A(2), \\ 
H^*BP{\left<n\right>} &\cong A \mmod E[Q_0, \ldots, Q_n],  \\
H^*k(n) & \cong A \mmod E[Q_n].
\end{split}
\end{equation}
We also have 
\begin{align*}
H^*Y &\cong_{A(1)} A(1)\mmod E[Q_1], \\
H^*Z &\cong_{A(2)} A(2)\mmod E[Q_2],
\end{align*}
where $\cong_{A(n)}$ denotes an isomorphism of $A(n)$-modules (the case of $Y$ is elementary, for the case of $Z$ see \cite{bhateggerZ}).

We note that the dual of $A(n)$ and $E[Q_0, \ldots, Q_n]$ are given by
\begin{align*}
A(n)_* &\cong A_*/(\zeta_1^{2^{n+1}}, \zeta_2^{2^n}, \ldots, \zeta_{n+1}^2, \zeta_{n+2}, \ldots), \\
E[Q_0, \ldots, Q_n]_* &\cong E[\zeta_{1}, \ldots, \zeta_{n+1}].
\end{align*}
We will denote the dual of $A \mmod B$ as $A \mmod B_*$.
The duals of $A\mmod A(n)$ and $A\mmod E[Q_0, \ldots, Q_n]$ are given by
\begin{align*}
A\mmod A(n)_* &\cong \FF_2[\zeta_1^{2^{n+1}}, \zeta_2^{2^n}, \ldots, \zeta_{n+1}^2, \zeta_{n+2}, \ldots], \\
A\mmod E[Q_0, \ldots, Q_n]_* &\cong \FF_2[\zeta_{1}^2, \ldots, \zeta_{n+1}^2, \zeta_{n+2}, \zeta_{n+3},  \ldots].
\end{align*}

In general, for $A_*$-comodules $M$ and $N$, the change of rings isomorphism gives
\begin{equation}\label{eq:changeofrings}
 \Ext^{*,*}_{A_*}(M,A \mmod B_* \otimes N) \cong \Ext^{*,*}_{B_*}(M,N). 
 \end{equation}


\section{The good/evil  decomposition of the $E_1$-term}\label{sec:goodevildecomp} 

The goal of this section is to analyze the $E_1$-term of the $\tmf$-resolution for $Z$.  Using (\ref{eq:tmfsmashZ}) we have 
\begin{equation}\label{eq:tmfresE1}
 \E{\tmf}{}_1^{n,t}(Z) = \pi_{t}(\tmf^{n + 1} \wedge Z) \cong k(2)_t (\tmf^{\wedge n}). 
 \end{equation}
For this reason, we will need a tool to compute connective Morava $K$-theory.

\subsection{Margolis homology}\label{sec:margolis}

For a spectrum $X$, consider the Adams spectral sequence for $k(n)_*X$
$$ \E{ass}{}_2^{s,t} = \Ext^{s,t}_{A_*}(\FF_2, H_* k(n) \wedge X) \Rightarrow k(n)_{t-s} (X). $$ 
Using (\ref{eq:homologyquotients}) and the change of rings isomorphism (\ref{eq:changeofrings}), the $E_2$-term of this Adams spectral sequence takes the form
$$ \Ext_{A_*}(\FF_2, H_*k(n) \wedge X) \cong \Ext^{*,*}_{E[Q_n]_*}(\FF_2, H_*X). $$
Note that Ext of comodules over $E[Q_n]_*$ is isomorphic to Ext of modules over $E[Q_n]$, using the dual action of $Q_n$ on homology.
$$ \Ext^{*,*}_{E[Q_n]_*}(\FF_2, H_*X) \cong \Ext^{*,*}_{E[Q_n]}(\FF_2, H_*X). $$
Because the dual action of $Q_n$ on homology lowers degree, we will regard $Q_n$ as having degree $-2^{n+1}+1$.

Margolis (see \cite[Part III]{spectraMargolis}) introduced some general tools for computing such $\Ext$ groups over exterior algebras.


\begin{defn}
Let $M$ be a module over $E[x]$. Let $\ker_x(M)$ be the kernel of multiplication by $x$ and $\im_x(M)$ be its image. Define 
\[H(M; x) := \ker_x(M)/\im_x(M).\]
\end{defn}

\begin{lem}\label{lem:marg}
Let $M$ be an $E[x]$-module, where $x$ has degree $k$. Then there is a short exact sequence
\[0 \rightarrow \im_x(M) \rightarrow \Ext_{E[x]}^{*,*}(\FF_2, M) \rightarrow \FF_2[y] \otimes H(M; x) \rightarrow 0   \]
for $y$ in $\Ext^{1, k}$ and $\im_x(M)$ is regarded as a graded $\FF_2$-vector space in cohomological degree zero.
\end{lem}

\begin{proof}
Consider the standard free resolution of $\FF_2$ as a $E[x]$-module, given by the differential graded $\FF_2$ algebra
$$ E[x] \otimes \Gamma[z] $$
where $\Gamma$ denotes the divided power algebra, $d(z) = x$, and $\abs{z} = (-1,k)$ (here the first index is the cohomological degree, which is negative because it is in positive \emph{homological} degree). 
Applying $\Hom_{E[x]}(-,M)$, gives a cochain complex
$$ C^{*,*}(M) := \FF_2[y] \otimes M $$
whose cohomology is $\Ext_{E[x]}(\FF_2, M)$ where $M$ has cohomological degree $0$, $y = z^*$, $\abs{y} = (1,-k)$, and 
$$ d(y^n \otimes m) = y^{n+1} \otimes x\cdot m. $$ 
We calculate
$$ H^{n,*}(C^{*,*}(M)) = \begin{cases}
H(M;x)\{y^n\}, & n > 0, \\
\ker_x (M), & n = 0.
\end{cases}
$$
The result then follows from the short exact sequence:
$$ 0 \rightarrow \im_x (M) \rightarrow \ker_x(M) \rightarrow H(M;x) \rightarrow 0. \qedhere $$
\end{proof}


We will apply these results to the exterior algebra $E[Q_n]$.
\begin{defn}
Let $M$ be an $A(n)$-module.
The \emph{$n$th Margolis homology} of $M$ is $H(M; Q_n)$. If $M = H_*(X)$, then we abbreviate $ H(H_*(X); Q_n)$ as $H(X; Q_n)$.
\end{defn}


Since $Q_n$ is primitive, the action of $Q_n$ on the tensor product $M \otimes N$ of $A(n)$-modules is given by  
$$ Q_n(a \otimes b) = Q_n(a) \otimes b + a \otimes Q_n(b). $$ 
From this, one can deduce the following lemma.
\begin{lem}
Let $M$ and $N$ be $A(n)$-modules of finite type. Then 
\[H(M\otimes N ; Q_n) \cong H(M ; Q_n) \otimes H(N;Q_n).\]
\end{lem}
\medskip
\begin{cor}\label{cor:margpow}
If $M$ is an $A(n)$-module of finite type, then there is a short exact sequence 
\[0 \rightarrow V^{k,*}(M) \rightarrow \Ext_{E[Q_n]}^{*,*}(\FF_2, M^{\otimes k}) \rightarrow \FF_2[v_n]\otimes H(M ; Q_n)^{\otimes k} \rightarrow 0 \]
where 
$$ V^{k,*}(M) := \im_{Q_n}(M^{\otimes k}). $$  
\end{cor}


The following result is a straightforward consequence the fact that the action of $Q_n$ is a derivation
and 
\[ Q_n(\zeta_k) = \begin{cases} \zeta_{k-n-1}^{2^{n+1}} & k\geq n+1, \\
 0 & k<n+1 . \end{cases} \]
\begin{lem}\label{lem:margAnEn}
There are isomorphisms
\[ H(A \mmod A(n)_*; Q_n) \cong  \FF_2[\zeta_2^{2^{n}}, \zeta_3^{2^{n-1}}, \ldots, \zeta_{n+1}^2, \zeta_{n+2}^2,  \zeta_{n+3}^2, \ldots]/(\zeta_{2}^{2^{n+1}}, \zeta_{3}^{2^{n+1}}, \ldots ) \]
and
\[ H(A \mmod E[Q_0,\ldots, Q_n]_*; Q_n) \cong  \FF_2[\zeta_1^2, \zeta_2^2, \ldots ]/(\zeta_{1}^{2^{n+1}}, \zeta_{2}^{2^{n+1}},  \ldots ). \]
\end{lem}

We end this section with a topological realization result (compare with \cite[Thm.~2]{Lellmann}).  The authors are very grateful to the referee for suggesting this streamlined formulation of the result, and the proof is due to the referee.

\begin{prop}\label{prop:k(n)}
Let $X$ be a connective spectrum with the property that the Margolis homology $H(X; Q_n)$ is concentrated in even degrees.  Then the Adams spectral sequence 
$$ \E{ass}{}_2^{s,t}(k(n)\wedge X) = \Ext^{s,t}_{E[Q_n]}(\FF_2, H_*X) \Rightarrow k(n)_{t-s}(X) $$
collapses, and there are no exotic $v_n$-extensions.  There is a fiber sequence of $k(n)$ modules
$$ HV_X \to k(n) \wedge X \rightarrow K_X $$
where 
$$ V_X := \im_{Q_n}(H_*X), $$
$HV_X$ is the generalized Eilenberg-MacLane spectrum associated to the graded $\FF_2$-vector space $V_X$, and $K_X$ is a free $k(n)$-module with 
$$ \pi_* K_X \cong \FF_2[v_n] \otimes H(X; Q_n). $$ 
This fiber sequence is natural in $X$ with $H_*(X;Q_n)$ in even degrees.  The fiber sequence is split, but not naturally.
\end{prop}

\begin{proof}
By hypothesis, for $s > 0$, $\E{ass}{}_2^{s,t} = 0$ unless $t-s$ is even.  Thus non-zero differentials must originate from $\E{ass}{}_2^{0,t}$ with $t$ odd.  Since $v_n$ annihilates that vector space and $\E{ass}{}_2^{*,*}$ is $v_n$-torsion free in positive cohomological degree, the ASS collapses.

There is an isomorphism 
$$ \FF_2 \otimes_{k(n)_*} k(n)_*X \cong \E{ass}{}_2^{0,*} $$
and a surjection
$$ \cdot v_n: k(n)_*X/\text{$v_n$-torsion} \twoheadrightarrow \E{ass}{}_2^{1,*}. $$ 
There is then a commutative diagram
$$
\xymatrix{
0 \ar[r] &
 V_0 \ar[r] \ar[d] &
k(n)_*X \ar[r] \ar[d] &
k(n)_*X/\text{$v_n$-torsion} \ar[r] \ar[d] &
0 \\
0 \ar[r] &
V_X \ar[r] &
\E{ass}{}_2^{0,*} \ar[r]_{\cdot v_n} &
\E{ass}{}_2^{1,*} \ar[r] &
0
}
$$
The map $V_0 \rightarrow V_X$ is an isomorphism by a diagram chase.

This defines a natural inclusion of $k(n)_*$-modules 
$$ V_X \cong V_0 \subseteq k(n)_*X. $$
A choice of basis for $V_X$ defines a map
$$ HV_X \to k(n) \wedge X $$
which, in the homotopy category of $k(n)$-modules, is independent of the choice.  Any choice of splitting of
$$ V_X \to k(n)_*X \to k(n)_*X/\text{$v_n$-torsion} $$
can be realized in the category of $k(n)$-modules.
\end{proof}

\subsection{The computation of the $E_1$-term of the $\tmf$-ASS for $Z$}\label{sec:setup}

Returning now to the computation of the $E_1$-term $\E{\tmf}{}_1^{n,*}(Z)$ (\ref{eq:tmfresE1}), we will compute the classical ASS
\begin{equation} \label{eqn:ASSfortmfSS}
\Ext_{A_*}^{s,t}(\FF_2,  H_*(k(2) \wedge \tmf^{\s n} )) \Longrightarrow k(2)_{t-s}({\tmf}^{\s n}) = \E{\tmf}{}_1^{n,t-s}(Z).
\end{equation}

Defining
\[\mc{C}^{n,*,*}_{alg}(Z) := \FF_2[v_2] \otimes H(  A\mmod A(2)_*, Q_2)^{\otimes n},\]
Lemma~\ref{lem:marg}, Corollary~\ref{cor:margpow}, and Lemma~\ref{lem:margAnEn} imply the following.

\begin{prop}\label{prop:algsplit}There is a short exact sequence of 
$\FF_2[v_2] $-modules
\begin{equation}\label{eqn:ASSn} 
0 \to V^{n,*,*}_{alg}(Z) \to \Ext_{A_*}^{*,*}( \FF_2, H_*k(2) \wedge {\tmf}^{\s n} ) \to \mc{C}^{n,*,*}_{alg}(Z) \to 0
\end{equation}
 where 
$$
\mc{C}_{alg}^{n,*,*} \cong \FF_2[v_2] \otimes \left\lbrack \FF_2[\zeta_2^4, \zeta_3^2, \zeta_4^2, \cdots]/(\zeta_2^8, \zeta_3^8, \cdots) \right\rbrack^{\otimes n}
$$ 
and
 $V^{n,*,*}_{alg}(Z)$ is a direct sum of shifted copies of $\FF_2$'s which are simple $v_2$-torsion (i.e., $v_2 \cdot x=0$ for all elements $x$) which are concentrated in Adams filtration zero:  
 \[ V^{n,*}(Z) := V^{n,0,*}_{alg}(Z) =V^{n,*,*}_{alg}(Z) .  \] 
\end{prop} 

There is one subtle issue which we now must discuss: both sides of the equivalence
\begin{equation}\label{eq:tmfZ=k(2)} 
\alpha: \tmf \wedge Z \xrightarrow{\simeq} k(2) 
\end{equation} 
have potentially different notions of $v_2$-multiplication.  The spectrum $Z$ has a $v_2$-self map
$$ v_2: \Sigma^6 Z \rightarrow Z $$
and $k(2)$ has the multiplication-by-$v_2$ map
$$ \cdot v_2: \Sigma^6 k(2) \rightarrow k(2). $$ 
Since the self-map of $Z$ is a $K(2)$-equivalence, and since $\pi_6(k(2))$ only consists of 2 elements, it is easy to see that the following diagram commutes.
$$
\xymatrix{
S^6 \ar[r] \ar[dr] & 
\Sigma^6\tmf \wedge Z \ar[r]^{1 \wedge v_2} & 
\tmf \wedge Z \ar[d]^{\simeq}_{\alpha} \\
& \Sigma^6 k(2) \ar[r]_{\cdot v_2} & k(2) 
}
$$
I.e., the two notions of ``$v_2$'' are the same when regarded as elements of $\pi_6$.
However, this does not imply that the self map
$$ 1 \wedge v_2: \Sigma^6 \tmf \wedge Z \rightarrow \tmf \wedge Z $$
is homotopic to the multiplication-by-$v_2$ map on $k(2)$, because the map $1 \wedge v_2$ does not necessarily give a map of $k(2)$-modules under the equivalence (\ref{eq:tmfZ=k(2)}).

However, all of our computations of $\E{\tmf}{}_1^{*,*}(Z)$ will arise from the Adams spectral sequence, and the following lemma makes it clear that on the level of the Adams spectral sequence the two notions of $v_2$-multiplication are the same.  In particular, the ``$v_2$'' in Proposition~\ref{prop:algsplit} may be taken to be the one coming from the $v_2$-self map on $Z$.

\begin{lem}
The diagram
$$
\xymatrix{
\Sigma^6 \tmf \wedge Z \ar[r]^{1 \wedge v_2} \ar[d]_{\simeq}^{\alpha} &
\tmf \wedge Z \ar[d]_{\simeq}^{\alpha} \\
\Sigma^6 k(2) \ar[r]_{\cdot v_2} & 
k(2)
}
$$
commutes up to elements of higher Adams filtration.
\end{lem}

\begin{proof}
The cofiber of the $v_2$-self map on $Z$
$$ \Sigma^6 Z \xrightarrow{v_2} Z \rightarrow A_2 $$
is a spectrum whose cohomology is free of rank 1 over $A(2)$.  We therefore deduce that there is a cofiber sequence
$$ \Sigma^6 \tmf \wedge Z \xrightarrow{1 \wedge v_2} \tmf \wedge Z \rightarrow H. $$ 
Consider the following diagram of cofiber sequences.
$$
\xymatrix{
\Sigma^{6} \tmf \wedge Z \ar[r]^{1 \wedge v_2} \ar@{.>}_{\beta} [d] &  
\tmf \wedge Z \ar[r] \ar[d]_{\simeq}^{\alpha} &
H \ar@{=}[d] \\
\Sigma^6 k(2) \ar[r]_{\cdot v_2} &
k(2) \ar[r] &
H  
}
$$
The right square in this diagram commutes, since $H^0(\tmf \wedge Z) = \FF_2$ has no non-trivial automorphisms.  Therefore the dotted map $\beta$ exists, making the diagram commute.  Since the top and bottom rows are cofiber sequences, $\beta$ must be an equivalence.  Since $H^*k(2)$ is generated by the non-trivial element of $H^0k(2)$ as an $A$-module, $\alpha$ and $\beta$ must induce the \emph{same} map on cohomology.  Therefore the difference $\alpha - \beta$ is in positive Adams filtration, and the result follows.
\end{proof}

Henceforth, by ``$v_2$'' we shall \emph{always} be referring to the $v_2$-multiplication arising from the self-map on $Z$.  

The following is an immediate corollary of Proposition~\ref{prop:algsplit} and Proposition~\ref{prop:k(n)}. 

\begin{cor}\label{cor:htpysplit}
There is a short exact sequence of $\FF_2[v_2]$-modules
\begin{equation}\label{eqn:htpysplit}
0 \rightarrow V^{n,*}(Z) \to \E{\tmf}{}^{n,*}_1(Z) \to \mc{C}^{n,*}(Z) \to 0,
\end{equation}
where $V^{*,*}(Z)$ is the 
module defined in Proposition~\ref{prop:algsplit}, and, 
$$
\mc{C}^{n,*}(Z) \cong \FF_2[v_2] \otimes \left\lbrack \FF_2[\zeta_2^4, \zeta_3^2, \zeta_4^2, \cdots]/(\zeta_2^8, \zeta_3^8, \cdots) \right\rbrack^{\otimes n}.
$$
\end{cor}

\subsection{The good and evil complexes}

We now upgrade the decomposition of Corollary~\ref{cor:htpysplit} to a short exact sequence of chain complexes.  The first observation is the following.

\begin{prop}
	The subspace $V^{*, *}(Z)$ forms a subcomplex of $\E{\tmf}{}^{*, *}_1(Z)$.
\end{prop}

\begin{proof}
This follows from the fact that the subspace $V^{*,*}(Z)$ is the subspace of of $v_2$-torsion, and the differentials commute with $v_2$-multiplication. \end{proof}

We will call $(V^{*,*}(Z), d_1)$ the \emph{evil complex}.
Since $(V^{*,*}(Z), d_1)$ forms a sub-complex of $ \E{\tmf}{}_1^{*,*}(Z)$, we can define $\mc{C}^{*,*}(Z)$ to be the quotient complex
\[ 0 \rightarrow V^{*,*}(Z) \rightarrow \E{\tmf}{}_1^{*,*}(Z) \rightarrow \mc{C}^{*,*}(Z) \rightarrow 0. \]
We will call $(\mc{C}^{*,*}(Z),d_1)$ the \emph{good complex}.

Abbreviate $ H^{*,*}(V) = H(V^{*,*}(Z),d_1)$ and $ H^{*,*}(\mc{C}) = H(\mc{C}^{*,*}(Z),{d}_1)$.
There is a long exact sequence
\begin{equation}\label{eq:LES}
\cdots \rightarrow H^{*,*}(V) \rightarrow \E{\tmf}{}_2^{*,*}(Z) \rightarrow H^{*,*}(\mc{C}) \xrightarrow{\partial} H^{*+1, *}(V) \rightarrow \cdots.
\end{equation}
We will see that $H^{*,*}(\mc{C})$ can be almost completely computed, while $H^{*,*}(V)$ is mysterious. We call the elements of $H^{*,*}(V)$ \emph{evil} and those of $H^{*,*}(\mc{C})$ \emph{good}. 

In \cite{BBBCX}, we establish a method for computing $H^{*,*}(V)$ in a range. The idea is to use the $\tmf$-Mahowald spectral sequence (MSS),
\begin{equation}\label{eqn:algss} \E{\tmf}{alg}_1^{n,s,t} = \Ext^{s,t}_{A}(H^*({\tmf}^{\s n+1} \s Z), \FF_2) \Rightarrow \Ext^{s+n,t}_{A}(H^*(Z), \FF_2). \end{equation}
with
$$ d_r: \E{\tmf}{alg}_r^{n,s,t} \rightarrow \E{\tmf}{alg}_r^{n+r, s-r+1, t}. $$
The construction of this spectral sequence is identical to that of \cite{BBBCX}.
The $E_1$-term fits into an exact sequence of chain complexes
 \[ 0 \to V^{*,*,*}_{alg}(Z) \to \E{\tmf}{alg}_1^{*,*,*} \to \mc{C}^{*,*,*}_{alg}(Z) \to 0  \]
 (see \eqref{eqn:ASSn})
from which we obtain a long exact sequence
\begin{equation}\label{eq:algLES}
\cdots \rightarrow H^{*,*,*}(V_{alg}) \rightarrow \E{\tmf}{alg}_2^{*,*,*}(Z) \rightarrow H^{*,*,*}(\mc{C}_{alg}) \xrightarrow{\partial_{alg}} H^{*+1, *,*}(V_{alg}) \rightarrow \cdots.
\end{equation}

We will compute the cohomology $H^{*,*,*}(\mc{C}_{alg}) $ explicitly, and the abutment of the $\tmf$-MSS (\ref{eqn:algss}) can be computed through a range, for example using Bruner's program. 
From this, we can inductively deduce information about $H^{*,*,*}(V_{alg})$, at least through a range. Further, $H^{n,s,t}(V_{alg})$ is concentrated in degree $s=0$ and
 the identification of cochain complexes
\[V^{n,t}(Z) \cong V^{n,0,t}_{alg}(Z)\]
implies that 
\[H^{*,*}(V) \cong H^{*,0,*}(V_{alg}). \]
This isomorphism allows us to transfer information from the $\tmf$-MSS to the $\tmf$-ASS.

In order to understand $\E{\tmf}{}_2^{*,*}(Z) $ and $ \E{\tmf}{alg}_2^{*,*,*}(Z) $, the first step is to compute $H^{*,*}(\mc{C}) $ and $H^{*,*,*}(\mc{C}_{alg})$  (see Theorems~\ref{thm:MRE1} and \ref{thm:HCalg} and Remark~\ref{rmk:MRE1}).

\section{Morava stabilizer groups and algebras}\label{sec:MSG}

Our goal will be to relate the good complex to the cobar complex for a certain subquotient $\sitd(2)$ of a form of the Morava stabilizer algebra --- this will be done in Section~\ref{sec:gooddiffs}.  The purpose of this section is to prepare some computations which we will use in the next section.  Of particular importance will be Proposition~\ref{prop:E2Z}, which gives a computation of the action of the group 
$$ G_{48} := \Aut(C) \rtimes \Gal < \GG_2 $$
on the $E_2$-homology of the finite complex $Z$.

\subsection{The Morava stabilizer algebra}\label{sec:morstabalg}

Historically, the forms of Morava $K$-theory $K(n)$ and Morava $E$-theory $E_n$ were typically taken to be those associated to the \emph{Honda height $n$ formal group} $\mc{H}_n$.  In the case of $K(n)$, it is regarded as a formal group over $\FF_p$, and in the case of $E_n$ is is regarded as a formal group over $\FF_{p^n}$.  The Honda height $n$ formal group law is characterized as the unique $p$-typical formal group law with $p$-series given by \cite[A2.1]{Ravenel}
$$ [p]_{\mc{H}_n}(x) = x^{p^n}. $$
Its endomorphism ring is given by
$$ \mr{End}(\mc{H}_n) \cong \WW(\FF_{p^n})\bra{S}/(Sa=a^\sigma S, \: S^n = p) $$
where $\WW(\FF_{p^n})$ is the Witt ring of $\FF_{p^n}$, and $\sigma$ is the lift of Frobenius.  Every endomorphism $\phi \in \mr{End}(\mc{H}_n)$ can be written uniquely as 
$$ a_0 + a_1S + a_2S^2 + \cdots $$
with $a_i \in \WW(\FF_{p^n})$ satisfying $a_i^{p^n} = a_i$.  The associated Morava stabilizer group is given by 
$$ \MS_n := \Aut(\mc{H}_n) = \{ \sum_{i}a_i S^i \in \mr{End}(\mc{H}_n) \: : \: a_0 \ne 0\}. $$

Because we are using $K(2)$, $K_2$, and $E_2$ to denote the forms of Morava $K$- and $E$-theory associated to the formal group $\widehat{C}$, we will let $K(2)'$, $K'_2$, $E_2'$ denote the forms of Morava $K$- and $E$-theory associated to the Honda height $2$ formal group $\mc{H}_2$.
The associated \emph{Morava stabilizer algebra} $\Sigma(2)$ is the Hopf algebra over $K(2)'_*$ given by \cite[Sec.~6.1]{Ravenel}
\begin{align}\label{eq:Sidef}
\Sigma(2) & := K(2)'_* \otimes_{BP_*} BP_*BP \otimes_{BP_*} K(2)'_*  \nonumber\\
& \cong \mathbb{F}_{2}[v_2^{\pm1}][t_1, t_2, \ldots]/(t_k^{4} - v_2^{2^k-1}t_k).
\end{align}
The $2$-periodic extension $K_2'$ of $K(2)'$ has homotopy groups
\[ (K'_2)_* \cong \FF_4[u^{\pm 1}]\]
with $\abs{u} = -2$ and
$$ v_2 = u^{-3}. $$  
We let 
$$ \Sigma_2 := (K'_2)_* \otimes_{K(2)'_*} \Sigma(2) $$
denote the associated Hopf algebra over $(K'_2)_*$.

Let 
$$ S_2  = \left\{ \sum_{i\geq 0} a_i S^i \in \MS_2 \: : \: a_0 = 1 \right\} $$ 
denote the $2$-Sylow subgroup of $\mathbb{S}_2$. 
The Morava stabilizer algebra $(\FF_4[u^{\pm 1}], \Sigma_2)$ can be regarded as an algebra of functions on $S_2$:
\begin{align}\label{eq:Siexdef}
\Sigma_2 
&\cong \Map^c(S_2, (K_2)_*) \nonumber \\
& \cong \mathbb{F}_{4}[u^{\pm1}][ t_1, t_2, \ldots]/(t_k^{4} - v_2^{2^k-1}t_k).
\end{align}
Here, $\Map^c$ denotes the continuous functions where $S_2$ is given its profinite topology and $(K_2)_*$ is given the discrete topology, and the functions $t_k$ are defined as
\begin{equation}\label{eq:tkdef}
t_k(1 +a_1S +a_2S^2+ \ldots) = a_k u^{1-2^k}.
\end{equation}
The coproduct $\psi$ is determined by $\psi(t_k) = \sum  t_k'\otimes t_k''$ where
\[t_k(g'g'')  = \sum t_k'(g')t_k''(g''), \ \ \ \  g',g'' \in S_2.\]

The cohomology of $\Sigma_2$ was essentially 
studied by Ravenel in \cite[Thm.~6.3.27]{Ravenel}, and Ravenel's approach to this computation will be used in Section~\ref{sec:good} to give an essential foothold in the computation of the cohomology of the good complex.

\subsection{The elliptic Morava stabilizer group}\label{sec:ellMSG}

We will begin this subsection with a discussion of the extended Morava stabilizer group associated to the unique isomorphism class of supersingular elliptic curve $C$ defined over $\FF_2$, and its relationship with both $\TMF$ and the more traditionally studied Morava stabilizer group associated to the Honda height $2$ formal group $\mc{H}_2$.
We will then introduce a certain quotient $\br{\Sigma}_2$ of $\Sigma_2$ associated to an open subgroup of this extended Morava stabilizer group.

We first recall some facts about the automorphism group of the supersingular elliptic curve $C$, and its associated formal group.
We refer to \cite{beaudry_towards} and \cite{henn_centr} for more details in this context.

Over $\FF_4$, the endomorphism ring of the elliptic curve $C : y^2+y =x^3$
is the maximal order (the Hurwitz integers)
$$ \mr{End}(C) = \ZZ\left\{1, i, j, \frac{1+i+j+k}{2}\right\} $$
in the quaternion algebra
$$ D = \QQ\bra{i,j}/(i^2 = -1, \: j^2 = -1, \: ij = -ji). $$
with $k := ij$ \cite[pp.~237-9]{Deuring}.  
Define
$$ \omega = -\frac{1}{2}(1+i+j+k). $$
Then we have 
$$ \omega^3 = 1, \quad \omega^2 + \omega + 1 = 0, $$
and 
$$ \omega i \omega^2 = j, \: \omega j \omega^2 = k, \: \omega k \omega^2 = i. $$
The automorphism group of $C$ is the subgroup of $D^\times$ generated by 
$$ Q_8 = \{ \pm 1, \pm i, \pm j, \pm k \} $$
and $\omega$, so we have
$$ G_{24} := \mr{Aut}(C) = Q_8 \rtimes C_3. $$

We define 
$$ T := j-k \in \mr{End}(C) $$
so we have
$$ T^2 = -2. $$
Then $D$ has the alternative presentation as
\begin{equation}\label{eq:Tpresentation}
\QQ(\omega)\bra{T}/(Ta = a^\sigma T, \: T^2 = -2) 
\end{equation}
where $\omega^\sigma = \omega^2$ is the action of the Galois group
$$ \Gal := \Gal(\QQ(\omega)/\QQ) \cong \Gal(\FF_4/\FF_2) = \bra{\sigma}. $$
For example, $i \in D$ can be expressed as $\frac{1}{1+2\omega}(1-T)$  in \eqref{eq:Tpresentation}.

Since the curve $C$ is defined over $\FF_2$, the Galois group $\Gal$ also acts on $\rm{End}(C)$, and hence on $\Aut(C)$ and $D$.  This action is encoded in the following lemma. 

\begin{lem}
The Galois action on an element $x \in D$ is given by
$$ x^\sigma = -\frac{1}{2}TxT. $$
\end{lem}

\begin{proof}
  As discussed in Section~\ref{sec:background}, the elliptic curve $C$ admits a Weierstrass presentation
  \begin{equation}\label{eq:Weierstrass}
   y^2 + y = x^3. 
   \end{equation}
This means that for an $\FF_2$-algebra $R$, the $R$-points of the elliptic curve $C$ is given by
$$ C(R) = \{ (x,y) \in R^2 \: : \: y^2+y=x^3 \} \cup \{\infty\}. $$
The $\FF_4$ points of $C$ form a group isomorphic to $\FF_3 \times \FF_3$.
A basis for this $\FF_3$-vector space is given in $(x,y)$ coordinates by
\begin{align*}
P_1 & := (0,0), \\
P_2 & := (1,\omega).
\end{align*}
The generators $i$ and $\omega$ of the group $G_{24} = \Aut(C)$ correspond to the automorphisms
\begin{align*}
i: (x,y) & \mapsto (x+1, y+x+\omega^2), \\
\omega: (x,y) & \mapsto (\omega^2 x, y).
\end{align*} 
The induced action of these automorphisms on the 
$\FF_4$-points of the curve $C$, 
with respect to the basis $(P_1,P_2)$, induces a representation
$$ \rho: G_{24} \hookrightarrow GL_2(\FF_3) $$
with
\begin{align*}
\rho(i) & = 
\begin{bmatrix}
0 & 1 \\ -1 & 0 
\end{bmatrix}, \\
\rho(\omega) & =   
\begin{bmatrix}
1 & -1 \\ 0 & 1 
\end{bmatrix}.
\end{align*}
The Galois action on $C(\FF_4)$ extends the representation $\rho$ to an isomorphism
\begin{equation}\label{eq:G48}
\td{\rho}: G_{48} := G_{24} \rtimes \Gal \xrightarrow{\cong} GL_2(\FF_3)
\end{equation}
given by
$$ \td{\rho}(\sigma) = 
\begin{bmatrix}
1 & 0 \\ 0 & -1 
\end{bmatrix}.
$$
One can therefore use this isomorphism to deduce that
\begin{align*}
i^\sigma & = -i, \\
\omega^\sigma & = \omega^2.
\end{align*}
One easily checks from this:
\begin{align*}
 T^\sigma & = (j-k)^\sigma \\
 & = (\omega i \omega^2 - \omega^2 i \omega)^\sigma \\
 & = \omega^2 (-i) \omega - \omega(-i)\omega^2 \\
 & = T. 
 \end{align*}
From the presentation (\ref{eq:Tpresentation}), every element $x \in D$ takes the form
$$ x_0 + x_1 T $$
with $x_i \in \QQ(\omega)$.
We then compute
\begin{align*}
-\frac{1}{2} TxT & = -\frac{1}{2}T(x_0 + x_1 T)T \\
& = -\frac{1}{2}Tx_0 T - \frac{1}{2}T x_1 T^2 \\
& = -\frac{1}{2}T^2 x_0^\sigma + T x_1 \\
& = x_0^\sigma + x_1^\sigma T \\
& = (x_0 + x_1 T)^\sigma \\
& = x^\sigma. \qedhere
\end{align*}
\end{proof}

The formal group of $\widehat{C}$ has $-2$-series
$$ [-2]_{\widehat{C}}(x) = x^4. $$
The endomorphism ring of the formal group $\widehat{C}$ is the maximal order
$$ \mr{End}(\widehat{C}) = \WW(\FF_4)\bra{T}/(Ta=a^\sigma T, \: T^2 = -2) $$
in the $2$-adic division algebra 
$$ D_2 := D \otimes \QQ_2, $$
where $\WW(\FF_4) = \ZZ_2[\omega]/(\omega^2+\omega+1)$ is the Witt ring.
The associated Morava stabilizer group 
$$ \MS_2 := \mr{Aut}(\widehat{C}) $$
is the group of units in the order $\mr{End}(\widehat{C})$.
Since $\widehat{C}$ is defined over $\FF_2$, its automorphism group $\MS_2$ also gets an action of $\Gal$, with Galois action given by
$$ g^\sigma = -\frac{1}{2}TgT, $$
and we let 
$$ \GG_2 := \MS_2 \rtimes \Gal $$
denote the resulting extended Morava stabilizer group.  The subgroup $G_{48}$ is a maximal finite subgroup of $\GG_2$.

It is observed in \cite[Sec.~3.1]{beaudry_towards} and \cite{henn_centr} that the formal group $\widehat{C}$ and the Honda formal group $\mc{H}_2$ have isomorphic endomorphism rings.  Explicitly, one gets an isomorphism
$$ \mr{End}(\widehat{C}) \cong \mr{End}(\mc{H}_2) $$ 
by mapping 
\begin{equation}\label{eq:TtoS}
 T \mapsto \alpha S, 
\end{equation} 
where
$$ \alpha = \frac{1-2\omega}{\sqrt{-7}} \in \WW(\FF_4) $$
(for the choice of $\sqrt{-7} \in \ZZ_2$ with $\sqrt{-7} \equiv 1 \pmod 4$).  The essential property of $\alpha$ is that
$$ \alpha \alpha^\sigma = -1. $$
This induces an isomorphism 
\begin{equation}\label{eq:MSGiso}
 \Aut(\mc{H}_2) \cong \aut(\widehat{C}) = \MS_2. 
 \end{equation}
\emph{However, this isomorphism is not $\Gal$-equivariant!} 

Thus the group $\MS_2$ admits two different Galois actions, one coming from the natural Galois action on $\Aut(\widehat{C})$ and one coming from the natural Galois action on $\Aut(\mc{H}_2)$ using the isomorphism (\ref{eq:MSGiso}).  We shall let
$$ \Gal < \Aut(\MS_2) $$
denote the subgroup generated by the Galois automorphism $\sigma$ coming from $C$, and let
$$ \Gal' < \Aut(\MS_2) $$
be the subgroup generated by the Galois automorphism $\sigma'$ coming from $\mc{H}_2$. 
The action of $\sigma'$ is given by 
$$ g^{\sigma'} = \frac{1}{2}SgS. $$
We will denote the corresponding  
extended Morava stabilizer group by
$$ \GG'_2 := \MS_2 \rtimes \Gal'. $$

\begin{lem}\label{lem:sigma'}
For $g \in \MS_2$ we have
$$ g^\sigma = -\alpha g^{\sigma'} \alpha^\sigma. $$
\end{lem}

\begin{proof}
We compute
\begin{align*}
g^\sigma & = -\frac{1}{2}TgT \\
& = -\frac{1}{2}\alpha S g \alpha S \\
& = -\alpha \frac{1}{2}S g S \alpha^\sigma  \\
& = -\alpha g^{\sigma'} \alpha^\sigma. \qedhere
\end{align*}
\end{proof}

The inclusion of $G_{24}$ in $\MS_2$ gives a splitting of the short exact sequence
\[ 1 \to K \to \mathbb{S}_2 \to G_{24} \to 1  \]
where $K$ is the open normal subgroup of $\MS_2$ 
\begin{equation}\label{eq:K}
 K = \{ 1 + a_2S^2 + a_3 S^3 + \cdots \in \MS_2 \: : \: a_2 \in \{0,\omega\} \}
\end{equation}
discussed at length in Section 2.5 of \cite{beaudryresolution}.

The inclusion of groups
\[ K \hookrightarrow S_2  \]
corresponds to a quotient of Hopf algebras
\[\Sigma_2 \to \overline{\Sigma}_2\]
where
\begin{equation}\label{eq:barSigma_2} 
\begin{split}
\overline{\Sigma}_2 & =  \Map^c(K,\FF_4[u^{\pm}]) \\
& \cong \Sigma_2/(t_1, \omega v_2t_2+t_2^2)
\end{split}
\end{equation} 
(compare with \cite[Proposition 6.3.30]{Ravenel}, but Ravenel's choice of $K$ is Galois conjugate to ours).

\subsection{The Morava $E$-homology of $Z$}

In this subsection we will use the computations of \cite{bhateggerK2Z} to derive the following result (where $G_{48}$ is the group (\ref{eq:G48})).

\begin{prop}\label{prop:E2Z}
There is an isomorphism of $G_{48}$-modules
\[(E_2)_*Z \cong  \mr{CoInd}_{C_3 \rtimes \mr{Gal}}^{G_{48}} \FF_4[u^{\pm1}] \]
where $C_3 \rtimes \mr{Gal} $ acts on $\FF_4[u^{\pm1}]$ via
\[\omega_*(\lambda u^k) = \lambda \omega^k u^k, \ \ \ \ \sigma_*(\lambda u^k) = \lambda^{\sigma} u^k. \]
\end{prop}

\begin{cor}\label{cor:E2Z}
There is an isomorphism of $Q_8$-modules 
\[(E_2)_*Z \cong  \mr{CoInd}_{1}^{Q_8} \FF_4[u^{\pm1}]. \]
\end{cor}

The proof of Proposition~\ref{prop:E2Z} will require some preliminary recollections from \cite{bhateggerK2Z}.  Recall we are using $E'_2$ to denote the Morava $E$-theory spectrum associated to the Honda height $2$ formal group over $\FF_4$.  The spectrum $E'_2$ has an action of the extended Morava stabilizer group $\GG_2' = \MS_2 \rtimes \Gal'$ of the previous subsection.

The third author and Egger computed $(E'_2)_*Z$ as
\begin{equation}\label{eq:E'Z}
 (E'_2)_* Z \cong \FF_4[u^\pm]\{\bar{x}_0, \bar{x}_2, \bar{x}_4, \bar{x}_6, \bar{y}_6, \bar{y}_8, \bar{y}_{10}, \bar{y}_{12} \}, \quad \abs{\bar{x}_i} = \abs{\bar{y}_i} = 0,
 \end{equation}
with an explicit action of $\MS_2$ \cite[Table~1]{bhateggerK2Z}.  Since the generators $u^{i/2}\bar{x}_i$ and $u^{i/2}y_i$ are in the image of the map
$$ BP_*Z \rightarrow (E'_2)_*Z, $$
they have trivial action of the Galois group $\Gal'$, and therefore $\Gal'$ acts on (\ref{eq:E'Z}) by acting on $\FF_4$.
Following the proof of \cite[Thm.~4.12]{bhateggerK2Z}, we see that for any $x \in (E_2')_0Z$ with 
\begin{equation}\label{eq:xformula}
x = \bar{y}_{12} + \alpha_0 \bar{x}_0 + \alpha_2 \bar{x}_2 + \alpha_4 \bar{x}_4 + \alpha_6 \bar{x}_6, \quad \alpha_i \in \FF_4 
\end{equation}
we have\footnote{In the notation of \cite{bhateggerK2Z}, we have $x = k \cdot c_3'+ \text{terms involving $c_i$}$, where $k \in Q_8$ is the unit quaternion.} 
\begin{equation}\label{eq:Q8E'Z}
 (E_2')_0Z = \FF_4[Q_8]\{x\}. 
 \end{equation}

\begin{proof}[Proof of Proposition~\ref{prop:E2Z}]
Let $\bar{E}_2$ denote the Morava $E$-theory associated to the height $2$ Honda formal group over the algebraic closure $\bar{\FF}_2$, with action of 
$$ \bar{\GG}'_2 = \MS_2 \rtimes \Gal(\bar{\FF}_2/\FF_2). $$
 Let $\sigma'$ denote the Frobenius, regarded as a generator of $\Gal(\bar{\FF}_2/\FF_2)$, acting on $\MS_2$ as in the previous subsection.  Then we have
 $$ E_2' \simeq \bar{E}^{h\bra{(\sigma')^2}}_2. $$
Since the formal group of the elliptic curve $C$ is isomorphic to the Honda formal group over $\bar{\FF}_2$, we deduce that the associated Morava $E$-theory is the same, but the action of the Galois group is different.  The calculations of the previous subsection imply that if we define
$$ \sigma := \alpha \sigma' \in \bar{\GG}'_2 $$ 
then the Morava $E$-theory associated to the formal group of $C$ over $\FF_4$ is given by
$$ E_2 \simeq \bar{E}^{h\bra{(\sigma)^2}}_2. $$
Since $\sigma^4 = (\sigma')^4$, we deduce that $E_2$ and $E'_2$ have the common extension
$$ E_2'' := \bar{E}_2^{h\bra{\sigma^4}}. $$
We therefore have
$$ (E_2'')_0Z = \FF_{16} \otimes_{\FF_4} (E_2')_0 Z \cong \FF_{16}[Q_8]\{ x\} $$
for any $x$ of the form (\ref{eq:xformula}) (with $\alpha_i \in \FF_{16}$).
Let $\td{\omega} \in \FF^{\times}_{16}$ be a generator, so that
$$ \td{\omega}^{\sigma^4} = \td{\omega}^{16} = \td{\omega}. $$
Since $\td{\omega} + \td{\omega}^4 \in \FF_4$ we can take $\td{\omega}$ so that
$$ \td{\omega} + \td{\omega}^4 = \omega \in \FF_4. $$
Define
$$ x := \bar{y}_{12}+(1+\td{\omega}^4 + \td{\omega}^8)\bar{x}_6 + (a+b)(\td{\omega}+\td{\omega}^8)\bar{x}_0 $$
(where $a,b \in \FF_2$ are those associated to the choice of $Z \in \td{\mc{Z}}$ as in \cite[Lem.~3.5]{bhateggerK2Z}).
Then it follows from \cite[Table~1]{bhateggerK2Z} and
$$ \alpha = 1 + 2\omega \mod 4 $$
that
\begin{enumerate}
\item $\sigma = \alpha\sigma'$ acts trivially on $x$,
\item $\bra{\omega} = C_3 < \MS_2$ acts trivially on $x$,
\item $x$ generates $(E_2'')_0Z$ as a free $\FF_{16}[Q_8]$-module.
\end{enumerate} 
It follows that $x$ generates
$$ (E_2)_0Z \cong [(E''_2)_0Z]^{\bra{\sigma^2}} $$
as an $\FF_4[Q_8]$-module.  This, together with (1) and (2) above, implies 
\[ (E_2)_*Z \cong  \mr{CoInd}_{C_3 \rtimes \mr{Gal}}^{G_{48}} \FF_4[u^{\pm 1}] \cong \Map_{C_3\rtimes \mr{Gal}}( G_{48}, \FF_4[u^{\pm1}]).  \]
\end{proof}

While Proposition~\ref{prop:E2Z} describes $(E_2)_*Z$ as a $G_{48}$-module, it is natural to ask for a similar description of $(E_2)_*Z$ as a $\GG_2$-module.  The following proposition does almost that: it computes $(E_2)_*Z$ as an $\MS_2$-module using the subgroup $K$ (\ref{eq:K}).  

\begin{prop}\label{prop:S2E2Z}
There is an isomorphism of $\MS_2$-modules
$$ (E_2)_*Z \cong \mr{CoInd}_{C_3 \ltimes K}^{\MS_2}\FF_4[u^{\pm 1} ]. $$
\end{prop}

\begin{cor}\label{cor:S2E2Z}
There is an isomorphism of $S_2$-modules
$$ (E_2)_*Z \cong \mr{CoInd}_{K}^{S_2}\FF_4[u^{\pm 1} ]. $$
\end{cor}

\begin{rmk}
It is tempting to look for an analog of Proposition~\ref{prop:S2E2Z} which also incorporates the Galois action, but this is complicated by the fact that the subgroup $K$ is not Galois invariant.
\end{rmk}

\begin{proof}Using the notation of Proposition~\ref{prop:E2Z}, consider the diagram
$$
\xymatrix@R-1em{
& (E'_2)_*Z \ar[dr] \\
BP_*Z \ar[ru] \ar[dr] & & (E''_2)_*Z \\
& (E_2)_*Z \ar[ru]
} $$
Since the generators 
$$ u^{i/2}\bar{x}_i, u^{i/2}\bar{y}_i \in (E'_2)_*Z $$ 
of (\ref{eq:E'Z}) come from $BP_*Z$, it follows that 
\begin{equation}\label{eq:EZgens}
 (E_2)_* Z \cong \FF_4[u^{\pm 1}]\{\bar{x}_0, \bar{x}_2, \bar{x}_4, \bar{x}_6, \bar{y}_6, \bar{y}_8, \bar{y}_{10}, \bar{y}_{12} \}.
 \end{equation}
The action of $\MS_2$ on (\ref{eq:EZgens}) is computed in Table~1 of \cite{bhateggerK2Z}.  In particular, it is easy to check that the map
$$ \pi: (E_2)_*Z \rightarrow \FF_4[u^{\pm 1} ] $$
given by
$$ \pi(\alpha_0\bar{x}_0 \cdots \alpha_6\bar{x}_6+\beta_0\bar{y}_6+\cdots +\beta_{12}\bar{y}_{12}) = \beta_{12} $$
is $C_3 \ltimes K$-equivariant.  Thus it induces a $\MS_2$-equivariant map
$$ \td{\pi}: (E_2)_*Z \rightarrow \mr{CoInd}_{C_3\ltimes K}^{\MS_2}\FF_4[u^{\pm 1}]. $$
This can be checked to be an isomorphism using (\ref{eq:Q8E'Z}) and the fact that the composite
$$ Q_8 \rightarrow \MS_2 \rightarrow \MS_2/(C_3 \ltimes K) $$
is an isomorphism.
\end{proof}

\section{Computation of the differentials in the good complex}\label{sec:gooddiffs}

 The main result of this section (Definition~\ref{defn:S2tilde}, Theorem~\ref{thm:goodcomplex})  is that there is a sub-Hopf algebra
$$(k(2)_*, \sitd(2)) \subset ((K_2)_*, \br{\Sigma}_2) $$ 
such that the good complex is isomorphic to the associated cobar complex \cite[Definition A1.2.11]{Ravenel}:
$$ \mc{C}^{*,*}(Z) \cong C^*_{\sitd(2)}(k(2)_*).$$ 

\subsection{The good complex as a subcomplex of the cobar complex of $\br{\Sigma}_2$.}

The map $\tmf \to \TMF$ induces a map of spectral sequences 
\begin{equation}\label{eq:maptoEO}
  \E{\tmf}{}_r^{*,*}(Z) \to \E{\TMF}{}_r^{*,*}(Z). 
  \end{equation}
The kernel of $\E{\tmf}{}_1^{*,*}(Z) \to \E{\TMF}{}_1^{*,*}(Z)$ is $V^{*,*}(Z)$ and the image is  
$$ \mc{C}^{*,*}(Z) \subseteq \E{\TMF}{}_1^{*,*}(Z). $$  
We will now show the complex $\E{\TMF}{}_1(Z)$ can be regarded as a subcomplex of the cobar complex for the Hopf algebra $\br{\Sigma}_2$.  
The first step will be to express the $E_1$-term in terms of the Morava stabilizer group (Corollary~\ref{cor:E1TMFZ}). 

For a profinite set $T = \varprojlim_i T_i$ and an abelian group $M$, let 
$$ \Map^c(T,M) = \varprojlim \Map(T_i, M) $$
denote the abelian group of continuous maps, where $T$ is given the profinite topology, and $M$ is given the discrete topology.  If $G$ is a group which acts on $T$ and on $M$, then there is an induced \emph{conjugation action} on $\Map^c(T,M)$, given by
$$ (g\cdot f)(t) = g f(g^{-1}t) $$
for $g \in G$, $f \in \Map^c(T,M)$, and $t \in T$.

\begin{lem}\label{lem:2tmfcase}
There is a $\GG_2$-equivariant isomorphism
\[(E_2)_*(\TMF \smsh Z) \cong  \Map^c(\GG_2/G_{48}, (E_2)_*Z) \]
(where $\GG_2$ acts on $\Map^c$ by the conjugation action on functions), and this leads to an isomorphism
\[\pi_*  \TMF \wedge \TMF \wedge Z \cong   \Map^c_{C_3 \rtimes \mr{Gal}}(\GG_2/G_{48}, \FF_4[u^{\pm 1}])\]
where $\Map^c_{C_3 \rtimes \mr{Gal}}(\GG_2/G_{48}, \FF_4[u^{\pm 1}])$ denotes 
the $C_3 \rtimes \mr{Gal}$ equivariant continuous maps.
\end{lem}

\begin{proof}
Since $Z$ is a type $2$ complex, $X \wedge Z$ is $K(2)$-local for any $E(2)$-local spectrum $X$ (see proof of \cite[Lem.~7.2]{HoveyStrickland}).  
In particular, (\ref{eq:K2tmf}) implies
\begin{equation}\label{eq:TMFZE48}
 \TMF \wedge Z \simeq E_2^{hG_{48}} \wedge Z.
 \end{equation}
Using the fact that for finite groups, homotopy fixed points and homotopy orbits of $K(2)$-local spectra are $K(2)$-locally equivalent \cite{Kuhn},
we get
 \[ \TMF \wedge \TMF \wedge Z \simeq E_2^{hG_{48}} \smsh E_2^{hG_{48}}  \smsh Z \simeq (E_2 \smsh (E_2^{hG_{48}} \smsh Z))^{hG_{48}}.\] 
 We use the homotopy fixed point spectral sequence
 \[H^s(G_{48}, (E_2)_t(E_2^{hG_{48}} \smsh Z)) \Longrightarrow \pi_{t-s}  \TMF \wedge \TMF \wedge Z .\]
 By \cite[Corollary 2.1]{detsphere},
 \[ (E_2)_*(E_2^{hG_{48}} \smsh Z)  \cong (E_2)_*(E_2 \smsh Z)^{hG_{48}} \cong \Map^c(\GG_2/G_{48}, (E_2)_*Z) \]
 with action of $G_{48}$ given by the conjugation action on functions.
 Since we have an isomorphism of $G_{48}$-modules
\[ (E_2)_*Z \cong  \mr{CoInd}_{C_3 \rtimes \mr{Gal}}^{G_{48}} \FF_4[u^{\pm 1}] \cong \Map_{C_3\rtimes \mr{Gal}}( G_{48}, \FF_4[u^{\pm1}])  \]
it follows that 
\[ (E_2)_*(E_2 \smsh Z)^{hG_{48}}  \cong  \Map_{C_3\rtimes \mr{Gal}}( G_{48},\Map^c(\GG_2/G_{48},\FF_4[u^{\pm1}])) . \]
In particular, the $E_2$-term of the homotopy fixed point spectral sequence is 
 \[H^*(G_{48}, (E_2)_*(E_2 \smsh Z)^{hG_{48}} ) \cong H^*(C_3\rtimes \mr{Gal}, \Map^c(\GG_2/G_{48},\FF_4[u^{\pm1}])).\]
 Since $C_3$ has order coprime to $2$ and $\Gal$ acts freely on $\FF_4$, the $E_2$-term is concentrated in degree $s=0$, and is given by
 \[ \Map^c(\GG_2/G_{48},\FF_4[u^{\pm1}])^{C_3\rtimes \mr{Gal}}. \]
 The spectral sequence collapses, giving the result.
\end{proof}

\begin{cor}\label{cor:E1TMFZ} 
For $s\geq 1$, there  is a $\GG_2$-equivariant isomorphism 
\[( E_2)_*(\TMF^{\wedge s} \wedge Z) \cong \Map^c( (\GG_2/G_{48})^{\times s}, (E_2)_*Z)  \]
with the diagonal action on  $(\GG_2/G_{48})^{\times s}$ and action on $\Map^c$  the conjugation action on functions. 
This leads to an isomorphism
\[  \E{\TMF}{}^{s,*}_1(Z) \cong \pi_*\TMF^{\wedge s+1} \wedge Z \cong \Map^c_{C_3 \rtimes \mr{Gal}}(\underbrace{\GG_2 \times_{G_{48}} \cdots \times_{G_{48}} \GG_2}_s/G_{48}, \FF_4[u^{\pm 1}]). \]
The left action of $C_3 \rtimes \Gal$ on 
\[\GG_2 \times_{G_{48}} \cdots \times_{G_{48}} \GG_2/G_{48}\] 
is via by left multiplication on the first factor of $\GG_2$.
\end{cor}
\begin{proof}
We proceed by induction on $s$.  The case of $s = 1$ is Lemma~\ref{lem:2tmfcase}.
Suppose that the claim holds for $s-1$. Then
\begin{align*}
E_2 \smsh \TMF^{\wedge s} \wedge Z& \simeq  E_2 \wedge E_2^{hG_{48}} \wedge \TMF^{\wedge (s-1)} \wedge Z  \\
& \simeq  (E_2 \wedge E_2 \wedge \TMF^{\wedge (s-1)} \wedge Z)^{hG_{48}}   
\end{align*}
where  $G_{48}$ acts on the second copy of $E_2$. The $E_2$-page of the homotopy fixed point spectral sequence is given by
\[H^*(G_{48}, (E_2)_*(E_2 \wedge \TMF^{\wedge (s-1)} \wedge Z) ). \]
Furthermore, 
\begin{align*}
(E_2)_*(E_2 \wedge \TMF^{\wedge (s-1)} \wedge Z) & \cong \Map^c(\GG_2, (E_2)_*\TMF^{\wedge (s-1)} \wedge Z )\\
&\cong  \Map^c(\GG_2, \Map^c((\GG_2/G_{48})^{\times (s-1)} ,  (E_2)_*Z) ).
\end{align*}
It follows that
\begin{align*} 
H^*(G_{48}, (E_2)_*(E_2 \wedge \TMF^{\wedge (s-1)} \wedge Z) )& \cong H^0(G_{48}, (E_2)_*(E_2 \wedge \TMF^{\wedge (s-1)} \wedge Z) )\\
& \cong \Map^c((\GG_2/G_{48})^{\times s} ,  (E_2)_*Z ).
\end{align*}
which proves the first claim. 

Next, 
\begin{align*}
\TMF^{\wedge (s+1)} \wedge Z& \simeq  (E_2 \wedge \TMF^{\wedge s} \wedge Z)^{hG_{48}}.  
\end{align*}
We use the homotopy fixed point spectral sequence again, together with the fact that 
\begin{align*} 
(E_2)_* (\TMF^{\wedge s} \wedge Z) &\cong \Map^c((\GG_2/G_{48})^{\times s} ,  (E_2)_*Z ) \\
&\cong \Map^c((\GG_2/G_{48})^{\times s} ,  \Map_{C_3 \rtimes  \mr{Gal}} (G_{48}, \FF_4[u^{\pm 1}]) ) \\
&\cong  \Map_{C_3 \rtimes \mr{Gal}} (G_{48}, \Map^c((\GG_2/G_{48})^{\times s} ,\FF_4[u^{\pm 1}]) ).
\end{align*}
The proof of the first isomorphism is finished in a way analogous to that of Lemma~\ref{lem:2tmfcase}.

For a group $G$, a subgroup $H \le G$, an a $G$-set $X$, the shearing isomorphism is the isomorphism
\begin{align*}
 G \times_H X & \xrightarrow{\cong} G/H \times X, \\
 (g,x) & \mapsto (g, gx).
 \end{align*}
Note that the shearing isomorphism is $G$-equivariant, where $G$ acts on the source through its action on the left factor, and $G$ acts on the target through the diagonal action.

Iterating the shearing isomorphism yields a $\GG_2$-equivariant isomorphism
$$ \underbrace{\GG_2 \times_{G_{48}} \cdots \times_{G_{48}}\GG_2}_s/G_{48} \cong (\GG_2/G_{48})^{\times s}, $$
and we therefore have an isomorphism
\[ \Map^c_{C_3 \rtimes \mr{Gal}}((\GG_2/G_{48})^{\times s}, \FF_4[u^{\pm 1}]) \cong \Map^c_{C_3 \rtimes \mr{Gal}}(\underbrace{\GG_2 \times_{G_{48}} \cdots \times_{G_{48}} \GG_2}_s/G_{48}, \FF_4[u^{\pm 1}]).\]
\end{proof}

It is not immediately clear how the groups
$$ \Map^c_{C_3 \rtimes \mr{Gal}}(\GG_2^{\times_{G_{48}} s} /G_{48}, \FF_4[u^{\pm 1}]) $$
in Corollary~\ref{cor:E1TMFZ} form a cochain complex.  We will now address this by showing that they are a subcomplex of the $E_2$-based Adams spectral sequence for $Z$.

The map of spectra $\TMF \rightarrow E_2$ induces a map of Adams spectral sequences.  The induced map on $E_1$-terms
$$ \E{\TMF}{}_1(Z) \rightarrow \E{E_2}{}_1(Z) $$
is given by the canonical inclusion
\begin{align*}
\Map^c_{C_3 \rtimes \mr{Gal}}(\GG_2^{\times_{G_{48}} s} /G_{48}, \FF_4[u^{\pm 1}])
& \cong \Map^c_{G_{48}}(\GG_2^{\times_{G_{48}} s} /G_{48}, \mr{CoInd}_{C_3 \rtimes \mr{Gal}}^{G_{48}} \FF_4[u^{\pm 1}]) \\
& \subseteq
\Map^c(\GG_2^{s}, \mr{CoInd}_{C_3 \rtimes \mr{Gal}}^{G_{48}} \FF_4[u^{\pm 1}])
\end{align*}
where, by Proposition~\ref{prop:E2Z}, the latter is the cobar complex for $\GG_2$ acting on $(E_2)_*Z$:
$$ C^s_{\GG_2}((E_2)_*Z) \cong \E{E_2}{}_1^{s,*}(Z). $$
In particular, the differential in the cobar complex for $\GG_2$ restricts to give the differential on the subcomplex
$$ \Map^c_{C_3 \rtimes \mr{Gal}}(\GG_2^{\times_{G_{48}} s} /G_{48}, \FF_4[u^{\pm 1}])  \subseteq
\Map^c(\GG_2^{s}, (E_2)_*Z).
$$

We now have the following lemma.

\begin{lem}\label{lem:embeddingSigma2}
There is an embedding of cochain complexes
$$ \E{\TMF}{}_1(Z) \subset C^*_{\br{\Sigma}_2}((K_2)_*). $$
where $C^*_{\br{\Sigma}_2}((K_2)_*)$ is the cobar complex of the Hopf algebra $((K_2)_*, \br{\Sigma}_2)$ of (\ref{eq:barSigma_2}).
\end{lem}

\begin{proof}
The injection comes from the map $(*)$ in the following diagram
$$
\xymatrix@R-1.5em@C+1em{
\E{\TMF}{}^{s,*}_1(Z) \ar@{=}[d] & C^s_{\br{\Sigma}_2}((K_2)_*) \ar@{=}[d]
\\
\Map^c_{C_3 \rtimes \mr{Gal}}(\GG_2^{\times_{G_{48}} s} /G_{48}, \FF_4[u^{\pm 1}]) \ar@{^{(}->}[dd]_{\alpha} \ar@{.>}[r]^-{(*)} &
\Map^c(K^s, \FF_4[u^{\pm 1}]) 
\\ \\
\Map^c(\GG_2^{\times_{G_{48}} s} /G_{48}, \FF_4[u^{\pm 1}]) \ar@{^{(}->}[r]_-\beta 
& \Map^c(\GG_2^s, (E_2)_*Z) \ar[uu]_\gamma
}
$$
where $\alpha$ is the natural inclusion,  $\beta$ is the composite coming from the isomorphism of Corollary~\ref{cor:E2Z}:
\begin{align*}
\beta: \Map^c(\GG_2^{\times_{G_{48}} s} /G_{48}, \FF_4[u^{\pm 1}]) 
& \cong \Map^c_{Q_8}(\GG_2^{\times_{G_{48}} s} /G_{48}, \mr{CoInd}_{1}^{Q_8}\FF_4[u^{\pm 1}]) \\
& \cong \Map^c_{Q_8}(\GG_2^{\times_{G_{48}} s} /G_{48}, (E_2)_*Z) \\
& \hookrightarrow \Map^c(\GG_2^{\times_{G_{48}} s} /G_{48}, (E_2)_*Z),
\end{align*}
and $\gamma$ is the composite coming from the isomorphism of Corollary~\ref{cor:S2E2Z}:
\begin{align*}
 \gamma: \Map^c(\GG_2^s, (E_2)_*Z) & \rightarrow \Map^c(S_2^s, (E_2)_*Z) \\
 & \cong \Map^c(S_2^s, \mr{CoInd}_{K}^{S_2}\FF_4[u^{\pm 1}]) \\
 & \rightarrow \Map^c(K^s, \mr{CoInd}_{K}^{S_2}\FF_4[u^{\pm 1}]) \\
 & \xrightarrow{(ev_1)_*} \Map^c(K^s, \FF_4[u^{\pm 1}]).
 \end{align*} 
Here, $(ev_1)_*$ is the map induced by the evaluation at $1 \in S_2$ map:
$$ ev_1: \mr{CoInd}_K^{S_2} \FF_4[u^{\pm 1}] = \Map_K(S_2, \FF_4[u^{\pm 1}]) \to \FF_4[u^{\pm}]. $$
The map $\gamma$ is easily seen to be a map of cochain complexes.
The discussion prior to the statement of this lemma implies that the composite $\beta \circ \alpha$ is a map of cochain complexes.  This implies that $(*)$ is a map of cochain complexes.  
It follows from the fact that the composite  
\begin{equation}\label{eq:Khomeo}
 K \rightarrow \GG_2 \rightarrow \GG_2/G_{48}
 \end{equation}
is a homeomorphism that the composite $\gamma \circ \beta$ is an isomorphism.  Since $\alpha$ is injective, we deduce that $(*)$ is injective.
\end{proof}

\subsection{The sub-Hopf algebra $\pmb{\td{\Sigma}(2) \subset \br{\Sigma}_2}$.}\label{sec:subhopf}

We shall now study a sub-Hopf algebra $(K(2)_*, \td{\Sigma}(2))$  of the Hopf algebra $((K_2)_*, \br{\Sigma}_2)$ of (\ref{eq:barSigma_2}) such that the image of $\E{\TMF}{}_1(Z)$ in the cobar complex for $\br{\Sigma}_2$ is the cobar complex for $\td{\Sigma}(2)$.

Define Hopf algebras
$$ \td{\Sigma}(2) \subset \br{\Sigma}(2) \subset \br{\Sigma}_2 $$
by letting $\td{\Sigma}(2)$ be the image of the map
\begin{align*}
\Map^c_{C_3 \rtimes \mr{Gal}}(\GG_2 /G_{48}, \FF_4[u^{\pm 1}]) & \hookrightarrow 
\Map^c(K, \FF_4[u^{\pm 1}]) = \br{\Sigma}_2
\end{align*}
and letting $\br{\Sigma}(2)$ be the image of the map
\begin{align*}
\Map^c_{C_3}(\GG_2 /G_{48}, \FF_4[u^{\pm 1}]) & \hookrightarrow 
\Map^c(K, \FF_4[u^{\pm 1}]) = \br{\Sigma}_2.
\end{align*}

Under the isomorphism
$$ \Map^c(\GG_2/G_{48}, \FF_4[u^{\pm}]) \cong \Map^c(K,\FF_4[u^{\pm}]) = \br{\Sigma}_2 $$
coming from the homeomorphism (\ref{eq:Khomeo}), the conjugation action of $C_3 \rtimes \Gal$ on $\Map^c(\GG_2/G_{48}, \FF_4[u^{\pm}])$ induces an action of $C_3 \rtimes \Gal$ on $\br{\Sigma}_2$ such that
\begin{align*}
\br{\Sigma}(2) & = \br{\Sigma}_2^{C_3},  \\
\td{\Sigma}(2) & = \br{\Sigma}(2)^{\Gal} = \br{\Sigma}_2^{C_3 \rtimes \Gal}.
\end{align*}
We now compute this action of $C_3 \rtimes \Gal$ on
\begin{align}\label{eq:defbrSiext}
 \br{\Sigma}_2 = \FF_4[u^{\pm 1}][\tb_2, \tb_3, \cdots ]/ (\tb_2^2 + \omega v_2 \tb_2, \tb_k^4 + v_2^{2^k-1}\tb_k). 
 \end{align}
Here we use $\tb_k$ to denote the image of $t_k \in \Sigma_2$ (see \ref{eq:tkdef}) in $\br{\Sigma}_2$.
Let $\sigma$ be the generator of $\Gal$, and we will denote the generator of $C_3 \subset \GG_2$ by $\omega$, our fixed choice of 3rd root of unity.

Recall \cite{beaudryresolution} that elements $x \in K$ can be written as
$$ x = 1 + a_2S^2 + a_3S^3 + \cdots $$
with $a_2 \in \{0, \omega\}$ and $a_i \in \{0,1,\omega,\omega^2\}$ for $i>2$.  The function
$$ \tb_i \in \br{\Sigma}_2 = \Map^c(K,\FF_4[u^{\pm 1}]) $$
is given on elements $x$ as above by the formula
$$ \tb_i(x) = a_i u^{1-2^i}. $$
Under the isomorphism
$$ \Map^c(K, \FF_4[u^{\pm 1}]) \cong \Map^c(\GG_2/G_{48}, \FF_4[u^{\pm 1}])$$ 
the function $\tb_i$ is given on a coset $gG_{48}$ by 
\begin{equation}\label{eq:tbidef}
 \tb_i(gG_{48}) = t_i(x)
\end{equation}
where $x$ is the unique element of $K$ so that $xG_{48} = gG_{48}$.

Note that $C_3$ acts on $\FF_4[u^{\pm 1}]$ through $\FF_4$-algebra maps by the formula
$$ \omega \cdot u = \omega u $$
and $\mr{Gal}$ acts through the Galois action on $\FF_4$,
so
$$ \FF_4[u^{\pm 1}]^{C_3 \rtimes \mr{Gal}} = \FF_2[v_2^{\pm1}]. $$

\begin{lem}
The functions $\tb_k \in \br{\Sigma}_2$ are $C_3$ equivariant, so the conjugation action of $C_3$ on $\tb_k$ is trivial. 
\end{lem}

\begin{proof}
We have (for $a_2 \in \{0,\omega\}$):
\begin{align*}
\tb_k(\omega(1+a_2S^2+a_3 S^3+\cdots)G_{48}) & = \tb_k((\omega+\omega a_2S^2+\omega a_3 S^3 +\cdots)G_{48}) \\
& = \tb_k((\omega+\omega a_2S^2+\omega a_3 S^3 +\cdots)\omega^2 G_{48}) \\
& = \tb_k((1+ a_2S^2+\omega^2 a_3 S^3+\cdots) G_{48}) \\
& = \begin{cases}
a_k u^{1-2^k}, & k \: \text{even}, \\
\omega^2 a_k u^{1-2^k}, & k \: \text{odd}
\end{cases} \\
& = \omega \cdot a_ku^{1-2^k} \\
& = \omega \cdot \tb_k((1+a_2S^2+a_3 S^3+\cdots)G_{48}). \qedhere
\end{align*}
\end{proof}

\begin{cor}\label{cor:brSidesc}
The sub-Hopf algebra $\br{\Sigma}(2) \subset \br{\Sigma}_2$ is given by
$$ \br{\Sigma}(2) = \FF_4[v_2^{\pm1}][\tb_2, \tb_3, \cdots ]/ (\tb_2^2 + \omega v_2 \tb_2, \tb_k^4 + v_2^{2^k-1}\tb_k). $$
\end{cor}

\begin{lem}
We have 
$$\sigma \cdot \tb_2 = \omega\tb_2$$
and the element $\td{t}_2 := \omega^2\tb_2 \in \br{\Sigma}(2)$ is Galois invariant.
\end{lem}

\begin{proof}
For $a_2 \in \{0,\omega\}$, we compute the conjugation action on $\tb_2$ (\ref{eq:tbidef}) using the fact that $\sigma^{-1}=\sigma$, Lemma~\ref{lem:sigma'}, and the fact that $\alpha \equiv 1 \pmod 2$:
\begin{align*}
(\sigma \cdot \tb_2)((1+a_2S^2 + \cdots)G_{48}) 
& = \sigma [\tb_2(\sigma(1+a_2S^2 + \cdots)G_{48})] \\
& =  \sigma [\tb_2(-\alpha(1+a^\sigma_2 S^2+\cdots)\alpha^\sigma G_{48})] \\
& =  \sigma [ \tb_2((1+a^\sigma_2 S^2+\cdots) G_{48})] \\
& = \begin{cases}
\sigma [\tb_2((1+ 0S^2 + \cdots) G_{48})], & a_2 = 0, \\
\sigma [\tb_2((1+ \omega^2S^2 +\cdots) G_{48})], & a_2 = \omega.
\end{cases}
\end{align*}
Now if $a_2 = 0$, it follows we have
\begin{align*}
\sigma [\tb_2(\sigma(1+0S^2+ \cdots)G_{48})]
& = 0 \\
& = \omega\tb_2((1+0S^2+\cdots)G_{48}).
\end{align*}
However, if $a_2 = \omega$, the element
$$ (1+ \omega^2S^2 +\cdots) $$
is not in $K$, and we have to rectify this by adjusting it by right multiplication with 
$$ -1 = 1 + S^2 + S^4 + \cdots \in \GG_2 $$
 to get it into $K$.  
We have 
\begin{align*}
\sigma [\tb_2((1+ \omega^2S^2 + \cdots) G_{48})]
 & = \sigma [\tb_2((1+ \omega^2S^2 + \cdots) (-1)G_{48})] \\
 & = \sigma [t_2((1+ \omega^2S^2 + \cdots) (-1))] \\
& = \sigma[\omega u^{-3}] \\
& =  \omega^2 u^{-3} \\
& = \omega \tb_2((1+\omega S^2 + \cdots)G_{48}).\qedhere
\end{align*}
\end{proof}

\begin{defn}\label{defn:S2tilde}
Define $\sitd(2)$ to be the image of the composite
$$ \pi_* \tmf \wedge \tmf \wedge Z \rightarrow \pi_*\TMF \wedge \TMF \wedge Z \hookrightarrow \br{\Sigma}(2). $$
\end{defn}

\begin{lem}
The Hopf algebra structure on $(\FF_4[v_2^{\pm1}], \br{\Sigma}(2))$ restricts to a Hopf algebra structure on $(k(2)_*,\sitd(2))$.
\end{lem}

\begin{proof}
The only thing which is not obvious is that the coproduct of $\br{\Sigma}(2)$ restricts to a coproduct on $\sitd(2)$.
Using the fact that $\tmf \wedge Z \simeq k(2)$, it suffices to consider the diagram, where $\ell$ is the unit:
$$
\xymatrix@C-1em{
k(2)_*(S \wedge \tmf) \ar[r] \ar[d]_{(\ell\wedge 1)_*} & 
K(2)_*(S \wedge \TMF) \ar[r] \ar[d]_{(\ell \wedge 1)_*} & 
\br{\Sigma}(2) \ar[dd]^{\psi} \\
k(2)_* (\tmf \wedge \tmf) \ar[r]^{(1)} & 
K(2)_* (\TMF \wedge \TMF) &  \\
k(2)_*(\tmf) \otimes_{k(2)_*} k(2)_*\tmf \ar[r]_-{(2)} \ar[u]^{(*)} &
K(2)_*(\TMF) \otimes_{K(2)_*} K(2)_*\TMF \ar[r] \ar[u]_\cong &
\br{\Sigma}(2) \otimes_{\FF_4[v_2^{\pm1}]} \br{\Sigma}(2)
}
$$
Since $(*)$ is an isomorphism after inverting $v_2$, it follows that maps $(1)$ and $(2)$ have isomorphic images. The result follows.
\end{proof}

We will now explain how $\sitd(2)$ has a decreasing ``Adams filtration''.  Recall that we have
\begin{align*}
\E{ass}{}^{s,t}_2(\tmf \wedge \tmf \wedge Z) & \cong 
\E{ass}{}^{s,t}_2(k(2) \wedge \tmf) \\
& \cong
\FF_2[v_2][\zeta_2^4, \zeta_3^2, \zeta_4^2, \ldots]/(\zeta_i^8) \\
& \quad \oplus \text{simple $v_2$-torsion in the $s = 0$ line}.
\end{align*}
Here the generators lie in $(t-s,s)$ bidegrees:
\begin{align*} 
\abs{\zeta_2^4} & = (12,0), \\
\abs{\zeta_i^2} & = (2(2^{i}-1),0), \\
\abs{v_2} & = (6,1).
\end{align*}
The Adams spectral sequence collapses, and endows $k(2)_*\tmf$ with its \emph{Adams filtration}.
 
The \emph{$v_2$-Bockstein filtration} on $k(2)_*\tmf$ is the decreasing filtration given by 
$$ \{ (v_2^i)k(2)_*\tmf \}. $$
The Adams filtrations and $v_2$-Bockstein filtrations on $k(2)_*\tmf$ agree. 
This implies that if we endow $\br{\Sigma}(2)$ with a decreasing multiplicative Adams filtration where we declare that $v_2$ has Adams filtration $1$ and that $\br{t}_i$ has Adams filtration $0$, then the map
$$ k(2)_*\tmf \to \br{\Sigma}(2) $$
preserves Adams filtration, and therefore the image of this map $\sitd(2)$ inherits an Adams filtration which is compatible with that of $k(2)_*\tmf$ and $\br{\Sigma}(2)$. 

\begin{thm}\label{thm:goodcomplex}
The Hopf algebra $\sitd(2) \subset \br{\Sigma}(2)$ has the form 
$$ \sitd(2) = \FF_2[v_2^{\pm1}][\td{t^2_2}, \td{t}_3, \ldots ]/((\td{t_2^2})^2 = v^2_2 \td{t_2^2}, \: \td{t}_k^4 = \text{terms with Adams filtration $> 0$}) $$
where $\td{t_2^2} = (\omega^2 \tb_2)^2$ and for $k \ge 3$
$$ \td{t}_k = \tb_k + \text{terms of higher Adams filtration}. $$
There is an isomorphism of cochain complexes
$$ \mc{C}^{*,*}(Z) \cong C^*_{\sitd(2)}(k(2)_*). $$
\end{thm}

\begin{proof}
By Lemma~\ref{lem:embeddingSigma2}, it suffices to establish that the image of the map
$$ \pi_* \tmf^{\wedge n+1} \wedge Z \rightarrow \E{\TMF}{}_1(Z) \hookrightarrow C^*_{\br{\Sigma}(2)}(K(2)_*) $$
is what we claim it is.  We focus on the case of $n = 2$; it will be apparent that the general case is essentially the same.  
By Lemma~\ref{lem:margAnEn} and Proposition~\ref{prop:k(n)}, the ASS
$$ \E{ass}{}_2^{s,t}(k(2) \wedge BP\bra{2}) \Rightarrow k(2)_{t-s}BP\bra{2} $$
has $E_2$ term
\begin{align*} 
\E{ass}{}^{s,t}_2(k(2) \wedge \tmf)
&  \cong
\FF_2[v_2][\zeta_1^2, \zeta_2^2, \zeta_3^2, \ldots]/(\zeta_i^8) \\
& \quad \oplus \text{simple $v_2$-torsion in the $s = 0$ line}.
\end{align*}
The images of the elements $t_i \in BP_*BP$ under the map
\begin{equation}\label{eq:tiBPBP}
 BP_*BP \to k(2)_*BP\bra{2} 
 \end{equation}
(where $BP\bra{2}$ is the Wilson spectrum of (\ref{eq:BP2tmf13})) are detected by $\zeta_i^2$ in the ASS for $k(2)_*BP\bra{2}$.  In particular, the elements $t_i \in k(2)_*BP\bra{2}$ have Adams filtration $0$.

Since the Adams filtration and $v_2$-Bockstein filtration on $k(2)_*\tmf$ agree, an element in $K(2)_*\tmf \cong K(2)_*\TMF$ is in the image of the map
$$\E{\tmf}{}_1^{1,*}(Z) \cong k(2)_*\tmf \rightarrow v_2^{-1}k(2)_*\tmf \cong K(2)_*\TMF\cong \E{\TMF}{}_1^{1,*}(Z) $$
if and only if it is detected (in the localized Adams spectral sequence) by an element in the image of the map
$$ \E{ass}{}_2(k(2)_*\tmf) \rightarrow v_2^{-1}\E{ass}{}_2(k(2)_*\tmf). $$
Consider the commutative diagram coming from the map (\ref{eq:tmfBP2map})
\begin{equation}\label{eq:k2tmfsquare}
\xymatrix{
k(2)_*\tmf \ar[r] \ar[d]_{(1)} & K(2)_*\TMF \ar[d] \ar@{^{(}->}[r] & \Map^c(\GG_2/G_{48}, \FF_4[u^{\pm 1}]) \ar@{^{(}->}[d]^{(3)} \\
k(2)_*BP\bra{2} \ar[r]_{(2)} & K(2)_*E_2 \ar@{^{(}->}[r] & \Map^c(\GG_2,\FF_4[u^{\pm 1}]) 
}
\end{equation}
We wish to determine which $v_2$ multiple of $\td{t}_2$ is in positive Adams filtration.  To that end, we must compute the image of $\td{t}_2$ under map (3) in (\ref{eq:k2tmfsquare}).
This is tantamount to computing, for $g \in \GG_2$, the value $\td{t}_2(gG_{48})$.
Since we have already established $\td{t}_2$ is $C_3 \rtimes \mr{Gal}$-equivariant, we may assume
$$ g = 1 + a_1 S + a_2S^2 + \cdots. $$
Write $a_1 = \alpha \omega + \beta \omega^2$ with $\alpha, \beta \in \FF_2$.  Using the fact that the elements $j$ and $k$ in $G_{48}$ are given by
\begin{align*}
j & = 1 + \omega^2S + \omega S^2 + \cdots \\
k & = 1 + \omega S+\omega S^2 + \cdots 
\end{align*}
(see \cite{beaudryresolution}) we compute:
\begin{align*}
\td{t}_2(gG_{48})
& = \td{t}_2((1 + (\alpha\omega + \beta \omega^2)S + a_2S^2 + \cdots )G_{48}) \\
& = \td{t}_2((1 + (\alpha\omega + \beta \omega^2)S + a_2S^2 + \cdots )k^\alpha j^\beta G_{48}) \\
& = \td{t}_2((1 + (a_2 + (\alpha+\beta) + \alpha\beta \omega^2) S^2 + \cdots ) G_{48}).
\end{align*}
Let 
$$ \mr{Tr}, \mr{N}: \FF_4 \rightarrow \FF_2 $$
be the trace and norm, respectively, so that $\mr{Tr}(a) =a +a^{\sigma}$ and $\mr{N}(a) = aa^{\sigma}$.
From the definition of $\td{t}_2$ we find
$$ \td{t}_2((1+a_2S^2 + \cdots)G_{48}) = \mr{Tr}(a_2)u^{-3}. $$
It follows from the above calculation that
$$ \td{t}_2((1+a_1S + a_2S^2 + \cdots)G_{48}) = (\mr{Tr}(a_2)+\mr{N}(a_1))u^{-3}. $$
Thus the image of $\td{t}_2$ under map (3) in (\ref{eq:k2tmfsquare}) is the image of
$$ t_2+t_2^2v_2^{-1}+t_1^3 $$
under map (2).  Since the elements $t_i \in k(2)_*BP\bra{2}$ all have Adams filtration $0$, it follows that $v_2\td{t}_2 = \td{t}_2^2 \in K(2)_*\TMF$ lifts to an element
\begin{equation}\label{eq:tdt2}
\td{t}_2^2 = v_2t_2 + t_2^2 + v_2t_1^3 
\end{equation}
of $k(2)_*\tmf$.  

For $k \ge 3$, we \emph{define} $\td{t}_k \in \sitd(2)$ 
to be the image of an element of $k(2)_*\tmf$ detected by $\zeta_k^2$.
Since in the Adams spectral sequence for $k(2)_*BP\bra{2}$ the element $\zeta_k^2$ detects $t_k$, we deduce that the image of $\td{t}_k$ under (1) satisfies
$$ \td{t}_k = t_k + \text{terms of positive Adams filtration}. $$
The result for $n = 2$ follows.

Similar reasoning shows that the image of 
$$\E{\tmf}{}_1^{n,*}(Z) \cong   k(2)_*\tmf^{\wedge n} \rightarrow K(2)_*\TMF^{\wedge n} = \td{\Sigma}(2)^{\otimes_{K(2)_*} n}  \cong \E{\TMF}{}_1^{n,*}(Z)$$
is $\sitd(2)^{\otimes_{k(2)_*}n}$.
\end{proof}

\begin{rmk}\label{rmk:S2barstructure}
Note that while we do not know the full structure of $\sitd(2)$ because of the complicated action of $\Gal$ on $\br{\Sigma}(2)$, we do completely know the structure of $\sibr(2) := \sitd(2) \otimes \FF_4 \subset \br{\Sigma}(2)$:
\begin{align}\label{eq:sibrdef} \sibr(2) = \FF_4[v_2][\td{t^2_2}, \tb_3, \cdots ]/ ((\td{t_2^2})^2 + v_2^2 \td{t^2_2}, \tb_k^4 + v_2^{2^k-1}\tb_k). \end{align}
\end{rmk}

\begin{table}
\begin{center}
\begin{tabular}{|  c | c | }
\hline
Name & Location  \\
\hline
\hline
$\Sigma(2) $   &   \eqref{eq:Sidef} \\
$\Sigma_2$  &   \eqref{eq:Siexdef}  \\
$\overline{\Sigma}_2 $ & \eqref{eq:barSigma_2}  \& \eqref{eq:defbrSiext} \\
$ \br{\Sigma}(2)$  & Cor.~\ref{cor:brSidesc}   \\
$\td{\Sigma}(2)$ & Sec.~\ref{sec:subhopf}  \\
$\si(2)$ &  \eqref{eq:sidef} \\
$\sibr(2)$   & \eqref{eq:sibrdef}  \\
$\sitd(2)$   & Def.~\ref{defn:S2tilde} \& Th.~\ref{thm:goodcomplex}\\
\hline
\end{tabular}
\end{center}
\caption{List of Hopf algebras and where to find them. Our convention is to use the bar to
remind the reader that we have taken a quotient. A symbol with a tilde denotes a sub-algebra of the same symbol with a bar. The lowercase denotes the respective
connective versions. }
\label{fig:listofhopfs}
\end{table}


\section{The cohomology of the good complex}\label{sec:good}

In the previous section we established that
$$ \mc{C}^{*,*}(Z) \cong C^*_{\sitd(2)}(k(2)_*). $$
In this section we will compute the $E_1$-term of a spectral sequence which computes the cohomology
\[H^*(\sitd(2) ) := H^*(C^*_{\sitd(2)}(k(2)_*))\cong   H(\mc{C}^{*,*}(Z)) =: H^{*,*}(\mc{C}).  \]  
In our low dimensional range, it will turn out that there are no possible differentials in this spectral sequence.

\subsection{Overview of the strategy}

Recall from the previous section that we really only have a complete understanding of the base change 
\begin{equation}\label{eq:S2bardef}
 \sibr(2) := \sitd(2) \otimes \FF_4
\end{equation} 
and we only know the generators of $\sitd(2)$ in $\sibr(2)$ modulo terms of higher Adams filtration.  Our approach to understanding $H^*(\sitd(2))$ will be to understand aspects of the cohomology of $\sibr(2)$, and then to infer results about the cohomology of $\sitd(2)$.

Our method of computing the cohomology of $\sibr(2)$, and comparing it with the cohomology of $\sitd(2)$, will be to adapt a filtration employed by Ravenel to compute the cohomology of the Morava stabilizer algebras.  This filtration will result in a pair of May-type spectral sequences, which we refer to as \emph{May-Ravenel spectral sequences}:
$$ 
\xymatrix{
\E{MR}{}_1(\sitd(2)) = H^*(E^{MR}_{0}\sitd(2)) \ar@{=>}[r] \ar[d] & H^*(\sitd(2)) \ar[d] \\ 
\E{MR}{}_1(\sibr(2)) = H^*(E^{MR}_{0}\sibr(2)) \ar@{=>}[r] & H^*(\sibr(2)) 
}$$
The $E_1$-terms $\E{MR}{}_1$ will be computed by endowing $E_0^{MR} \sitd(2)$ and $E_0^{MR}\sibr(2)$ with Adams filtrations, resulting in a pair of \emph{Adams filtration spectral sequences} (AFSS)
 $$ 
 \xymatrix{
 \E{AF}{}_1(\sitd(2)) \ar@{=>}[r] \ar[d] & H^*(E_0^{MR}\sitd(2)) \ar[d] \\ 
 \E{AF}{}_1(\sibr(2)) \ar@{=>}[r] & H^*(E_0^{MR}\sibr(2)) 
 }$$
 
The May-Ravenel $E_1$-term $\E{MR}{}_1(\sibr(2))$ is the cohomology of a certain restricted Lie algebra $\br{l}(2)$.  This cohomology may be computed by a Chevallay-Eilenberg complex, whose differentials were explicitly computed by Ravenel.  
The key observations which we employ are:
\begin{enumerate}
\item The Chevallay-Eilenberg complex for $\br{l}(2)$ is isomorphic to $\E{AF}{}_1(\sibr(2))$.
\item The differentials in the Adams filtration spectral sequence $\{\E{AF}{}_r(\sibr(2))\}$ are determined by the differentials in the Chevallay-Eilenberg complex.
\item The image of $\E{AF}{}_1(\sitd(2))$ in $\E{AF}{}_1(\sibr(2))$ can be computed precisely, since we know the generators of $\sitd(2)$ modulo terms of higher Adams filtration.  This allows us to completely compute the differentials in the Adams filtration spectral sequence $\{\E{AF}{}_r(\sitd(2))\}$. 
\end{enumerate}

Even with knowing the differentials, the combinatorics for computing the spectral sequence $\{\E{AF}{}_r(\sibr(2))\}$ is complicated.  The computation of the spectral sequence $\{\E{AF}{}_r(\sibr(2))\}$ will be facilitated by refining the Adams filtration with a lexicographical filtration. This results in a \emph{lexicographical filtration spectral sequence} (LFSS)
$$  \E{AF}{}_1(\sibr(2)) = \E{LF}{}_0(\sibr(2)) \Rightarrow \E{MR}{}_1(\sibr(2)). $$ 
We will completely compute the LFSS spectral sequence, deduce from this the AFSS for $\sibr(2)$, deduce from that the AFSS for $\sitd(2)$, and thus completely compute $\E{MR}{}_1(\sitd(2))$.  In the low dimensional range we consider for our application, there will be no possible differentials in the May-Ravenel spectral sequence
$$ \E{MR}{}_1(\sitd(2)) \Rightarrow H^*(\sitd(2)). $$

\subsection{The May-Ravenel spectral sequence}

Let $(\FF_2, \Sirav(2))$ be the Hopf algebra obtained from $(K(2)_*,\Sigma(2))$ by setting $v_2=1$. In \cite[Chapter 3]{Ravenel}, Ravenel computed 
\[H^*(\Sirav(2)) = \Ext_{\Sirav(2)}^{*}(\FF_2, \FF_2). \]
The computation for $(K(2)_*,\Sigma(2))$ and $((K_2)_*, \Sigma_2)$ can be done using similar methods and all differentials follow from Ravenel's work by reintroducing the grading. We begin by summarizing Ravenel's method, which we then apply to our cases.

In \cite[Section 4.3]{Ravenel}, Ravenel defines a filtration of Hopf algebroids on $BP_*BP/I_N$. Specializing to the case of $N = p = 2$, this induces a filtration on $(k(2)_*,\si(2))$, where
\begin{align}\label{eq:sidef}
\si(2) = \FF_2[v_2] [t_1, t_2, \ldots]/(t_k^{4}-v_2^{2^k-1}t_k).
\end{align}
There is a unique increasing multiplicative filtration (which we call the \emph{May-Ravenel filtration}) on $\si(2)$ such that
\begin{align*}
\deg_{MR}(v_2)&=0, \\
\deg_{MR}(t^{2^j}_1)&= 1, \\
\deg_{MR}(t_{2k+1}^{2^j})& = 3\cdot 2^{k-1}, \quad k > 0, \\
\deg_{MR}(t_{2k}^{2^j}) &= 2^k.
\end{align*}
Further, Ravenel \cite[4.3.24]{Ravenel} proves that this is a filtration of Hopf algebroids, so that the associated graded $E_0(\si(2))$ 
is a Hopf algebra. It is given by the exterior algebra
\[ E_0(\si(2)) \cong  \FF_2[v_2] \otimes E[t_{i,j} \: : \: 0<i, j \in \{0,1\} ] \]
where $t_{i,j}$ is the image of $t_i^{2^j}$. 

From this filtration, we get a May type spectral sequence, which we call the \emph{May-Ravenel spectral sequence}:
\[  \E{MR}{}^{s,t,f}_1(\si(2))  
\Longrightarrow H^s(\si(2))_t. \]
Here $s$ is the cohomological degree, $t$ is the internal degree, and $f$ is the May-Ravenel filtration.
The first step is to compute $ \E{MR}{}_1(\si(2))$. 

Let $E^0(\si(2))$ be the $\FF_2$-linear dual of  $E_0(\si(2))$ and $x_{i,j}$ be the dual of $t_{i,j}$. Since the $t_{i,j}$'s form a basis of the indecomposables of $E_0(\si(2))$, it follows that $x_{i,j}$ forms a basis for the restricted Lie algebra of primitives 
$$ l(2) := PE^0(\si(2)) $$
and $\E{MR}{}_1 = H^*(l(2))$.  
Applying the methods of {\cite[Remark 10]{May_lie}}, we obtain a Chevallay-Eilenberg cochain complex 
\[  C^{*,*,*}_{CE}(l(2))  := \FF_2[v_2] \otimes \FF_2[h_{i,j} \: : \: 0<i, \: 0\leq j \le 1  ]\]
for elements $h_{i,j}$ of cohomological degree $s = 1$, internal degree $t = 2^{j+1}(2^i-1)$ and with May-Ravenel filtration given by that of $t_i^{2^j}$.  
Here, $h_{i,j}$ represents the dual of the element May calls $\gamma_1(x_{i,j})$. 
The $E_1$-term of the May-Ravenel spectral sequence is the cohomology of the Chevallay-Eilenberg complex:
$$ H^{s,t,f}(C^{*,*,*}_{CE}(l(2))) = \E{MR}{}^{s,t,f}_1(\si(2)). $$

The differentials are determined by the Lie bracket and restriction of $PE^0(\si(2))$.  For $\si(2)$,  these are obtained by ``remembering the grading'' in \cite[6.3.3]{Ravenel}.
We obtain the following differentials.

\begin{thm}\label{thm:diffs}
Let 
$\chi_2 = v_2h_{2,0} + h_{2,1}$. 
The differentials in $C^{*,*,*}_{CE}(l(2))$ are determined by $d(h_{1,0})= d(h_{1,1}) = 0$ and 
\begin{align*}
d(h_{2,0}) &=    h_{1,0} h_{1,1}  & d(h_{2,1})  &= v_2 h_{1,0} h_{1,1}\\
d(h_{3,0}) &= h_{1,0}\chi_2    & d(h_{3,1}) &= v_2^2 h_{1,1}\chi_2    \\
d(h_{4,0}) &=h_{1,0}h_{3,1}+v_2^2h_{1,1}h_{3,0}+v_2\chi_2^2 & d(h_{4,1})  &=v_2^5h_{1,0}h_{3,1}+v_2^{7}h_{1,1}h_{3,0}+v_2^6\chi_2^2   \\
d(h_{i,0}) &= v_2 h_{i-2,1}^2  & d(h_{i,1})  &=  v_2^{2^{i-1}}h_{i-2,0}^2 \\
\end{align*}
where the last two identities hold for $i\geq 5$.
\end{thm}

Now, we can put the same filtration on $\Sigma_2$, and this induces a filtration on $\overline{\Sigma}_2$ which restricts to a filtration on $\sibr(2)$ (\ref{eq:S2bardef}) and $\sitd(2)$.  The corresponding associated graded Hopf algebra in the case of $\sibr(2)$ is given by
\[E^{MR}_0(\sibr(2)) \cong \FF_4[v_2] \otimes E[\td{t}_{2,1}, \tb_{3,0}, \tb_{3,1}, \tb_{4,0}, \tb_{4,1}, \cdots].\]

As before, we have a May-Ravenel spectral sequence 
\[  \E{MR}{}^{s,t,f}_1(\sibr(2))  \Longrightarrow H^{s,t}(\sibr(2)), \]
and $\E{MR}{}^{s,t,f}_1(\sibr(2)) = H^{s,t,f}(\br{l}(2))$ is the cohomology of the Chevallay-Eilenberg complex
\[ C^{*,*,*}_{CE}(\br{l}(2)) \cong  \FF_4[v_2, \td{h}_{2,1}, {h}_{3,0}, {h}_{3,1}, {h}_{4,0}, {h}_{4,1}, \ldots ]. \]


\begin{thm}\label{thm:diffsquotient} 
The differentials in the Chevallay-Eilenberg complex $C_{CE}^{*,*,*}(\br{l}(2))$ are determined by 
\[d(\td{h}_{2,1}) =d(\bh_{3,0})  =d(\bh_{3,1}) =0 \] 
and
\begin{align*}
d(\bh_{4,0}) &= v_2\td{h}_{2,1}^2 & d(\bh_{4,1})  &=v_2^6\td{h}_{2,1}^2 \\
d(\bh_{i,0}) &= v_2 \bh_{i-2,1}^2  & d(\bh_{i,1})  &=  v_2^{2^{i-1}}\bh_{i-2,0}^2 \\
\end{align*}
where the last two identities hold for $i\geq 5$.
\end{thm}

\begin{proof}
The element $\td{t_2^2}$ of (\ref{eq:tdt2}) is given by 
$$ \td{t_2^2} = v_2t_2 + t_2^2+v_2t_1^3. $$
Therefore, since we have $t_1 \equiv 0$ in $\sibr(2) \subset \br{\Sigma}_2$ (\ref{eq:barSigma_2}), we deduce that under the map 
$$ \si(2) \rightarrow \sibr(2) $$
we have   
$$ v_2t_2 + t_2^2 \mapsto \td{t_2^2}. $$
It follows that under the map of Chevallay-Eilenberg complexes 
$$ C^{*,*,*}_{CE}(l(2)) \rightarrow C^{*,*,*}_{CE}(\br{l}(2)) $$
we have
$$ \chi_2 \mapsto \td{h}_{2,1}. $$ 
The result therefore follows from Theorem~\ref{thm:diffs} (and the fact that $t_1 \equiv 0$ in $\sibr(2)$). 
\end{proof}

\subsection{The lexicographical filtration spectral sequence}

In order to compute the cohomology of the Chevallay-Eilenberg complex $C^{*,*,*}_{CE}(\br{l}(2))$, 
we place an increasing filtration on the Chevallay-Eilenberg complex by declaring that a monomial
$$ v_2^m h^{k_3}_{3,0}h^{k_4}_{4,0} \cdots \td{h}^{l_2}_{2,1} h_{3,1}^{l_3} \td{h}^{l_4}_{4,1}\cdots $$
has lexicographical filtration tridegree
$$ \mr{deg}_{LF} = (-m, l, k) $$
where 
\begin{align*}
l & = \bar{l}_5 + 2\bar{l}_6 + 2^2\bar{l}_7 + 2^3\bar{l}_8 + \cdots, \\
k & = \bar{k}_4+ 2 \bar{k}_5 + 2^2 \bar{k}_6 + 2^3 \bar{k}_7 + \cdots, 
\end{align*}
$\bar{n} \in \{0,1\}$ is $n$ mod $2$, and 
$$ \td{h}_{4,1} := h_{4,1} + v_2^5h_{4,0}. $$
We order these tridegrees via left lexicographical order.  That is to say, 
$$ (m,l,k) < (m',l',k') $$
if $m < m'$, or $m = m'$ and $l < l'$, or $m=m'$ and $l = l'$ and $k < k'$.

Note that the value $m$ above is the negative of the Adams filtration (defined by declaring the Adams filtration of $v_2$ is $1$, and all other generators have Adams filtration $0$), so lexicographical filtration is a refinement of Adams filtration.
The differentials of Theorem~\ref{thm:diffsquotient}  decrease lexicographical ordering, resulting in an increasing filtration on $C^{*,*,*}_{CE}(\br{l}(2))$ and a transfinite lexicographical filtration spectral sequence (LFSS)
$$ C^{s,t,f}_{CE}(\br{l}(2))_{m,k,l} = \E{LF}{}^{s,t,f,(m,l,k)}_{1,0,0} \Rightarrow \E{MR}{}^{s,t,f}_1(\sibr(2)). $$
Transfinite spectral sequences were introduced by Hu in \cite{Hu}.  Hu's indexing uses ordinals, but to simplify matters we are repackaging the relevant ordinals as lexicographical tridegrees.  In this way, we can explain the transfinite nature of the spectral sequence in our particular case.

Namely, the lexicographical filtration spectral sequence has terms
$$ \E{LF}{}^{s,t,f,(m,l,k)}_{r,r',r''} $$
with differentials
$$ d_{r,r',r''}: \E{LF}{}^{s,t,f,(m,l,k)}_{r,r',r''} \to \E{LF}{}^{s+1,t,f,(m-r,l-r',k-r'')}_{r,r',r''} $$
and
$$ \E{LF}{}^{s,t,f,(m,l,k)}_{r,r',r''+1} \cong H^{s,t,f,(m,l,k)}(\E{LF}{}_{r,r',r''}, d_{r,r',r''}). $$
Because this spectral sequence is finitely generated in each multi-degree, convergence can be explained as follows.  For $r'' \gg 0$, we have
$$ \E{LF}{}^{s,t,f,(m,l,k)}_{r,r',r''} = \E{LF}{}^{s,t,f,(m,l,k)}_{r, r'+1, 0}. $$
and for $r' \gg 0$ we have 
$$ \E{LF}{}^{s,t,f,(m,l,k)}_{r,r',0} = \E{LF}{}^{s,t,f,(m,l,k)}_{r+1, 0, 0}. $$
There is a lexicographically indexed increasing filtration $\{F^{s,t,f}_{m,l,k}\}$ on $\E{MR}{}_1^{s,t,f}$ such that for $r \gg 0$,
$$ \E{LF}{}^{s,t,f,(m,l,k)}_{r,0,0} \cong F^{s,t,f}_{m,l,k}/F^{s,t,f}_{m,l,k-1}. $$
Because the lexicographical filtration is a multiplicative filtration on a differential graded algebra, the lexicographical filtration spectral sequence is a spectral sequence of algebras.  By Theorem~\ref{thm:diffsquotient}, 
the elements
$$ \td{h}_{2,1}, \: h_{3,0}, \: h_{3,1}, \: \mr{and} \: \td{h}_{4,1} $$
are permanent cycles in the lexicographical filtration spectral sequence, and we have 
\begin{equation}\label{eq:LFSSdiffs}
\begin{split}
d_{1,0,1}(h_{4,0}) & = v_2 \td{h}_{2,1}^2,  \\
d_{1,0,2^{i-4}}(h_{i,0}) & = v_2 h^2_{i-2,1}, \\
d_{2^{i-1},2^{i-5},0}(h_{i,1}) & = v_2^{2^{i-1}} h^2_{i-2,0}.
\end{split}
\end{equation}
We note that the elements $h^2_{i,0}$ and $h^2_{i,1}$ are permanent cycles, because they correspond to cocyles in the Chevallay-Eilenberg complex.

We now run the lexicographical filtration spectral sequence.  We will run the differentials in two rounds.  The first round (Lemma~\ref{lem:AFE1})) will consist of those differentials of the form $d_{1,r',r''}$ which change Adams filtration by $1$.  The second round (Theorem~\ref{thm:MRE1bar}) will consist of those differentials of the form $d_{r,r',r''}$ with $r > 1$ which change Adams filtration by a quantity greater than 1.

\begin{lem}\label{lem:AFE1}
The $E_{2,0,0}$ page of the lexicographical filtration spectral sequence obtained by running all differentials of the form $d_{1,r',r''}$ has a basis given by:
\begin{align*}
\mr{(I)} \quad & v_2^m h^{k_3}_{3,0}h^{2k_4}_{4,0} h^{2k_5}_{5,0}  \cdots \td{h}^{\epsilon_2}_{2,1} h_{3,1}^{\epsilon_3} \cdots, \qquad \qquad \qquad
m, k_j \ge 0; \: \epsilon_j \in \{0,1\}, \\
\mr{(II)} \quad & h^{k_3}_{3,0}h^{2k_4}_{4,0} \cdots h^{2k_{i+2}}_{i+2,0} h^{k_{i+3}}_{i+3,0}\cdots \td{h}^{\epsilon_2}_{2,1} \cdots h_{i-1,1}^{\epsilon_{i-1}}h_{i,1}^{l_i+2} h^{l_{i+1}}_{i+1,1}\cdots,
\shortintertext{\begin{flushright}
$i \ge 2; \: k_j, l_j \ge 0; \: \epsilon_j \in \{0,1\}.$
\end{flushright}}
\end{align*}
\end{lem}

\begin{proof}
The strategy will be to first observe that the monomials of type (I) and (II) are $d_{1,r',r''}$ cycles for all $r'$ and $r''$.  We will then show that all of the other monomials are either the source or target of a non-trivial differential in the lexicographical filtration spectral sequence.

To show that a monomial $x$ of the form (I) or (II) of Adams filtration $m$ is a $d_{1,r',r''}$-cycles, it suffices to show that the element
$$ x \in C^{*,*,*}_{CE}(\br{l}(2)) $$
in the Chevallay-Eilenberg complex can be completed to an element
$$ x + y \in C^{*,*,*}_{CE}(\br{l}(2)) $$
where
\begin{enumerate}
\item $y$ has lower lexicographical filtration than $x$,  and
\item the Chevallay-Eilenberg differential
$$ d^{CE}(x+y)
$$
has Adams filtration greater than $m+1$. 
\end{enumerate}
In the case of the monomials of type (I), this is trivially true - the elements $x$ of Adams filtration $m$ already satisfy the property that $d^{CE}(x)$ has Adams filtration greater than $m+1$.  In the case of terms of type (II), one can check that the sum (with $\bar{\epsilon}_j \in \{0,1\}$)
\begin{multline*}
x(k_3,2k_4, \ldots 2k_{i+2}, 2k_{i+3}+\bar{\epsilon}_{i+3}, \ldots; \epsilon_2, \ldots \epsilon_{i-1}, l_{i}+2, l_{i+1}, \ldots ) := 
\\
h^{k_3}_{3,0}h^{2k_4}_{4,0} \cdots h^{2k_{i+2}}_{i+2,0} h^{2k_{i+3}+\bar{\epsilon}_{i+3}}_{i+3,0}h^{2k_{i+4}+\bar{\epsilon}_{i+4}}_{i+4,0}h^{2k_{i+5}+\bar{\epsilon}_{i+5}}_{i+5,0}\cdots \td{h}^{\epsilon_2}_{2,1} \cdots h_{i-1,1}^{\epsilon_{i-1}}h_{i,1}^{l_i+2} h^{l_{i+1}}_{i+1,1}\cdots 
\\
+ \bar{\epsilon}_{i+3} h^{k_3}_{3,0}h^{2k_4}_{4,0} \cdots h^{2k_{i+2}+1}_{i+2,0} h^{2k_{i+3}}_{i+3,0}h^{2k_{i+4}+\bar{\epsilon}_{i+4}}_{i+4,0}h^{2k_{i+5}+\bar{\epsilon}_{i+5}}_{i+5,0}\cdots \td{h}^{\epsilon_2}_{2,1} \cdots h_{i-1,1}^{\epsilon_{i-1}}h_{i,1}^{l_i} h^{l_{i+1}+2}_{i+1,1} h^{l_{i+2}}_{i+2,1}\cdots 
\\
+ \bar{\epsilon}_{i+4} h^{k_3}_{3,0}h^{2k_4}_{4,0} \cdots h^{2k_{i+2}+1}_{i+2,0} h^{2k_{i+3}+\bar{\epsilon}_{i+3}}_{i+3,0}h^{2k_{i+4}}_{i+4,0}h^{2k_{i+5}+\bar{\epsilon}_{i+5}}_{i+5,0}\cdots \td{h}^{\epsilon_2}_{2,1} \cdots h_{i-1,1}^{\epsilon_{i-1}}h_{i,1}^{l_i} h^{l_{i+1}}_{i+1,1} h^{l_{i+2}+2}_{i+2,1}\cdots 
\\
+ \bar{\epsilon}_{i+5} h^{k_3}_{3,0}h^{2k_4}_{4,0} \cdots h^{2k_{i+2}+1}_{i+2,0} h^{2k_{i+3}+\bar{\epsilon}_{i+3}}_{i+3,0}h^{2k_{i+4}+\bar{\epsilon}_{i+4}}_{i+4,0}h^{2k_{i+5}}_{i+5,0}\cdots \td{h}^{\epsilon_2}_{2,1} \cdots h_{i-1,1}^{\epsilon_{i-1}}h_{i,1}^{l_i} \cdots h^{l_{i+2}}_{i+2,1} h^{l_{i+3}+2}_{i+3,1}\cdots 
\\
+ \cdots
\end{multline*}
satisfies (1) and (2) above.

We now compute the differentials of the form $d_{1,r',r''}$ in the lexicographical filtration spectral sequence, using (\ref{eq:LFSSdiffs}).
The $d_{1,0,1}$ differentials are
$$ d_{1,0,1}(v_2^m h^{k_3}_{3,0}h^{2k_4+1}_{4,0} h^{k_5}_{5,0} \cdots \td{h}^{l_2}_{2,1} h_{3,1}^{l_3} \td{h}^{l_4}_{4,1}\cdots) = 
v_2^{m+1} h^{k_3}_{3,0}h^{2k_4}_{4,0} h^{k_5}_{5,0} \cdots \td{h}^{l_2+2}_{2,1} h_{3,1}^{l_3} \td{h}^{l_4}_{4,1}\cdots
. $$
The remaining monomials
\begin{align*}
& h^{k_3}_{3,0}h^{2k_4}_{4,0} h^{k_5}_{5,0} \cdots \td{h}^{l_2+2}_{2,1} h_{3,1}^{l_3} \td{h}^{l_4}_{4,1}\cdots, \\
& v_2^m h^{k_3}_{3,0}h^{2k_4}_{4,0} h^{k_5}_{5,0} \cdots \td{h}^{\epsilon_2}_{2,1} h_{3,1}^{l_3} \td{h}^{l_4}_{4,1}\cdots,
\end{align*}
where $\epsilon_2 \in \{0,1\}$, are $d_{1,0,1}$-cycles, and therefore constitute a basis for the $E_{1,0,2}$ page.

The $d_{1,0,2}$ differentials are of the form
$$ d_{1,0,2}(v_2^m h^{k_3}_{3,0}h^{2k_4}_{4,0} h^{2k_5+1}_{5,0}h^{k_6}_{6,0} \cdots \td{h}^{\epsilon_2}_{2,1} h_{3,1}^{l_3} \td{h}^{l_4}_{4,1}\cdots) = 
v_2^{m+1} h^{k_3}_{3,0}h^{2k_4}_{4,0} h^{2k_5}_{5,0} h^{k_6}_{6,0}\cdots \td{h}^{\epsilon_2}_{2,1} h_{3,1}^{l_3+2} \td{h}^{l_4}_{4,1}\cdots
. $$
The remaining monomials  
\begin{align}
& h^{k_3}_{3,0}h^{2k_4}_{4,0} h^{k_5}_{5,0} \cdots \td{h}^{l_2+2}_{2,1} h_{3,1}^{l_3} \td{h}^{l_4}_{4,1}\cdots, \label{eq:E102.1} \\ 
& h^{k_3}_{3,0}h^{2k_4}_{4,0} h^{2k_5}_{5,0} h^{k_6}_{6,0}\cdots \td{h}^{\epsilon_2}_{2,1} h_{3,1}^{l_3+2} \td{h}^{l_4}_{4,1}\cdots, \label{eq:E102.2} \\ 
& v_2^m h^{k_3}_{3,0}h^{2k_4}_{4,0} h^{2k_5}_{5,0} h^{k_6}_{6,0} \cdots \td{h}^{\epsilon_2}_{2,1} h_{3,1}^{\epsilon_3} \td{h}^{l_4}_{4,1}\cdots, \label{eq:E102.3} 
\end{align}
with $\epsilon_3 \in \{0,1\}$, are $d_{1,0,2}$ and $d_{1,0,3}$ cycles (in the case of (\ref{eq:E102.1}), this is because it is a cocycle of type (II), and in the cases of (\ref{eq:E102.2}) and (\ref{eq:E102.3}), this follows from (\ref{eq:LFSSdiffs})).  Thus these form a basis of the $E_{1,0,4}$-page.
 
Repeating this process, the result follows.
\end{proof}

We now run the differentials which change Adams filtration by more than $1$.  The idea is that these differentials are non-trivial only on terms of type (I), and these differentials hit terms of type (I).  Terms of type (II) are going to be permanent cycles in the lexicographical filtration spectral sequence.

\begin{thm}\label{thm:MRE1bar}
The May-Ravenel $E_1$-term $\E{MR}{}_1(\sibr(2))$ has a basis over $\FF_4$ whose representatives in the lexicographical filtration spectral sequence are given by:
\begin{align*}
\mr{(I')} & \quad v_2^m h^{\bar{\epsilon}_3}_{3,0}\td{h}_{2,1}^{\epsilon_2} h^{\epsilon_3}_{3,1} \td{h}^{\epsilon_4}_{4,1}, 
\shortintertext{\begin{flushright}
$m \ge 0; \: \epsilon_j, \bar{\epsilon}_j \in \{0,1\},$
\end{flushright}}
\mr{(I'')} & \quad v_2^{<2^{i+1}}h^{\bar{\epsilon}_3}_{3,0} h^{2(k_i+1)}_{i,0} h^{2k_{i+1}}_{i+1,0} h^{2k_{i+2}}_{i+2,0} \cdots \td{h}_{2,1}^{\epsilon_2} h^{\epsilon_3}_{3,1} \td{h}^{\epsilon_4}_{4,1} h^{\epsilon_{i+3}}_{i+3,1} \cdots.  
\shortintertext{\begin{flushright}
$i \ge 3; \: k_j \ge 0; \: \epsilon_j, \bar{\epsilon}_j \in \{0,1\},$
\end{flushright}}
\mr{(II)} & \quad h^{k_3}_{3,0}h^{2k_4}_{4,0} \cdots h^{2k_{i+2}}_{i+2,0} h^{k_{i+3}}_{i+3,0}\cdots \td{h}^{\epsilon_2}_{2,1} \cdots h_{i-1,1}^{\epsilon_{i-1}}h_{i,1}^{l_i+2} h^{l_{i+1}}_{i+1,1}\cdots.
\shortintertext{\begin{flushright}
$i \ge 2; \: k_j, l_j \ge 0; \: \epsilon_j \in \{0,1\}$.
\end{flushright}}
\end{align*}
For the monomials of type (I''), the notation $v_2^{<n}x$ means monomials of the form $v_2^ix$ for $i < n$.
\end{thm} 

\begin{proof}
We proceed using the strategy of the proof of Lemma~\ref{lem:AFE1}.  We start by showing that the monomials of types (I'), (I''), and (II) are permanent cycles by showing they complete to cocycles in the Chevallay-Eilenberg complex.

The terms $\mr{(I')}$ are simply cocycles.  The terms $\mr{(I'')}$ complete to cocycles given by
\begin{multline*}
h^{\bar{\epsilon}_3}_{3,0} h^{2(k_i+1)}_{i,0} h^{2k_{i+1}}_{i+1,0} \cdots \td{h}_{2,1}^{\epsilon_2} h^{\epsilon_3}_{3,1} \td{h}^{\epsilon_4}_{4,1} h^{\epsilon_{i+3}}_{i+3,1} \cdots \\
+
\epsilon_{i+3} v_2^{2^{i+2}-2^{i+1}}h^{\bar{\epsilon}_3}_{3,0} h^{2k_i}_{i,0} h^{2(k_{i+1}+1)}_{i+1,0}h^{2k_{i+2}}_{i+2,0} \cdots \td{h}_{2,1}^{\epsilon_2} h^{\epsilon_3}_{3,1} \td{h}^{\epsilon_4}_{4,1} h_{i+2,1}h^{\epsilon_{i+4}}_{i+4,1} \cdots \\
+
\epsilon_{i+4} v_2^{2^{i+3}-2^{i+1}}h^{\bar{\epsilon}_3}_{3,0} h^{2k_i}_{i,0} h^{2k_{i+1}}_{i+1,0}h^{2(k_{i+2}+1)}_{i+2,0} \cdots \td{h}_{2,1}^{\epsilon_2} h^{\epsilon_3}_{3,1} \td{h}^{\epsilon_4}_{4,1} h_{i+2,1}h^{\epsilon_{i+3}}_{i+3,1} h^{\epsilon_{i+5}}_{i+5,1} \cdots
\\ + \cdots 
\end{multline*}
For the terms of type (II), we observe that the Cartan-Eilenberg differential $d^{CE}$ is given on the terms $x(-)$ appearing in the proof of Lemma~\ref{lem:AFE1} by
\begin{multline*}
d^{CE}x(k_3,2k_4, \ldots 2k_{i+2}, 2k_{i+3}+\bar{\epsilon}_{i+3}, \ldots; \epsilon_2, \ldots \epsilon_{i-1}, l_{i}+2, l_{i+1}, \ldots ) =
\\
\epsilon_5 v_2^{2^4} x(k_3+2,2k_4, \ldots 2k_{i+2}, 2k_{i+3}+\bar{\epsilon}_{i+3}, \ldots; \epsilon_2, \ldots, \epsilon_4, 0, \epsilon_{6}, \ldots,  \epsilon_{i-1}, l_{i}+2, l_{i+1}, \ldots )
\\
+ 
\epsilon_6 v_2^{2^5} x(k_3,2(k_4+1), \ldots 2k_{i+2}, 2k_{i+3}+\bar{\epsilon}_{i+3}, \ldots; \epsilon_2, \ldots, \epsilon_4, \epsilon_5, 0, \epsilon_7, \ldots,  \epsilon_{i-1}, l_{i}+2, l_{i+1}, \ldots )
\\
+ \cdots
\\
+ \bar{l}_{i} v_2^{2^{i-1}}x(k_3,2k_4, \ldots, 2(k_{i-2}+1), \ldots, 2k_{i+2}, 2k_{i+3}+\bar{\epsilon}_{i+3}, \ldots; \epsilon_2, \ldots \epsilon_{i-1}, l_{i}-1+2, l_{i+1}, \ldots )
\\
+ \bar{l}_{i+1} v_2^{2^{i}}x(k_3,2k_4, \ldots, 2(k_{i-1}+1), \ldots, 2k_{i+2}, 2k_{i+3}+\bar{\epsilon}_{i+3}, \ldots; \epsilon_2, \ldots \epsilon_{i-1}, l_{i}+2, l_{i+1}-1, \ldots )
\\
+ \cdots.
\end{multline*}
However, also note that 
\begin{multline*}
d^{CE}(h^{k_3}_{3,0}h^{2k_4}_{4,0} \cdots h^{2k_{i+2}+1}_{i+2,0} h^{2k_{i+3}+\bar{\epsilon}_{i+3}}_{i+3,0}h^{2k_{i+4}+\bar{\epsilon}_{i+4}}_{i+4,0}\cdots \td{h}^{\epsilon_2}_{2,1} \cdots h_{i-1,1}^{\epsilon_{i-1}}h_{i,1}^{l_i} h^{l_{i+1}}_{i+1,1}\cdots 
\\ 
= v_2 x(k_3,2k_4, \ldots, 2k_{i+2}, 2k_{i+3}+\bar{\epsilon}_{i+3}, \ldots; \epsilon_2, \ldots \epsilon_{i-1}, l_{i}+2, l_{i+1}, \ldots ).
\end{multline*}
We therefore find that the terms of type (II) complete to the following cocycles:
\begin{multline*}
x(k_3,2k_4, \ldots 2k_{i+2}, 2k_{i+3}+\bar{\epsilon}_{i+3}, \ldots; \epsilon_2, \ldots \epsilon_{i-1}, l_{i}+2, l_{i+1}, \ldots )
\\
+ \epsilon_5 v_2^{2^4-1} h_{3,0}^{k_3+2}h_{4,0}^{2k_4} \cdots h_{i+2,0}^{2k_{i+2}+1} h_{i+3,0}^{2k_{i+3}+\bar{\epsilon}_{i+3}} \cdots \td{h}_{2,1}^{\epsilon_2} \cdots \td{h}_{4,1}^{\epsilon_4}h_{6,1}^{\epsilon_{6}} \cdots h_{i-1,1}^{\epsilon_{i-1}}h_{i,1}^{l_{i}} \cdots 
\\
+ 
\epsilon_6 v_2^{2^5-1} h_{3,0}^{k_3}h_{4,0}^{2(k_4+1)} \cdots h_{i+2,0}^{2k_{i+2}+1} h_{i+3,0}^{2k_{i+3}+\bar{\epsilon}_{i+3}} \cdots \td{h}_{2,1}^{\epsilon_2} \cdots \td{h}_{4,1}^{\epsilon_4}h_{5,1}^{\epsilon_{5}}h_{7,1}^{\epsilon_{7}} \cdots h_{i-1,1}^{\epsilon_{i-1}}h_{i,1}^{l_{i}} \cdots 
\\
+ \cdots
\\
+ \bar{l}_{i} v_2^{2^{i-1}-1}h_{3,0}^{k_3}h_{4,0}^{2k_4} \cdots h_{i-2,0}^{2(k_{i-2}+1)} \cdots h_{i+2,0}^{2k_{i+2}+1} h_{i+3,0}^{2k_{i+3}+\bar{\epsilon}_{i+3}} \cdots \td{h}_{2,1}^{\epsilon_2} \cdots \td{h}_{4,1}^{\epsilon_4}h_{5,1}^{\epsilon_{5}} \cdots h_{i-1,1}^{\epsilon_{i-1}}h_{i,1}^{l_{i}-1} \cdots
\\
+ \bar{l}_{i+1} v_2^{2^{i}-1}h_{3,0}^{k_3}h_{4,0}^{2k_4} \cdots h_{i-2,0}^{2(k_{i-1}+1)} \cdots h_{i+2,0}^{2k_{i+2}+1} h_{i+3,0}^{2k_{i+3}+\bar{\epsilon}_{i+3}} \cdots \td{h}_{2,1}^{\epsilon_2} \cdots \td{h}_{4,1}^{\epsilon_4}h_{5,1}^{\epsilon_{5}} \cdots h_{i-1,1}^{\epsilon_{i-1}} h_{i,1}^{l_{i}}h_{i+1,1}^{l_{i+1}-1} \cdots
\\
+ \cdots 
\\
+ \bar{l}_{i+4} v_2^{2^{i+3}-1}h_{3,0}^{k_3}h_{4,0}^{2k_4} \cdots h_{i+2,0}^{2(k_{i+2}+1)+1} h_{i+3,0}^{2k_{i+3}+\bar{\epsilon}_{i+3}} \cdots \td{h}_{2,1}^{\epsilon_2} \cdots \td{h}_{4,1}^{\epsilon_4}h_{5,1}^{\epsilon_{5}} \cdots h_{i-1,1}^{\epsilon_{i-1}}h_{i,1}^{l_{i}} \cdots h_{i+4,1}^{l_{i+4}-1} \cdots
\\
+ \bar{l}_{i+5} v_2^{2^{i+4}-1}h_{3,0}^{k_3}h_{4,0}^{2k_4} \cdots h_{i+2,0}^{2k_{i+2}+1} h_{i+3,0}^{2(k_{i+3}+1)+\bar{\epsilon}_{i+3}} \cdots \td{h}_{2,1}^{\epsilon_2} \cdots \td{h}_{4,1}^{\epsilon_4}h_{6,1}^{\epsilon_{6}} \cdots h_{i-1,1}^{\epsilon_{i-1}}h_{i,1}^{l_{i}} \cdots h_{i+5,1}^{l_{i+5}-1} \cdots
\\
+ \cdots.
\end{multline*}

We now will proceed by showing that the rest of the monomials of type (I) are either sources or targets of non-trivial $d_{r,r', r''}$ differentials with $r \ge 2$. 

The first round of differentials in the LFSS will be of the form
\begin{multline*}
d_{16,1,0}(v_2^m h^{\bar{\epsilon}_{3}}_{3,0}h^{2k_3}_{3,0}h^{2k_4}_{4,0} h^{2k_5}_{5,0}  \cdots \td{h}^{\epsilon_2}_{2,1} h_{3,1}^{\epsilon_3} \td{h}^{\epsilon_4}_{4,1}h_{5,1}h^{\epsilon_6}_{6,1} \cdots)
\\
= v_2^{m+16} h^{\bar{\epsilon}_3}_{3,0}h^{2(k_3+1)}_{3,0}h^{2k_4}_{4,0} h^{2k_5}_{5,0}  \cdots \td{h}^{\epsilon_2}_{2,1} h_{3,1}^{\epsilon_3} \td{h}^{\epsilon_4}_{4,1}h^{\epsilon_6}_{6,1} \cdots
\end{multline*}
with $m, k_j \in \mb{N}$ and $\epsilon_j, \bar{\epsilon}_j \in \{0,1\}$.
Of the terms of type (I), what remains are terms of the forms
\begin{gather*}
v_2^m h^{\bar{\epsilon}_3}_{3,0}h^{2k_4}_{4,0} h^{2k_5}_{5,0} \cdots \td{h}_{2,1}^{\epsilon_2} h^{\epsilon_3}_{3,1} \td{h}^{\epsilon_4}_{4,1} h^{\epsilon_6}_{6,1} \cdots, 
\\
v_2^{<16}h^{\bar{\epsilon}_3}_{3,0} h^{2(k_3+1)}_{3,0} h^{2k_4}_{4,0} \cdots \td{h}_{2,1}^{\epsilon_2} h^{\epsilon_3}_{3,1} \td{h}^{\epsilon_4}_{4,1} h^{\epsilon_6}_{6,1} \cdots.  
\end{gather*}
These are seen to persist to the $E_{32,2,0}$-page by $(\ref{eq:LFSSdiffs})$.
The next round of differentials will be of the form
\begin{multline*}
d_{32,2,0}(v_2^m h^{\bar{\epsilon}_{3}}_{3,0}h^{2k_4}_{4,0} h^{2k_5}_{5,0}  \cdots \td{h}^{\epsilon_2}_{2,1} h_{3,1}^{\epsilon_3} \td{h}^{\epsilon_4}_{4,1}h_{6,1}h^{\epsilon_7}_{7,1} \cdots)
\\
= v_2^{m+32} h^{\bar{\epsilon}_3}_{3,0}h^{2(k_4+1)}_{4,0} h^{2k_5}_{5,0}  \cdots \td{h}^{\epsilon_2}_{2,1} h_{3,1}^{\epsilon_3} \td{h}^{\epsilon_4}_{4,1}h^{\epsilon_7}_{7,1} \cdots.
\end{multline*}
Of the terms of type (I), what remain are terms of the forms
\begin{gather}
v_2^m h^{\bar{\epsilon}_3}_{3,0}h^{2k_5}_{5,0} \cdots \td{h}_{2,1}^{\epsilon_2} h^{\epsilon_3}_{3,1} \td{h}^{\epsilon_4}_{4,1} h^{\epsilon_7}_{7,1} \cdots, 
\label{eq:E3220.1}\\
v_2^{<16}h^{\bar{\epsilon}_3}_{3,0} h^{2(k_3+1)}_{3,0} h^{2k_4}_{4,0} \cdots \td{h}_{2,1}^{\epsilon_2} h^{\epsilon_3}_{3,1} \td{h}^{\epsilon_4}_{4,1} h^{\epsilon_6}_{6,1} \cdots,  
\label{eq:E3220.2}\\
v_2^{<32}h^{\bar{\epsilon}_3}_{3,0} h^{2(k_4+1)}_{4,0} h^{2k_5}_{5,0} \cdots \td{h}_{2,1}^{\epsilon_2} h^{\epsilon_3}_{3,1} \td{h}^{\epsilon_4}_{4,1} h^{\epsilon_7}_{7,1} \cdots.
\label{eq:E3220.3}  
\end{gather}

The terms (\ref{eq:E3220.2}) are cocycles of type (I''), and the terms (\ref{eq:E3220.1}), (\ref{eq:E3220.3}) persist to $E_{64,4,0}$ by (\ref{eq:LFSSdiffs}).

Continuing in this manner, we get all of the differentials in the LFSS.
\end{proof}

\subsection{The Adams filtration spectral sequence}

Ideally we would like to reproduce the analysis of the previous section by replacing the Hopf algebra $\sibr(2)$ with $\sitd(2)$.  However, we do not have an analog of Theorem~\ref{thm:diffsquotient} for $\sitd(2)$.  This would be a prerequisite to forming a LFSS, as we need to know the May-Ravenel $d_0$-differentials decrease lexicographical filtration.  We instead will work with the coarser filtration given by Adams filtration.  The advantage of Adams filtration is that we know differentials preserve Adams filtration for topological reasons. 

Endow $\sibr(2)$ and its subalgebra $\sitd(2)$ with an increasing ``Adams filtration,'' by declaring $AF(v_2) = 1$, and giving all other generators ``Adams filtration'' $0$.  Note that in the case of $\sitd(2) = k(2)_*\tmf/(\text{$v_2$-torsion})$, this agrees with the filtration coming from the ASS for $k(2) \wedge \tmf$.
Therefore, the differentials in the cobar complex for $\sitd(2)$ respect Adams filtration because, by Theorem~\ref{thm:goodcomplex}, they come from maps of spectra (the connecting maps in the $\tmf$-Adams resolution for $Z$).  Since $\sibr(2) \cong \sitd(2) \otimes \FF_4$, the same is true for the cobar complex of $\sibr(2)$.

The algebra generators of
$$ E_0^{AF}\sibr(2) = \FF_4[v_2, \td{t^2_2}, \tb_3, \tb_{4}, \ldots]/(\td{t_2^2} = 0, \tb_k^4 = 0) $$
are seen to be primitive (see, for example,  \cite[Prop.~B.5.15]{Ravenelorange}).  Furthermore, Theorem~\ref{thm:goodcomplex} implies that $E_0^{AF}\sitd(2)$ is the primitively generated $k(2)_*$-subalgebra
$$ \FF_2[v_2, \td{t^2_2}, \tb_3, \tb_{4}, \ldots]/(\td{t_2^2} = 0, \tb_k^4 = 0) \subset E_0^{AF}\sibr(2). $$
Since there is an isomorphism of cochain complexes
$$ \mc{C}^{*,*,*}_{alg}(Z) \cong C^*_{E_0^{AF}\sitd(2)} (k(2)_*), $$
we immediately deduce the following important algebraic consequence.

\begin{thm}\label{thm:HCalg}
The cohomology of the algebraic good complex for $Z$ is given by
$$ H^{*,*,*}( \mc{C}_{alg} )\cong \FF_2[v_2, \td{h}_{2,1}, h_{i,j}]_{\genfrac{}{}{0pt}{}{i \ge 3}{j = 0,1}}. $$
\end{thm}

We may likewise endow $E_0^{MR}\sibr(2)$ and $E^{MR}_0\sitd(2)$ with Adams filtration. Then $E_0^{AF}E_0^{MR}\sibr(2)$ is given by 
$$
\FF_4[v_2]\otimes E[\td{t}_{2,1}, \br{t}_{i,j}]_{\genfrac{}{}{0pt}{}{i \ge 3}{j = 0,1}} 
$$
with $\td{t}_{2,1}$, $\br{t}_{i,j}$ primitive,
and $E_0^{AF}E_0^{MR}\sitd(2)$ is given by the subalgebra
$$ \FF_2[v_2]\otimes E[\td{t}_{2,1}, \br{t}_{i,j}]_{\genfrac{}{}{0pt}{}{i \ge 3}{j = 0,1}}. $$
This results in a pair of 
\emph{Adams filtration} spectral sequences
 $$ 
 \xymatrix{
 \E{AF}{}_1(\sitd(2)) \ar@{=>}[r] \ar[d] & H^*(E_0^{MR}\sitd(2)) \ar[d] \\ 
 \E{AF}{}_1(\sibr(2)) \ar@{=>}[r] & H^*(E_0^{MR}\sibr(2)) 
 }$$
with
\begin{align*}
\E{AF}{}_1(\sitd(2)) & \cong \FF_2[v_2, \td{h}_{2,1}, h_{i,j}]_{\genfrac{}{}{0pt}{}{i \ge 3}{j = 0,1}}, \\
\E{AF}{}_1(\sibr(2)) & \cong \FF_4[v_2, \td{h}_{2,1}, h_{i,j}]_{\genfrac{}{}{0pt}{}{i \ge 3}{j = 0,1}}.
\end{align*}

We will now compute the Adams filtration spectral sequence $\E{AF}{}_r(\sibr(2))$ by relating it to the LFSS.  

The Chevallay-Eilenberg complex $C^*_{CE}(\br{l}(2))$ is a quotient of the cobar complex for $E^{MR}_0\sibr(2)$
$$ C^*_{E^{MR}_0\sibr(2)} \twoheadrightarrow C^*_{CE}(\br{l}(2)). $$
By endowing $C^*_{CE}(\br{l}(2))$ with an Adams filtration, we get an associated spectral sequence $\E{AF}{}_r(\br{l}(2))$ and a map of spectral sequences
 $$ 
 \xymatrix{
 \E{AF}{}_1(\sibr(2)) \ar@{=>}[r] \ar[d] & H^*(E_0^{MR}\sibr(2)) \ar@{=}[d] \\ 
 \E{AF}{}_1(\br{l}(2)) \ar@{=>}[r] & H^*(E_0^{MR}\sibr(2)) 
 }$$
From Theorem~\ref{thm:diffsquotient} we see that all differentials in $C^*_{CE}(\br{l}(2))$ increase Adams filtration, and thus
\begin{align*}
\E{AF}{}_1(\br{l}(2)) & = H^*(E_0^{AF}C^*_{CE}(\br{l}(2))) \\
& = E_0^{AF}C^*_{CE}(\br{l}(2)) \\
& = \FF_4[v_2, \td{h}_{2,1}, h_{i,j}]_{\genfrac{}{}{0pt}{}{i \ge 3}{j = 0,1}}.
\end{align*}

We deduce the following.

\begin{prop}
The map
$$ \E{AF}{}_1(\sibr(2)) \rightarrow \E{AF}{}_1(\br{l}(2)) $$
is an isomorphism, and thus there is an isomorphism of spectral sequences
$$ \{ \E{AF}{}_r(\sibr(2))\} \cong \{\E{AF}{}_r(\br{l}(2))\}. $$
\end{prop}

Since lexicographic filtration is a refinement of Adams filtration, the differentials in the AFSS $\E{AF}{}_r(\br{l}(2))$ are those differentials in the LFSS which change Adams filtration by $r$.  We therefore have
$$ \E{AF}{}_r(\sibr(2)) = \E{LF}{}_{r,0,0}, $$
and for every differential
$$ d^{LF}_{r,r',r''}(x) = y $$ in the LFSS we have a corresponding differential
$$ d^{AF}_r(x) = y $$
in the AFSS.
Therefore, as we have determined the LFSS, we have implicitly determined the AFSS for $\br{l}(2)$, and therefore the AFSS for $\sibr(2)$.  We deduce that the AFSS for $\sitd(2)$ is obtained by restricting the differentials from the AFSS for $\sibr(2)$.   Therefore, from Theorem~\ref{thm:MRE1bar}  we deduce:

\begin{thm}\label{thm:MRE1}
The May-Ravenel $E_1$-term $\E{MR}{}_1(\sitd(2))$ has a basis over $\FF_2$ whose representatives in the Adams filtration spectral sequence have leading terms (with respect to lexicographical filtration) given by:
\begin{align*}
\mr{(I')} & \quad v_2^m h^{\bar{\epsilon}_3}_{3,0}\td{h}_{2,1}^{\epsilon_2} h^{\epsilon_3}_{3,1} \td{h}^{\epsilon_4}_{4,1}, 
\shortintertext{\begin{flushright}
$m \ge 0; \: \epsilon_j, \bar{\epsilon}_j \in \{0,1\},$
\end{flushright}}
\mr{(I'')} & \quad v_2^{<2^{i+1}}h^{\bar{\epsilon}_3}_{3,0} h^{2(k_i+1)}_{i,0} h^{2k_{i+1}}_{i+1,0} h^{2k_{i+2}}_{i+2,0} \cdots \td{h}_{2,1}^{\epsilon_2} h^{\epsilon_3}_{3,1} \td{h}^{\epsilon_4}_{4,1} h^{\epsilon_{i+3}}_{i+3,1} \cdots.  
\shortintertext{\begin{flushright}
$i \ge 3; \: k_j \ge 0; \: \epsilon_j, \bar{\epsilon}_j \in \{0,1\},$
\end{flushright}}
\mr{(II)} & \quad h^{k_3}_{3,0}h^{2k_4}_{4,0} \cdots h^{2k_{i+2}}_{i+2,0} h^{k_{i+3}}_{i+3,0}\cdots \td{h}^{\epsilon_2}_{2,1} \cdots h_{i-1,1}^{\epsilon_{i-1}}h_{i,1}^{l_i+2} h^{l_{i+1}}_{i+1,1}\cdots.
\shortintertext{\begin{flushright}
$i \ge 2; \: k_j, l_j \ge 0; \: \epsilon_j \in \{0,1\}$.
\end{flushright}}
\end{align*}
For the monomials of type (I''), the notation $v_2^{<n}x$ means monomials of the form $v_2^ix$ for $i < n$.
\end{thm} 

\begin{rmk}\label{rmk:MRE1}
We do not know if there are differentials in the May-Ravenel spectral sequence
$$ \E{MR}{}_1(\sitd(2)) \Rightarrow H^*(\sitd(2)) \cong H^{*,*}(\mc{C}). $$
  Even in relatively low degrees, possibilities are plentiful.  For example, there could be a differential
$$ d_4^{MR}(h^2_{5,0}) \overset{?}{=} v_2^{14} \td{h}_{2,1}h_{3,0}^2. $$
We also do not know if there are possible hidden $v_2$-extensions in the May-Ravenel spectral sequence. Again, there are endless possibilities - as an example, there could be a hidden extension
$$ v_2^{16} h_{3,0}^4 \overset{?}{=} v_2^{14}h_{2,1}h_{3,0}^2h_{3,1}. $$
However, in the very low degrees  which  are relevant  to the computations later in this paper, there  are no possibilities of differentials or hidden $v_2$-extensions.
\end{rmk}


\section{The agathokakological method}\label{sec:agatho}

In this section we will adapt the agathokakological method introduced in \cite{BBBCX} to our present setting, to compute the $E_2$-term of the tmf-ASS for $Z$.

\subsection{Overview of the method}

The goal is to compute $\E{\tmf}{}_2^{*,*}(Z)$.  The short exact sequences
\begin{gather*}
 0 \rightarrow V^{*,*}(Z) \rightarrow \E{\tmf}{}_1^{*,*}(Z) \rightarrow \mc{C}^{*,*}(Z) \rightarrow 0, 
 \\
0 \to V^{*,*,*}_{alg}(Z) \to \E{\tmf}{alg}_1^{*,*,*}(Z) \to \mc{C}^{*,*,*}_{alg}(Z) \to 0 
\end{gather*}
give rise to long exact sequences
\begin{gather}
\cdots \rightarrow H^{*,*}(V) \rightarrow \E{\tmf}{}_2^{*,*}(Z) \rightarrow H^{*,*}(\mc{C}) \xrightarrow{\partial} H^{*+1, *}(V) \rightarrow \cdots, 
\label{eq:partialakss} \\
\cdots \rightarrow H^{*,*,*}(V_{alg}) \rightarrow \E{\tmf}{alg}_2^{*,*,*}(Z) \rightarrow H^{*,*,*}(\mc{C}_{alg}) \xrightarrow{\partial_{alg}} H^{*+1, *,*}(V_{alg}) \rightarrow \cdots \label{eq:partialalgakss}
\end{gather}
In particular, (\ref{eq:partialakss}) reduces the computation of $\E{\tmf}{}_2(Z)$ to the computation of $H^{*,*}(\mc{C})$ and $H^{*,*}(V)$.  We have established a means to understand $H^{*,*}(\mc{C})$ (Theorem~\ref{thm:MRE1}).  We are therefore left to compute $H^{*,*}(V)$.

By using Bruner's Ext program to compute 
$$ \Ext^{*,*}_{A_*}(\FF_2, H_*Z) $$ 
through a range, we can then use the Mahowald spectral sequence
$$ \E{\tmf}{alg}^{*,*,*}_2(Z) \Rightarrow \Ext_{A_*}^{*,*}(\FF_2, H_*Z) $$
to deduce $\E{\tmf}{alg}^{*,*,*}_2(Z)$ by reverse engineering (note that this is backwards from the usual direction of deduction with a spectral sequence).
We have computed $ H^{*,*,*}(\mc{C}_{alg})$ (Theorem~\ref{thm:HCalg}).
We can then use (\ref{eq:partialalgakss}) to deduce 
$H^{*,*}(V) = H^{*,0,*}(V_{alg}) $.  

\begin{rmk}
Note that our only interest in $\Ext^{*,*}_{A_*}(\FF_2, H_*Z)$ is to determine $H^{*,*}(V)$.  We are not investigating the classical Adams spectral sequence of $Z$.
\end{rmk}

\subsection{The agathokakological spectral sequences}

The strategy outlined in the previous subsection will be aided by the construction of a pair of spectral sequences: the \emph{topological agathokakological spectral sequence} (topological AKSS), and the \emph{algebraic agathokakological spectral sequence} (algebraic AKSS).

We begin with the topological AKSS.
Consider the short exact sequence
\begin{equation}\label{eq:topSES} 
0 \rightarrow V^{n,*}(Z) \rightarrow \E{\tmf}{}^{n,*}_1(Z) \xrightarrow{g} \mc{C}^{n,*}(Z) \rightarrow 0.
\end{equation}
We will now introduce a refinement of the $\tmf$-ASS which separates the good and evil complexes into different degrees.  The good complex $\mc{C}^{n,*}(Z)$ will be regarded as being in filtration $n$, and the evil complex $V^{n,*}(Z)$ will be regarded as in filtration $n+\epsilon$, where $\epsilon$ is regarded as a fixed quantity with
$$ n < n+\epsilon < n+1-\epsilon < n+1. $$

Let $\br{\tmf}$ denote the fiber of the unit
$$ \br{\tmf} \rightarrow S \rightarrow \tmf. $$
The $\tmf$-ASS for $Z$ arises from the decreasing filtration of $Z$ given by 
$$
\xymatrix{
Z \ar@{=}[r] &
F_0 \ar[d] & 
F_1 \ar[l] \ar[d] &
F_2 \ar[l] \ar[d] &
\cdots \ar[l] 
\\
& k(2) &
k(2) \wedge \br{\tmf} &
k(2) \wedge \br{\tmf}^{2}
}
$$
with 
$$ F_s := \br{\tmf}^s \wedge Z. $$
By Proposition~\ref{prop:k(n)}, there are fiber sequences
$$ H_s \rightarrow k(2) \wedge \br{\tmf}^s \rightarrow K_s $$
where $H_s$ is a wedge of mod 2 Eilenberg-MacLane spectra and $K_s$ is a wedge of $k(2)$'s.  By Verdier's axiom we get a braid of fiber sequences

$$
\xymatrix{
F_{s+1} \ar@/^1.5pc/[rr] \ar[dr] && F_s \ar[dr] \ar@/^1.5pc/[rr] && K_s \\
& F_{s+\epsilon} \ar[ur] \ar[dr] && k(2) \wedge \br{\tmf}^s \ar[ur] \\
&& H_s \ar[ur]
}
$$
 
where $F_{s+\epsilon}$ is defined to be the fiber of the map $F_s \to K_s$. This results in a refinement of the $\tmf$-Adams filtration of $Z$
\begin{equation}\label{eq:akfilt}
\xymatrix{
Z \ar@{=}[r] &
F_0 \ar[d] & 
F_\epsilon \ar[l] \ar[d] &
F_1 \ar[l] \ar[d] &
F_{1+\epsilon} \ar[l] \ar[d] &
F_2 \ar[l] \ar[d] &
F_{2+\epsilon} \ar[l] \ar[d] &
\cdots \ar[l] 
\\
& K_0 &
H_0 &
K_1 & 
H_1 &
K_2 & 
H_2 &
}
\end{equation}

The spectral sequence associated to this filtration is the \emph{topological agathokakological spectral sequence (AKSS)}:
$$ \{ \E{akss}{}^{n+\alpha\epsilon,t}_{r + \beta \epsilon} \} \Rightarrow \pi_{t-n}(Z) $$
\begin{align*}
n, t \in \mb{N}, \\
\alpha \in \{ 0, 1\}, \\
\beta \in \{ -1, 0, 1\} 
\end{align*}
The pages of this spectral sequence are ordered by
$$ n - \epsilon < n < n + \epsilon < n+1 $$
with differentials 
\begin{align*}
d^{akss}_{r-\epsilon}:  & \E{akss}{}_{r - \epsilon}^{n+ \epsilon,t} \rightarrow \E{akss}{}_{r- \epsilon}^{n+r,t},  \\
d^{akss}_{r}:  & \E{akss}{}_{r}^{n+ \alpha\epsilon,t} \rightarrow \E{akss}{}_{r}^{n+r +\alpha\epsilon, t},  \\
d^{akss}_{r+\epsilon}:  & \E{akss}{}_{r+\epsilon}^{n, t} \rightarrow \E{akss}{}_{r+ \epsilon}^{n+r +\epsilon, t} .
\end{align*}

\begin{rmk}
The reader will notice that the AKSS could be reindexed to a more standard format by reindexing the filtration by:
\begin{align*}
n & \mapsto 2n, \\
n+\epsilon & \mapsto 2n+1.
\end{align*}
Our reason for choosing this non-standard indexing is that it displays the AKSS as a refinement of the $\tmf$-ASS, so that there are short exact sequences
$$ 0 \rightarrow \E{akss}{}^{n+\epsilon,t}_{r-\epsilon} \to \E{\tmf}{}^{n,t}_r  \to \E{akss}{}^{n,t}_{r-\epsilon} \to 0. $$ 
\end{rmk}

The $E_1$-term takes the form
$$\E{akss}{}^{n+\alpha \epsilon, t}_1= 
\begin{cases}
\mc{C}^{n,t}(Z), & \alpha = 0, \\
V^{n, t}(Z), & \alpha = 1. 
\end{cases}
$$
The $d_1$-differential 
$$ d^{akss}_{1}: \E{akss}{}^{n+\alpha \epsilon, t}_1 \rightarrow \E{akss}{}^{n+ 1+\alpha\epsilon, t}_1
$$
is given by the differentials in the good and evil complexes:
$$d_1^{akss} = 
\begin{cases}
d_1^{good}, & \alpha = 0, \\
d_1^{evil}, & \alpha = 1. 
\end{cases}
$$
We therefore have
$$\E{akss}{}^{n+\alpha \epsilon, t}_{1+\epsilon} = 
\begin{cases}
H^{n,t}(\mc{C}), & \alpha = 0, \\
H^{n, t}(V), & \alpha = 1.
\end{cases}
$$
The only nonzero $d_{1+\epsilon}$-differentials are of the form
$$ H^{n, t}(\mc{C}) = \E{akss}{}^{n, t}_{1+\epsilon} \xrightarrow{d_{1+\epsilon}} \E{akss}{}^{n+1 + \epsilon, t}_{1+\epsilon} = H^{n+1, t}(V), $$
for which we have 
$$ d_{1+\epsilon} = \partial $$
where $\partial$ is the connecting homomorphism of (\ref{eq:partialakss}).
 
The algebraic AKSS is constructed by applying $\Ext^{*,*}_{A_*}(\FF_2, H_*(-))$ to the diagram (\ref{eq:akfilt}) (compare with \cite[Sec.~7]{BBBCX}).  The resulting spectral sequence takes the form
$$ \{ \E{akss}{alg}^{n+\alpha\epsilon,s,t}_{r+\beta \epsilon}(Z) \} \Rightarrow \E{ass}{}^{n+s,t}_2(Z) $$  
with differentials 
\begin{align*}
d^{akss}_{r-\epsilon}:  & \E{akss}{alg}_{r - \epsilon}^{n+ \epsilon, s, t} \rightarrow \E{akss}{alg}_{r- \epsilon}^{n+r, s-r+1, t},  \\
d^{akss}_{r}:  & \E{akss}{alg}_{r}^{n+ \alpha\epsilon, s, t} \rightarrow \E{akss}{alg}_{r}^{n+r +\alpha\epsilon, s-r+1, t},  \\
d^{akss}_{r+\epsilon}:  & \E{akss}{alg}_{r+\epsilon}^{n, s, t} \rightarrow \E{akss}{alg}_{r+ \epsilon}^{n+r +\epsilon, s-r+1, t} .
\end{align*} 
We have
$$
\E{akss}{alg}^{n+\alpha\epsilon,s,t}_{1+\epsilon}(Z) =
\begin{cases}
H^{n,s,t}(\mc{C}_{alg}), & \alpha = 0, \\
H^{n,t}(V), & \alpha = 1, s = 0, \\
0, & \text{otherwise}
\end{cases}
$$
and 
$$ d^{akss}_{1+\epsilon} = \partial_{alg} $$
where $\partial$ is the connecting homomorphism of (\ref{eq:partialalgakss}).

Because for $s > 0$ we have 
$$ \E{akss}{alg}^{n+\epsilon,s,t}_1 = 0, $$
there are no non-trivial differentials
$$ d_{r+\beta\epsilon}(x) = y $$
with $x$ in filtration $n+\epsilon$ and $r > 1$.

The following very useful lemma shows that the $d_{1+\epsilon}$ differentials in the topological AKSS can be deduced from the $d^{alg}_{1+\epsilon}$ differentials in the algebraic AKSS.

\begin{lem}\label{lem:d1epsilon}
For $n = 0$, the differentials
$$ d_{1+\epsilon}: \E{akss}{}^{n, t}_{1+\epsilon} \xrightarrow{d_{1+\epsilon}} \E{akss}{}^{n+1 + \epsilon, t}_{1+\epsilon} $$
are trivial.  For $n \ge 1$, they are determined by the following commutative diagram:
$$
\xymatrix{
\E{akss}{}^{n, t}_{1+\epsilon} \ar[d] \ar[r]^-{d_{1+\epsilon}} &
\E{akss}{}^{n+1 + \epsilon, t}_{1+\epsilon} \ar@{=}[d] \\
\E{akss}{alg}^{n,{0}, t}_{1+\epsilon}  \ar[r]_-{d^{alg}_{1+\epsilon}} &
\E{akss}{alg}^{n+1 + \epsilon, {0}, t}_{1+\epsilon}
}
$$
\end{lem}
\begin{proof}
Topologically, $d_{1+\epsilon}$ derives from applying $\pi_*$ to the composite
\begin{equation}\label{eq:topconnect}
K_n  \rightarrow \tmf^{\smsh n+1} \smsh Z \rightarrow \tmf^{\smsh n+2} \smsh Z \rightarrow H_{n+1}.
\end{equation}
The first statement follows from the fact that the only elements in $H^{n,*}(\mc{C})$ for $n = 0$ are powers of $v_2$.
The second statement follows from the fact that $d^{alg}_{1+\epsilon}$ is the induced map of Adams $E^{0,*}_2$-terms coming from the composite (\ref{eq:topconnect}):
\[
\mc{C}_{alg}^{n,0,*} = \E{ass}{}^{0,*}\left(K_n\right) \rightarrow \E{ass}{}^{0,*}\left(H_{n+1}\right) = V^{n+1,*}. \qedhere \]
\end{proof}

The $E_2$-term of the $\tmf$-ASS is deduced from the short exact sequence
$$ 0 \rightarrow \E{akss}{}^{n+\epsilon,t}_2 \rightarrow \E{\tmf}{}^{n,t}_2 \rightarrow \E{akss}{}^{n,t}_2 \rightarrow 0.
$$


\subsection{The dichotomy principle}

Elements in $\E{akss}{alg}^{n,s,t}_{r+\beta \epsilon}(Z)$ are called \emph{good}, and elements in $\E{akss}{alg}^{n+\epsilon,s,t}_{r+\beta \epsilon}(Z)$ are called \emph{evil}.  Non-trivial elements of $\E{ass}{}_2(Z)$ are called \emph{good} (respectively \emph{evil}) if they are detected in the AKSS by good (respectively evil) classes.  

The key to computing the algebraic AKSS is to determine which elements of $\E{ass}{}_2(Z)$ are good and which are evil.  This is done by linking $v_2$-periodicity with goodness.  An element of $\E{ass}{}_2(Z)$ is \emph{$v_2$-periodic} if its image under the homomorphism
$$ \E{ass}{}_2(Z) \rightarrow v_2^{-1}\E{ass}{}_2(Z) $$
is non-trivial.  Otherwise it is said to be \emph{$v_2$-torsion}.

The following two propositions give a practical means of determining whether an element of $\E{ass}{}_2(Z)$ is $v_2$-periodic.

\begin{prop}\label{prop:v2Ext}
We have
$$ v_2^{-1}\E{ass}{}_2(Z) \cong \FF_2[v_2^{\pm}, \td{h}_{2,1}, h_{3,0}, h_{3,1}, h_{4,0}, h_{4,1}, \cdots ]. $$
\end{prop}

\begin{proof}
The computation is almost identical to that of \cite[(2.20)]{MahowaldRavenelShick}.
\end{proof}

\begin{cor}\label{cor:gooddiffs}
For $r > 1$, there are no $d_r$ differentials between good classes in the algebraic AKSS.
\end{cor}

\begin{proof}
Proposition~\ref{prop:v2Ext} implies that the $v_2$-localized algebraic AKSS collapses at $E_{1+\epsilon}$.  The result follows from the fact that the map
$$ \E{akss}{alg}^{n+\alpha\epsilon,s,t}_{1+\epsilon}(Z) \hookrightarrow v_2^{-1} \E{akss}{alg}^{n+\alpha\epsilon,s,t}_{1+\epsilon}(Z) $$
is an injection for $\alpha = 0$ (the good part). 
\end{proof}

In order to state and prove the dichotomy principle, we will need to establish bounds on $v_2$-periodicity in $\Ext$, and on the evil complex.  Let $A_2$ denote the cofiber of the $v_2$-self map 
$$ \Sigma^6 Z \rightarrow Z. $$
We have
$$ H^*(A_2) \cong A(2) $$
as an $A(2)$-module (see \cite[Sec.~2]{bhateggerZ}).

\begin{lem}\label{lem:A2vanishing}
We have
$$ \E{ass}{}^{s,t}_2(A_2) = 0 $$
for
$$ s > \frac{(t-s)+12}{11}. $$
\end{lem}

\begin{proof}
The May spectral sequence 
for $\E{ass}{}_2(A_2 \wedge C\sigma)$ has $E_1$-term of the form
\begin{equation}\label{eq:Mayss}
 \E{May}{}_1^{*,*,*}(A_2 \wedge C\sigma) \cong \FF_2[h_{1,j_1}, h_{2,j_2}, h_{3,j_3}, \cdots \: : \: j_1 \ge 4; j_2 \ge 2; j_3 \ge 1; j_k \ge 0, k \ge 4]. \end{equation}
One checks that the smallest slope $\frac{s}{t-s}$ of these generators is $\frac{1}{11}$, given by $h_{2,2}$.  Therefore we have 
$$ \E{ass}{}_2^{s,t}(A_2 \wedge C\sigma) = 0 $$
for
$$ s > \frac{t-s}{11}. $$
It follows from the fact that $h_{1,3}^4 = 0$ in $\E{ass}{}_2^{*,*}(S)$ that the $h_{1,3}$-Bockstein spectral sequence
$$ \E{ass}{}^{*,*}_2(A_2 \wedge C\sigma)[h_{1,3}] \Rightarrow \E{ass}{}^{*,*}_2(A_2) $$
has a horizontal vanishing line at $E_\infty$, and one deduces that the translation of this $\frac{1}{11}$-vanishing line passing through $(t-s,s) = (21,3)$ (the bidegree of $h_{1,3}^3$) serves as a vanishing line for $\E{ass}{}_2^{*,*}(A_2)$.
\end{proof}

\begin{rmk}
The reader will notice that the notation $h_{i,j}$ is used both for the May spectral sequence generators of (\ref{eq:Mayss}) and for the May-Ravenel spectral sequence generators in $H^{*,*,*}(\mc{C}_{alg}(Z))$.
\emph{We warn the reader that these naming conventions are not consistent.}  The May spectral sequence generator $h_{i,j}$ corresponds to the 
element $\zeta_i^{2^j} \in A_*$, whereas the May-Ravenel generator $h_{i,j}$ corresponds to the element $t_i^{2^j} \in BP_*BP$.  Since under the map
$$ BP_*BP \to A_* $$
we have
$$ t^{2^j}_i \mapsto \zeta^{2^{j+1}}_{i}, $$ 
and the May-Ravenel generator $h_{i,j}$ actually corresponds to the May generator $h_{i,j+1}$.
\end{rmk}

\begin{prop}\label{prop:evilvanishing}
The evil complex $V^{n,t}(Z)$ satisfies
$$ H^{n,t}(V^{*,*}) = 0 $$
for
$$ n > \frac{(t-n)+12}{11}. $$
\end{prop}

\begin{proof}
We explain the relationship between $H^{*,*}(V)$ and $\E{ass}{}^{*,*}_2(A_2)$ by constructing a spectral sequence which relates them.  We first note that because $H^*(A_2) \cong A(2)$, we have
$$ \E{\tmf}{alg}^{n,s,t}_1(A_2) = 0 $$
for $s > 0$.  Therefore, the only possible non-trivial differentials in the $\tmf$-MSS are $d_1$ differentials, and
$$ \E{\tmf}{alg}^{n,0,t}_2(A_2) \cong \E{ass}{}^{n,t}_2(A_2). $$
The short exact sequence of $A_*$-comodules 
$$ 0 \to H_*Z \to H_*A_2 \to H_* \Sigma^7 Z \to 0 $$
induces a long exact sequence
$$ 0 \to \E{\tmf}{alg}^{n,0,t}_1(Z) \to \E{\tmf}{alg}^{n,0,t}_1(A_2) \to \E{\tmf}{alg}_1^{n,0,t-7}(Z) \xrightarrow{v_2} \E{\tmf}{alg}_1^{n,1,t}(Z) \to \cdots $$
We therefore deduce that there is a short exact sequence
$$ 0 \to \E{\tmf}{alg}^{n,0,t}_1(Z) \to \E{\tmf}{}^{n,0,t}_1(A_2) \to V^{n,t-7}(Z) \to 0. $$
This allows us to consider the decreasing filtration of cochain complexes, with associated filtration quotients:
$$  
\xymatrix{
\E{\tmf}{alg}_1^{n,0,t}(A_2) \ar@{->>}[d] & \E{\tmf}{alg}^{n,0,t}_1(Z) \ar@{_{(}->}[l] \ar@{->>}[d]  & V^{n,t}(Z) \ar@{_{(}->}[l] \ar@{=}[d] & 0 \ar[l] \\
V^{n,t-7}(Z) & \mc{C}_{alg}^{n,0,t}(Z) & V^{n,t}(Z)
}
$$
Taking cohomology, we get a strange little spectral sequence which we will dub the \emph{algebraic AKSS for $A_2$} as it more or less arises as a kind of mod $v_2$ version of the algebraic AKSS for $Z$.  If we index it as follows:\footnote{With this indexing convention the map $Z \to A_2$ results in a map of spectral sequences $\E{akss}{alg}^{n+\alpha\epsilon,s,t}_*(Z) \to \E{akss}{alg}_*^{n+\alpha\epsilon,t}(A_2)$ (which one takes to be the zero map on terms with $s > 0$).} 
\begin{align*}
\E{akss}{alg}^{n-\epsilon,t}_{1+\epsilon}(A_2) & =
H^{n, t-7}(V), \\
\E{akss}{alg}^{n,t}_{1+\epsilon}(A_2) & =
H^{n,0,t}(\mc{C}_{alg}), \\
\E{akss}{alg}^{n+\epsilon,t}_{1+\epsilon}(A_2) & =
H^{n,t}(V), \\
\end{align*}
then the resulting spectral sequence takes the form
$$ \E{akss}{alg}^{n+\alpha\epsilon,t}_{1+\epsilon}(A_2) \Rightarrow \E{ass}{}_2^{n,t}(A_2)  $$
with differentials
\begin{align*}
 d_{1+\epsilon}: & \E{akss}{alg}_{1+\epsilon}^{n-\epsilon,t}(A_2) \to \E{akss}{alg}_{1+\epsilon}^{n+1,t}(A_2) \\
 d_{1+\epsilon}: & \E{akss}{alg}_{1+\epsilon}^{n,t}(A_2) \to \E{akss}{alg}_{1+\epsilon}^{n+1+\epsilon,t}(A_2) \\
 d_{1+2\epsilon}: & \E{akss}{alg}_{1+2\epsilon}^{n-\epsilon,t}(A_2) \to \E{akss}{alg}_{1+2\epsilon}^{n+1+\epsilon,t}(A_2)
 \end{align*} 
and
$$ \E{akss}{alg}_2^{n+\alpha\epsilon,t}(A_2) = \E{akss}{alg}_\infty^{n+\alpha\epsilon,t}(A_2).
$$
The result follows for dimensional reasons (by induction on $t-n$) using Lemma~\ref{lem:A2vanishing} and the fact that
$$ H^{n,0,t}(\mc{C}_{alg}) = 0 $$
for
$$ n > \frac{t-n}{11} $$
(since the generator of $H^{*,*,*}(\mc{C}_{alg})$ with lowest slope is $\td{h}_{2,1}$, with slope $\frac{n}{t-n} = \frac{1}{11}$).
\end{proof}

\begin{prop}\label{prop:v2line}
The map
$$ \E{ass}{}^{s,t}_2(Z) \rightarrow v_2^{-1}\E{ass}{}^{s,t}_2(Z) $$
is an isomorphism for
$$ s > 
\frac{(t-s) + 12}{11}.$$
\end{prop}

\begin{proof}
The result follows from considering the map of algebraic AKSS's
$$ \E{akss}{alg}^{*,*,*}_*(Z) \rightarrow v_2^{-1}\E{akss}{alg}_*^{*,*,*}(Z) $$
and using 
Proposition~\ref{prop:v2Ext}, Corollary~\ref{cor:gooddiffs}, Proposition~\ref{prop:evilvanishing}, and the observation that the map
$$ H^{n,s,t}(\mc{C}_{alg}) \to v_2^{-1}H^{n,s,t}(\mc{C}_{alg}) $$
is an isomorphism for
$$ n+s > \frac{t-n-s}{11}. $$
\end{proof}
 
Given a class $x \in \E{ass}{}_2(Z)$, Proposition~\ref{prop:v2line} gives a straightforward technique to determine from low dimensional computations if $x$ is $v_2$-periodic. Let $k$ be chosen such that $v_2^k x$ lies in the range of Proposition~\ref{prop:v2line}. Then $x$ is $v_2$-periodic if an only if $v_2^k x \ne 0$.

The following theorem, analogous to the dichotomy principle in \cite{BBBCX}, completely determines whether classes in $\E{ass}{}_2$ are good or evil.  Note that because of Corollary~\ref{cor:gooddiffs} (which does not have an analog in the context studied in \cite{BBBCX}), the proof of the dichotomy principle for the algebraic AKSS is much more straightforward in the present context.

\begin{thm}[Dichotomy Principle]\label{thm:dichotomy}
Suppose that $x$ is a non-trivial class in $\E{ass}{}^{s,t}_2(Z)$.
\begin{enumerate}
\item If $x$ is $v_2$-torsion, it is evil.

\item Every class in the range 
\begin{equation}\label{eq:goodrange}
 s > 
\frac{(t-s) + 12}{11}
\end{equation}
is good.

\item Suppose $x$ is $v_2$-periodic,  and suppose that $k$ is taken large enough so that $v_2^k x$ lies in the range (\ref{eq:goodrange}).  Suppose that $v_2^k x$ is detected in the algebraic AKSS by a class in $\E{akss}{alg}^{n,*,*}_{1+\epsilon}$.  Then $x$ is good if and only if 
$$ s \ge n. $$
\end{enumerate} 
\end{thm}

\begin{proof}
We deduce (1) from Corollary~\ref{cor:gooddiffs}.
We deduce (2) from Proposition~\ref{prop:evilvanishing}.  For (3) suppose that $x$ is $v_2$-periodic with $v_2^k x$ detected in $\E{akss}{alg}^{n,*,*}_{1+\epsilon}$.  We will first consider the case where $x$ is evil, and then we will consider the case where $x$ is good. For the first case, suppose that $x$ is detected by an evil class 
$$ \td{x} \in \E{akss}{alg}^{n'+\epsilon,s-n',t}_{1+\epsilon}$$ in the algebraic AKSS.  Then we must have
$$ s = n'. $$  
Since $\td{x}$ is $v_2$-torsion, we deduce that the $v_2^k$-multiplication must arise from a hidden extension in the AKSS, and therefore
$$ s = n' < n. $$
For the second case, suppose that $x$ is detected by a good class  
$$ \td{x} \in \E{akss}{alg}^{n',s-n',t}_{1+\epsilon}.$$
Then we must have
$$ s - n' \ge 0. $$
We deduce from the proof of Corollary~\ref{cor:gooddiffs} that $n' = n$, and therefore $s - n \ge 0$ and
$$ s \ge n. $$
\end{proof}

\begin{war}
There is no dichotomy principle in the topological AKSS.
\end{war}


\section{Stem by stem computations}\label{sec:stemwise}

In this section, we apply the agathokakological techniques of the previous section to do low dimensional computations of $\pi_*Z$. Furthermore, we settle the ambiguity left in \cite{bhateggerK2Z} regarding the differentials in the Adams Novikov spectral sequence for $Z_{E(2)}$ (Theorem~\ref{thm:k2locss}).

\subsection{The algebraic AKSS}\label{sec:algAKSS}

In this section, we use the algebraic AKSS
$$ \{ \E{akss}{alg}^{n+\alpha\epsilon,s,t}_{r+\beta \epsilon}(Z) \} \Rightarrow \E{ass}{}^{n+s,t}_2(Z) $$  
to identify $H^*(V(Z))$ in the range relevant for computing $\pi_*Z$ in degrees $* \leq 39$.

 More specifically, we do these computations for a specific choice of $Z$ and $v_2$-self map. It is shown in \cite[\S 2]{bhateggerZ}
 that for any $Z \in \td{\mathcal{Z}}$ and $v_2^1$-self map $f \colon \Sigma^6 Z \to Z$, there is a cofiber sequence
\begin{equation}\label{eq:lesZA} \xymatrix{ \Sigma^6 Z \ar[r]^-{f} & Z \ar[r] & C(f) \ar[r] & \Sigma^7Z }\end{equation}
where $C(f)$ is a spectrum with the property that $H^*C(f)$ is isomorphic to $A(2)$ as an $A(2)$-module. Different choices of $Z \in \td{\mathcal{Z}}$ and $v_2^1$-self maps give rise to different $A$-module structures on $A(2)$.  

We will be working with a specific choice of $Z$.  To this end, 
endow the subalgebra $A(2) \subset A$ with the $A$-module structure given by Roth in \cite[p.30]{Roth}. Appendix~\ref{apx:data} gives the Bruner module definition data that encodes this $A$-module strucure.
Following \cite{bhateggerZ},\footnote{In  \cite{bhateggerZ}, $A(2)$ is denoted by $A_2$ and $B(2)$ by $B_2$.} define $B(2)$ as 
\[ B(2):= A(2) \otimes_{E(Q_2)} \mathbb{F}_2. \] 
The Bruner module definition data for this $A$-module is given in \cite[Appendix 1]{bhateggerZ}.

For the rest of this section we restrict our attention to those $Z \in \widetilde{\mathcal{Z}}$ 
with
$$ H^*Z \cong B(2) $$
as $A$-modules.
By \cite[Remark 5.4]{bhateggerZ}, there are four different homotopy types of finite spectra realizing $B(2)$. 
As explained in \cite[Sec.~2]{bhateggerZ}, the cofiber of any $v_2^1$-self map of our chosen $Z$ is a realization of the module $A(2)$. 

Since
$\Ext^{s,s+1}_{A}(A(2), A(2))=0$ for $s \geq 2$, it follows from \cite[Prop.~5.1]{bhateggerZ} that 
there is a unique homotopy type of spectra realizing our chosen $A$-module structure on $A(2)$.
Therefore, different choices of a $v_2^1$-self map on our chosen $Z$ will not affect the calculations that follow. For this choice, we let
\[A_2  := C(f).\]

In this section, we also define
\begin{align*}
\Ext_{A}^{s,t}(Z) &:= \Ext^{s,t}_A(H^*(Z), \FF_2),  &   \Ext_{A}^{s,t}(A_2) &:= \Ext^{s,t}_A(H^*(A_2), \FF_2).
\end{align*}
Both $ \Ext_{A}^{*,*}(Z) $ and $ \Ext_{A}^{*,*}(A_2) $ can be computed using Bruner's program \cite{Bruner}. The results are depicted in Figure~\ref{fig:EXTA2CHART} and Figure~\ref{fig:EXTA2CHART2} in Adams grading $(x,y) = (t-s,s)$.

\subsection{$v_2$-multiplication in $\Ext_A(Z)$}
To proceed with our computations, we will need to determine which classes in $\Ext_A^{*,*}(Z)$ are detected by evil classes, and which are detected by good classes. This will be done using the dichotomy principle (Theorem~\ref{thm:dichotomy}), and so we need to identify the $v_2$-periodic classes in $\Ext_A^{*,*}(Z)$. To do this, we proceed as follows.

Note that there is a long exact sequence
\begin{equation}\label{eq:lesZA} \xymatrix@C=1.5pc{\ldots \ar[r] & \Ext_{A}^{s,t}(Z) \ar[r] &\Ext_{A}^{s,t}(A_2)  \ar[r] &  \Ext_{A}^{s,t}(\Sigma^7Z) \ar[r]^{\delta} & \Ext_{A}^{s+1,t}(Z) \ar[r] & \ldots  } \end{equation}
where the connecting homomorphism $\delta$ corresponds to multiplication by $v_2$, 
\[ \delta = v_2 \colon  \Ext_{A}^{s,t}(\Sigma^7Z)  \cong   \Ext_{A}^{s,t-7}(Z) \to  \Ext_{A}^{s+1,t}(Z)  . \] 
The $v_2$-multiplications in $ \Ext_{A}^{*,*}(Z) $ are indicated by dotted lines of slope $(6,1)$ in Figures~\ref{fig:EXTA2CHART} and \ref{fig:EXTA2CHART2}. The indicated multiplications are completely determined by the long exact sequence \eqref{eq:lesZA}. 
In Example~\ref{ex:v2muliplication}, we give a sample proof deducing the existence of a $v_2$-multiplication from the long exact sequence.
The proofs for the other $v_2$-multiplications indicated in Figures~\ref{fig:EXTA2CHART} and \ref{fig:EXTA2CHART2} are also straightforward, though the arguments involving classes in stems $*\geq 40$ become more tedious due to the growing dimensions of $ \Ext_{A}^{*,*}(A_2)$ and of $ \Ext_{A}^{*,*}(Z)$. The $v_2$-multiplication data in Figures~\ref{fig:EXTA2CHART} and \ref{fig:EXTA2CHART2} is complete in stems $x \leq 39$. In stems $40 \leq  x\leq 60$, we only draw those multiplications which are necessary to apply part (3) of Theorem~\ref{thm:dichotomy} to do computations up to $* =39$.

\begin{ex}\label{ex:v2muliplication}
If $x$ is the non-zero class in $(t-s,s) = (15,1)$ of $\Ext^{*, *}_A(Z)$, then $v_2x\neq 0$. Indeed, in degree $(t-s,s) =(21,2)$ (the target of $v_2$-multiplication on $x$), $\Ext^{*, *}_A(A_2)$ is one dimensional over $\FF_2$. However, there are two possible contributions to $\Ext^{*, *}_A(A_2)$ in this degree from the long exact sequence \eqref{eq:lesZA}. (See Figure~\ref{fig:connecting} and its caption.) There is a class $\Sigma^7 y$ of $\Ext_A^{*,*}(\Sigma^7 Z)$, labeled ${\color{gray} \bullet 1}$ of Figure~\ref{fig:connecting}, where $y$ is the class labeled $1 \bullet$ in degree $(14,2)$ of Figure~\ref{fig:connecting}. There is also a class $z$ of $\Ext_A^{*,*}(Z)$, labeled $\bullet 6$ in Figure~\ref{fig:connecting}. Since $v_2y=0$ for degree reasons, $\Sigma^7 y$ is in the kernel of the connecting homomorphism $\delta$. Therefore, the non-zero element of $\Ext^{*, *}_A(A_2)$  corresponds to the class $\Sigma^7y$. For degree reasons, $\delta(z)=0$, and so there must be a class $w$ of degree $(22,1)$ in $\Ext_A^{*,*}(\Sigma^7 Z)$ such that $\delta(w) =z$. The only possibility is the class labeled by  ${\color{gray} \bullet 4}$ of Figure~\ref{fig:connecting}. The class $x$ corresponds to $4 \bullet$ in Figure~\ref{fig:connecting}, and so $w=\Sigma^7x$. It follows that $v_2x = z$.
\end{ex}

\begin{figure}
\centering
\includegraphics[angle=90, height=0.9\textheight]{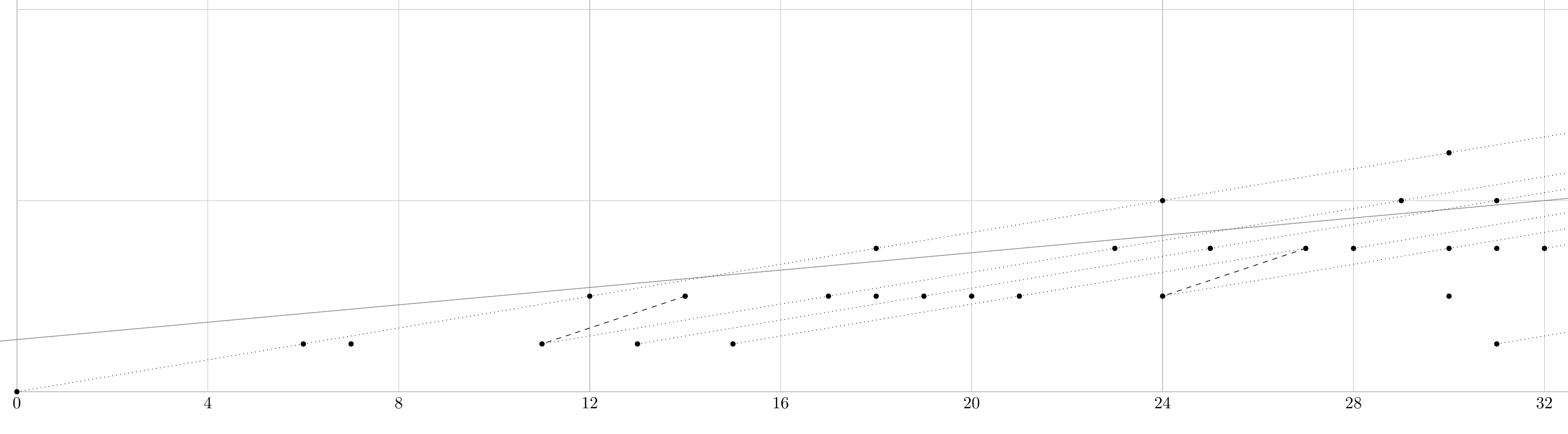}
\includegraphics[angle=90, height=0.9\textheight]{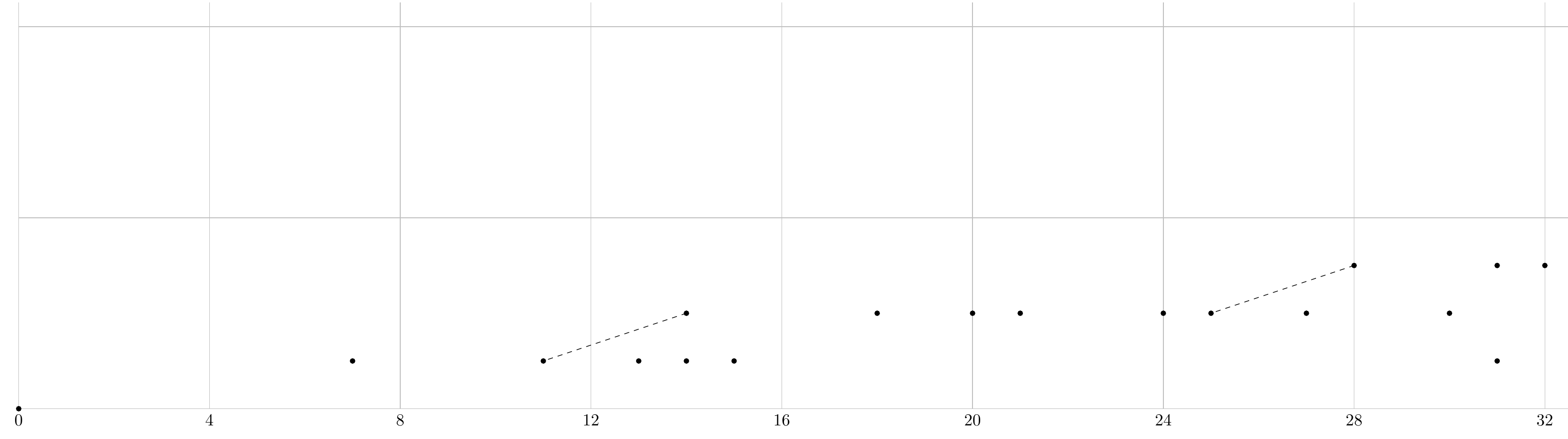}
\caption{$\mathrm{Ext}^{s,t}_{A}(Z)$ (left) and  $\mathrm{Ext}^{s,t}_{A}(A_2)$ (right) drawn in Adams coordinates $(x,y) = (t-s, s)$ in degrees $x\leq 32$. The dotted lines of slope $(6,1)$ denote $v_2$-multiplication. The solid lines of slope $(1,1)$ denote  $h_1$ (i.e. $\eta$) multiplications and those of slope $(3,1)$ denote $h_2$ (i.e. $\nu$) multiplications. The gray line of slope $1/11$ is the line of Proposition~\ref{prop:v2line}.}
\label{fig:EXTA2CHART}
\label{fig:EXTZCHART}
\end{figure}

\begin{figure}
\centering
\includegraphics[angle=90, height=0.86\textheight]{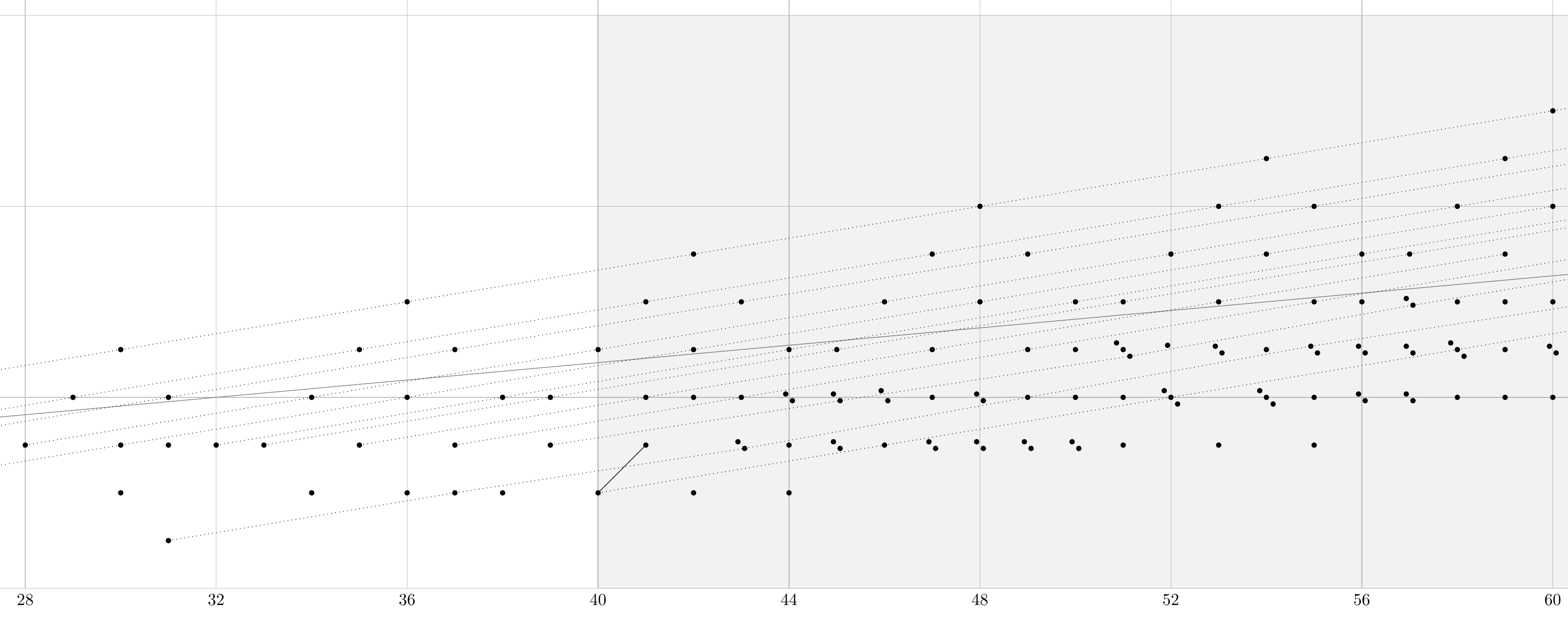}
\includegraphics[angle=90, height=0.86\textheight]{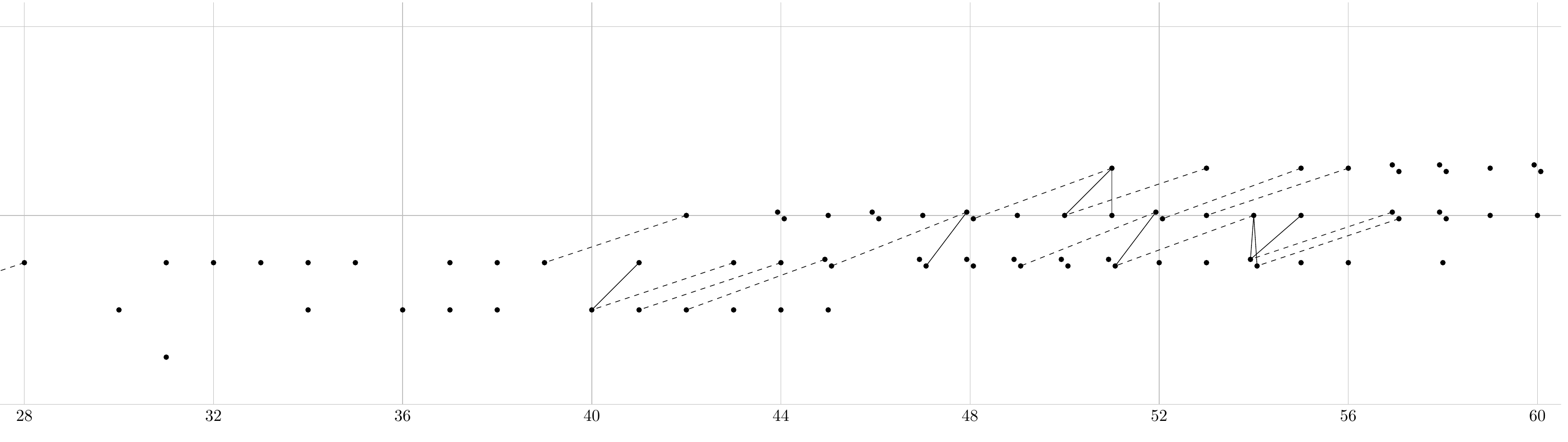}
\caption{$\mathrm{Ext}^{s,t}_{A}(Z)$ (left) and  $\mathrm{Ext}^{s,t}_{A}(A_2)$ (right) drawn in Adams coordinates $(x,y) = (t-s, s)$ in degrees $28 \leq x \leq 60$. In $\mathrm{Ext}^{s,t}_{A}(Z)$, not all $v_2$-multiplications are drawn in the shaded area, but we have included those needed for our computation. The gray line of slope $1/11$ is the line of Proposition~\ref{prop:v2line}.}
\label{fig:EXTA2CHART2}
\label{fig:EXTZCHART2}
\end{figure}

\begin{figure}
\centering
\includegraphics[angle=90, height=0.9\textheight]{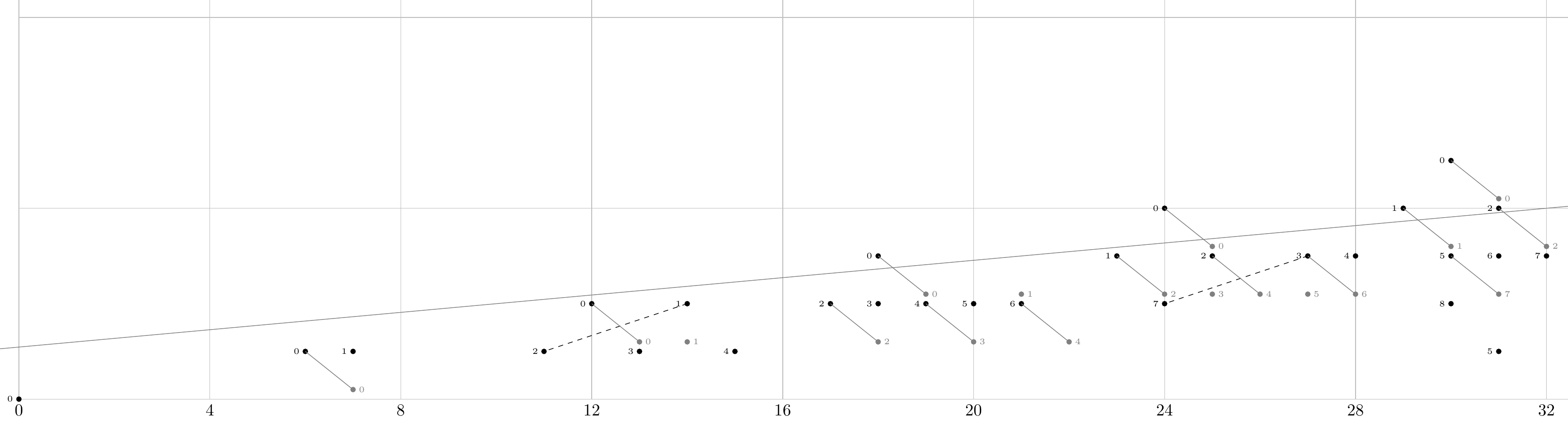}
\caption{The connecting homomorphism $ \Ext^{s,t}_A(\Sigma^7 Z) \to \Ext^{s+1,t}_A(Z)$. The gray classes are elements of $\Ext^{s,t}_A(\Sigma^7 Z)$, the black classes are elements of $\Ext^{s+1,t}(Z)$. The gray lines of slope $(-1,1)$ give the connecting homomorphism, which in turn corresponds to $v_2$-multiplication. The gray line of slope $1/11$ is the line of Proposition~\ref{prop:v2line}.  }
\label{fig:connecting}
\end{figure}

\begin{figure}
\centering
\includegraphics[angle=90, height=0.9\textheight]{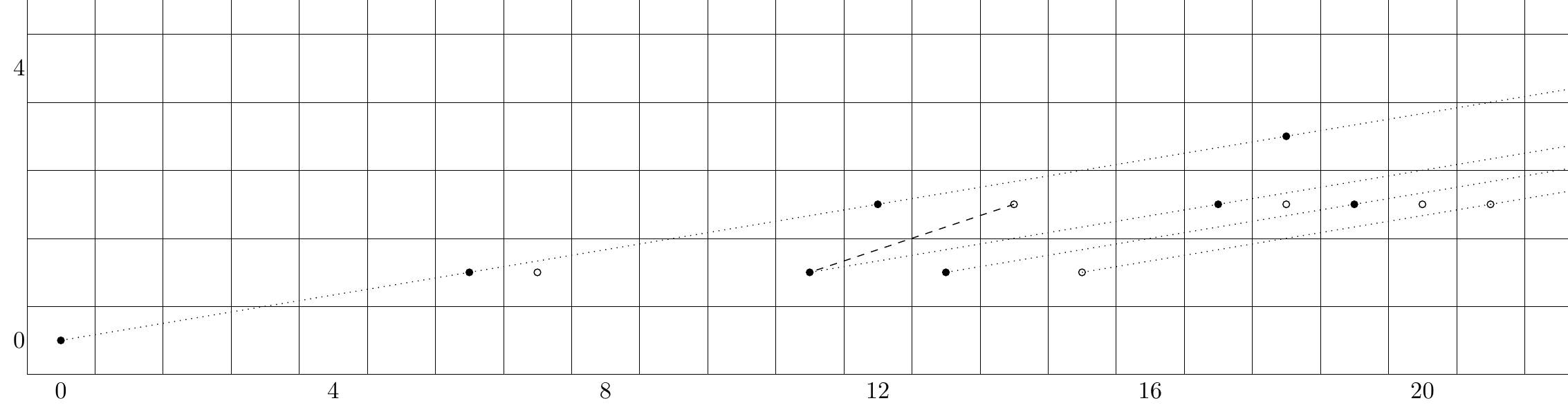}
\includegraphics[angle=90, height=0.9\textheight]{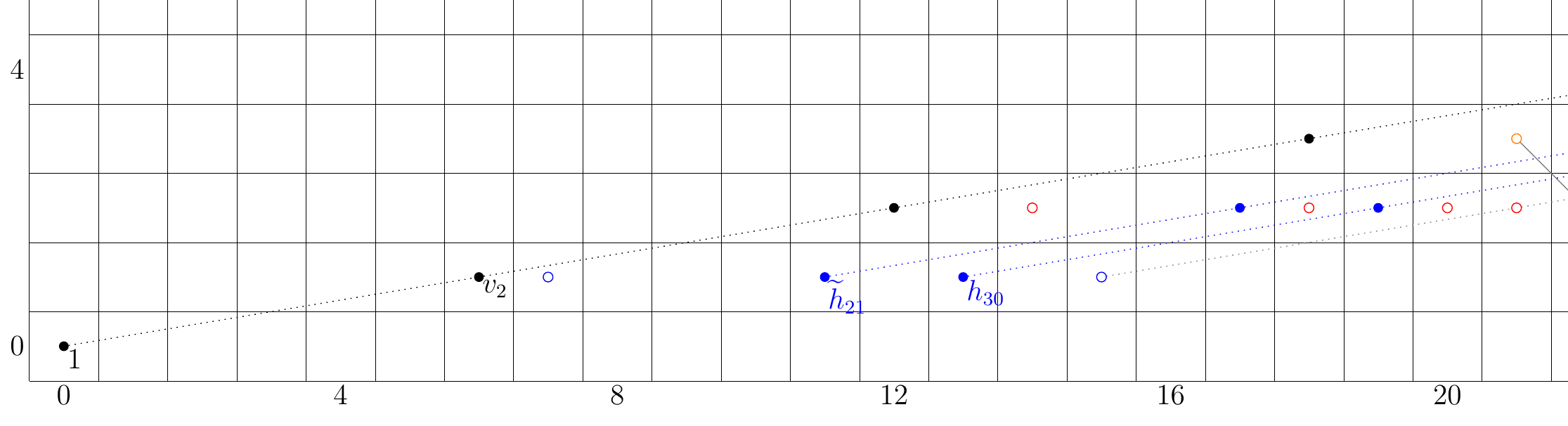}
\caption{The left chart is the $E_2$-term of the ASS for $Z$ in stems $0 \leq t-s \leq 21$. Classes detected by good are denoted by $\bullet$ and classes detected by evil by $\circ$. The right chart is the algebraic AKSS for $Z$, starting at the $E_{1+\epsilon}$-page.}
\label{fig:goodandevil}
\label{fig:AAKSSchart}
\end{figure}

\begin{figure}
\centering
\includegraphics[angle=90, height=0.8\textheight]{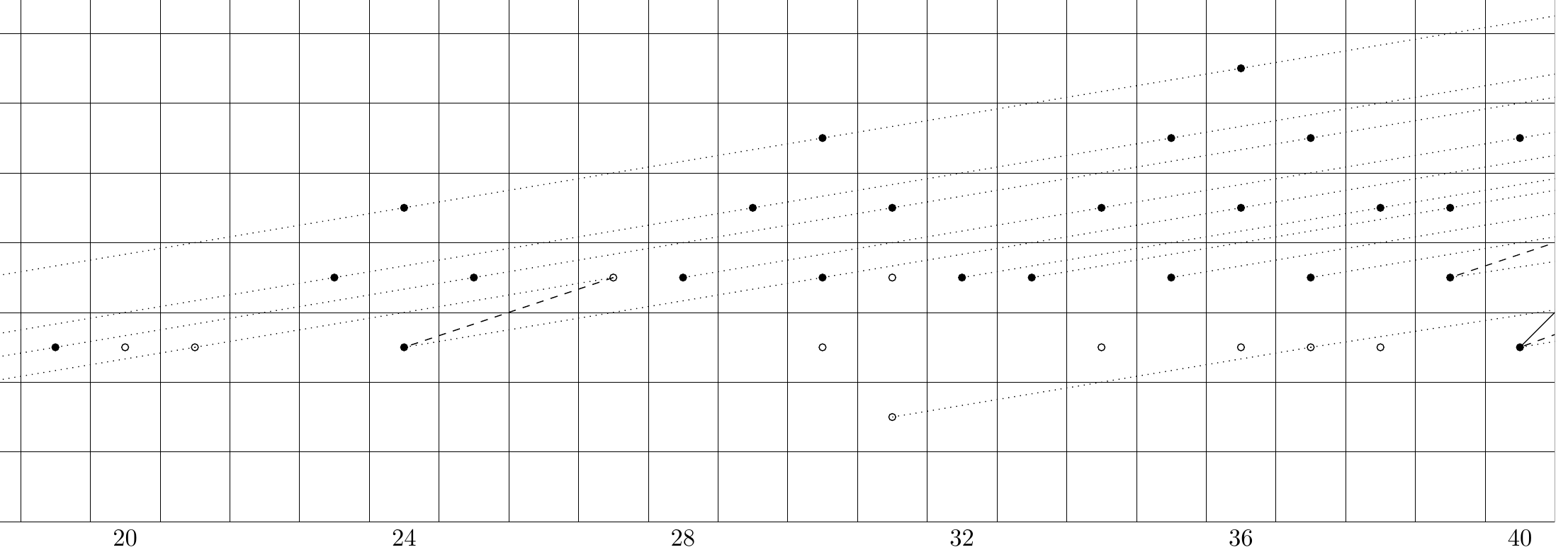}
\includegraphics[angle=90, height=0.8\textheight]{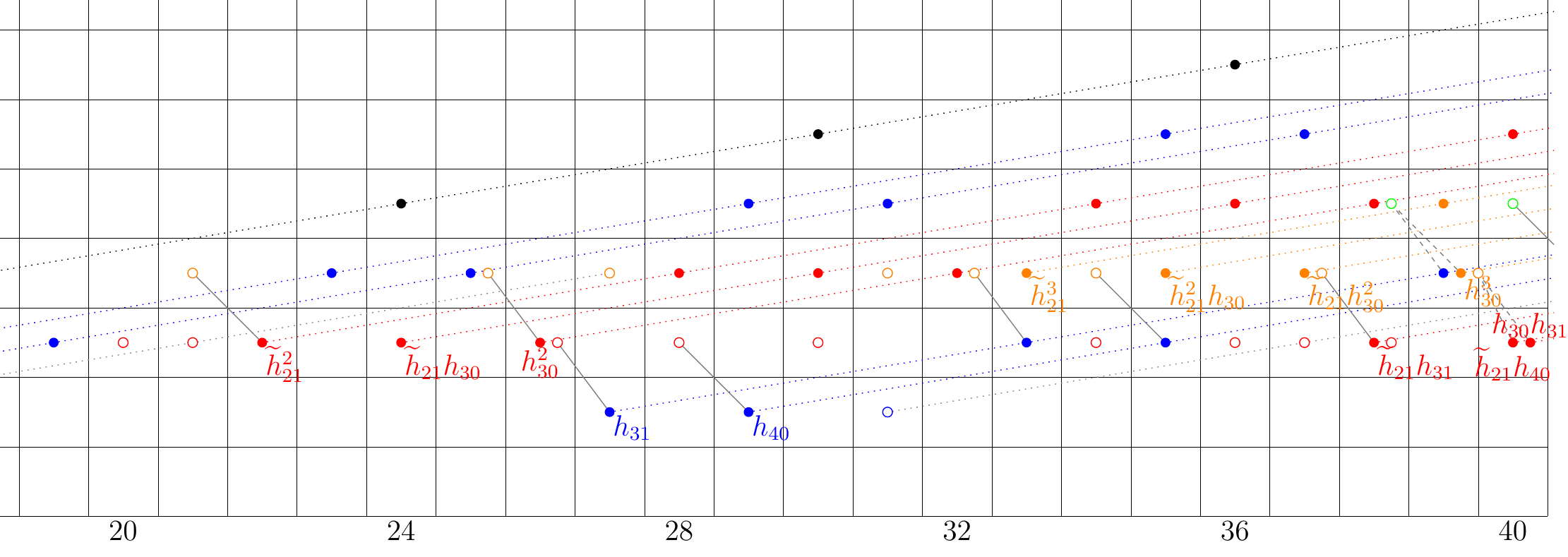}
\caption{The left chart is the $E_2$-term of the ASS for $Z$ in stems $19 \leq t-s \leq 40$. Classes detected by good are denoted by $\bullet$ and classes detected by evil by $\circ$. The right chart is the algebraic AKSS for $Z$, starting at the $E_{1+\epsilon}$-page.}
\label{fig:goodandevil2}
\label{fig:AAKSSchart2}
\end{figure}

\subsection{The differentials in the algebraic AKSS}
We turn to the computation of the algebraic AKSS. From Theorem~\ref{thm:HCalg}, we have that 
\begin{equation}\label{eq:alggood}
H^{*,*,*}(\mc{C}_{alg}) \cong \FF_2[v_2, \widetilde{h}_{2,1}, h_{3,0}, h_{3,1}, h_{4,0}, h_{4,1}, \ldots]  .\end{equation}
We use the dichotomy principle to determine which classes of $\Ext_A(Z)$ are good and which are evil. With \eqref{eq:alggood} and the results of the previous section on $v_2$-multiplications, this is  straightforward and result of this analysis is depicted in Figure~\ref{fig:goodandevil}. 

Having determined which classes in $\Ext_A^{*,*}(Z)$ are detected by good and evil, we can now deduce $H^{*,*}(V)$ from the algebraic AKSS. 
We name the evil classes in the algebraic AKSS (Figure~\ref{fig:AAKSSchart}) by
\[(x,y: n)^{ev},\]
where $(x,y)=(t-(s+n),s+n)$ is the Adams coordinate and $n$ is the $\tmf$-filtration. These classes are denoted by open circles in Figure~\ref{fig:AAKSSchart}. The good classes are denoted by solid circles.  For example, the class in degree $(x,y)= (7,1)$ in $\Ext_A(Z)$ is detected by evil and denoted by $\ev{blue}{7}{1}$ in the algebraic AKSS.

\begin{table}
\begin{tabular}{ | c | c | }
  \hline
  $n$ & color \\
  \hline
  \hline
  0 & black  \\
  \hline
  1 & {\color{blue}blue}  \\
  \hline
  2 & {\color{red}red}   \\
  \hline
   3 & {\color{orange}orange}  \\
  \hline
   4 & {\color{limegreen}green} \\
  \hline
\end{tabular}
\caption{The $\tmf$-filtration.}
\label{tab:colortmf}
\end{table}

In stems $0\leq x\leq 39$, the following evil classes exist for degree reasons. More precisely, these evil classes detect a class in $\Ext_A(Z)$ in a degree which contains no non-zero element of  $H^*(\mc{C}_{alg}) $:
\begin{align*}
\ev{blue}{7}{1} 		& & \ev{red}{14}{2}  		& &  \ev{orange}{27}{3}  						 \\
\ev{blue}{15}{1} 	& &  \ev{red}{18}{2} 		& &  \ev{orange}{31}{3} 				  	\\
\ev{blue}{31}{1}  	& &  \ev{red}{20}{2}  		& & 										\\
				& &  \ev{red}{21}{2}  		& & 				  	  		 		 \\
				& &  \ev{red}{30}{2} 		& & 											 \\
				& & 	\ev{red}{34}{2} 		& &								 \\
				& & 	\ev{red}{36}{2} 		& &								 \\
				& & 	\ev{red}{37}{2} 		& &						 			 \\
				& & 	\ev{red}{38}{2} 		& &						 	  
\end{align*}

The following evil classes exist because of the following differentials
 \begin{align*}
d_{1+\epsilon}(\widetilde{h}_{2,1}^2) &= \ev{orange}{21}{3}   \\
d_{1+\epsilon}(h_{3,0}^2) &= \ev{orange}{25}{3}  \\
d_{1+\epsilon}(h_{3,1}) &= \ev{red}{26}{2}  \\
d_{1+\epsilon}(h_{4,0}) &= \ev{red}{28}{2}  \\
d_{2+\epsilon}(v_2h_{3,1}) &= \ev{orange}{32}{3}  \\
d_{2+\epsilon}(v_2h_{4,0}) &= \ev{orange}{34}{3}  \\
d_{1+\epsilon}(\widetilde{h}_{2,1}h_{3,1}) &= \ev{orange}{37}{3}   \\
d_{3+\epsilon}(v_2^2h_{4,0}) &= \ev{limegreen}{40}{4}   .
\end{align*}
Examples of how we deduce these differentials is given in Example~\ref{ex:diff}.

\begin{ex}\label{ex:diff}
In degree $(t-s,s) = (26,2)$, $\Ext_A(Z)$ is trivial. Therefore, ${h}_{3,0}^2$ cannot survive the spectral sequence so must support a differential. Since the class in $(25,3)$ of $\Ext_A(Z)$ is detected by a good class, the only good class in that bidegree ($v_2^2h_{3,0}$) cannot be hit by a differential. So the target of the differential on  ${h}_{3,0}^2$ must be evil, and we obtain the differential
\[d_{1+\epsilon}(h_{3,0}^2) = \ev{orange}{25}{3}  .\]
The only non-trivial class in degree $(38,2)$ of $\Ext_A(Z)$ is detected by evil. Therefore $\widetilde{h}_{2,1}h_{3,1}$ must support a non-trivial differential. A similar analysis as before gives the differential
\[d_{1+\epsilon}(\widetilde{h}_{2,1}h_{3,1}) = \ev{orange}{37}{3} . \]
\end{ex}

Furthermore, 
 \begin{align}
 \label{eq:h30h31orh21h40}
d_{1+\epsilon}(h_{3,0}h_{3,1}) &= \alpha_1\ev{orange}{39}{3}   & \text{and} &  & d_{1+\epsilon}(\widetilde{h}_{2,1}h_{4,0}) &= \alpha_2\ev{orange}{39}{3}  
\end{align}
where at least one of the coefficients $\alpha_i$ is non-zero.
Similarly, at least one of the following $d_{2+\epsilon}$-differentials must occur
 \begin{align*}
d_{3+\epsilon}(v_2^2h_{3,1}) &= \ev{limegreen}{38}{4} & \text{or}  & & d_{1+\epsilon}(h_{3,0}^3) &= \ev{limegreen}{38}{4}   
\end{align*}
These ambiguities will be mostly settled in the next section.


\subsection{The topological AKSS and the computation of the $\tmf$-based ASS for $Z$}

Now, we turn to our analysis of the spectral sequence 
\[ \E{\tmf}{}_1^{n,t} = \pi_t ({\tmf}^{\s n+1} \s Z) \Longrightarrow \pi_{t-n} (Z)\]
and low-dimensional computations of $\pi_*Z$. Our analysis of the algebraic AKSS has allowed us to identify $H^{*,*}(V)$, together with the boundary homomorphism
\[H^{*,*,*}(\mc{C}_{alg}) \xrightarrow{\partial_{alg}} H^{*,*}(V) \]
in the form of $d_{1+\epsilon}$ differentials in the algebraic AKSS. Theorem~\ref{thm:MRE1} gives the $E_1$-term of the May-Ravenel SS
\begin{equation}\label{eq:EMR} \E{MR}{}_1(\sitd(2)) \Rightarrow H^{*,*}(\mc{C}). \end{equation}
It does not exclude the possibility of differentials, but there are no possibilities of differentials in the range of interest.

We record the following fundamental observations regarding the $d_1$-differential in the $\tmf$-ASS.
\begin{itemize}
\item An evil class cannot kill a good class via a $d_1$-differential since $V^{*,*}(Z)$ is a subcomplex of $\E{\tmf}{}_1^{*,*}(Z)$.
\item The $d_1$-differentials between evil classes are completely determined by those in the algebraic AKSS since $V^{*,*}(Z)  \cong V^{*,0,*}_{alg}(Z)$.
\item The $d_{1}$-differentials from good classes to evil classes are determined by the differentials in the algebraic AKSS. This is Lemma~\ref{lem:d1epsilon}.
\end{itemize}

In Figure~\ref{fig:tmfss}, we draw $\E{MR}{}_1(\sitd(2))$ in the range $0\leq t-n \leq 40$, together with the information about $H^{*,*}(V)$ and differentials obtained from the algebraic AKSS.

We use the map of spectral sequences from $\tmf$-based ASS to the classical ASS to ascertain that, in the range $t-s\leq 39$, there are no additional differentials. 

\begin{prop}\label{prop:classicalASS}
There are no non-trivial differentials in the classical ASS for $Z$ with source in stem $t-s \leq 39$.
\end{prop}
\begin{proof}
In the computations of $\pi_*Z$ for $0\leq *\leq 39$, the possible differentials have source in stems 
\[t-s= 30,31, 36, 37, 38, 40. \]
In stems $t-s<40$, the potential sources for differentials are the image of evil classes which are permanent cycles in the $\tmf$-based ASS.  Indeed, for degree reasons, these classes are permanent cycles provided that they are $d_1$-cycles. Since all $d_1$-differentials on evil classes have been recorded in Figure~\ref{fig:tmfss} and all of the potential sources are $d_1$-cycles, the claim follows. 
\end{proof}

\begin{rmk}
There is a potential $d_2$-differential in stem $t-s=40$ in the classical ASS for $Z$. In fact, this problem is tied to the ambiguity in \eqref{eq:h30h31orh21h40}, as we will see in the proof of the next proposition, where we will establish that such a non-trivial $d_2$ differential must occur in the ASS for $Z$. 
 \end{rmk}

 \begin{prop}\label{prop:higherdiffs}
The only non-trivial differential $d_r$ for $r>1$ in the $\tmf$-based ASS with source in the range $t-n\leq 40$ is
\[d_2(v_2h_{3,1}) = \ev{orange}{32}{3}. \]
 \end{prop}
 
 \begin{proof}
 Combining degree arguments with $v_2$-linearity, the only two possibilities are 
\begin{align*}
 d_2(v_2h_{3,1}) & = \ev{orange}{32}{3}, \\
 d_3(v_2^2 h_{3,1}) & = \ev{limegreen}{38}{4}.
\end{align*}
 By Proposition~\ref{prop:classicalASS}, the classical ASS for $Z$ collapses in this range. Therefore, $\pi_{32}Z$ and $\pi_{33}Z$ have order $2$. For this to be the case, we must have
 $d_2(v_2h_{3,1}) = \ev{orange}{32}{3}$
 in the $\tmf$-based ASS.  This settles the first possibility.
  
 We turn now to the second possible differential
 $d_3(v_2^2 h_{3,1})$.
 Recall from  (\ref{eq:h30h31orh21h40}) that we were unable to determine the coefficient $\alpha_1$ in 
  \[d_{1+\epsilon}(h_{3,0}h_{3,1}) = \alpha_1 \ev{orange}{39}{3}. \]
 It will turn out that these two ambiguities are interrelated, and through analyzing this relationship we will settle both.
  
 Since $\widetilde{h}_{2,1}h_{4,0}$ is not an element in $H^{*,*}(\mathcal{C})$, if $\alpha_1 = 0$ and 
  \[d_{1+\epsilon}(h_{3,0}h_{3,1})  = 0 , \]
 then it follows from the $\tmf$-based ASS that we must have
 $$ d_3(v_2^2 h_{3,1}) = \ev{orange}{39}{3} $$
 and $\pi_{39}Z$ has order $4$.
If, however, $\alpha_1 = 0$ and
 \[d_{1+\epsilon}(h_{3,0}h_{3,1}) = \ev{orange}{39}{3} , \]
 then it follows from the $\tmf$-based ASS that $\pi_{39}Z$ has order $2$. 
  
From the structure of the $\tmf$-ASS we deduce that the map
  $$ v_2: \pi_{33}(Z) \xrightarrow{v_2} \pi_{39}(Z) $$
  is zero.  
 It is immediate from Figure~\ref{fig:EXTA2CHART2} that the ASS for $A_2$ collapses in degree 39 to give
  $$ \pi_{39} (A_2) = \ZZ/2. $$
  It follows from the long exact sequence associated to the cofiber sequence
  $$ \Sigma^6 Z \xrightarrow{v_2} Z \to A_2 $$
  that we must have
  $$ \pi_{39}(Z) = \ZZ/2. $$
We therefore conclude that $\alpha_1 = 1$, so
  \[d_{1+\epsilon}(h_{3,0}h_{3,1}) = \ev{orange}{39}{3} \]
 and
  $$ d_3(v_2^2 h_{3,1}) = 0. $$
 \end{proof}
 
It follows from Proposition~\ref{prop:classicalASS} and Proposition~\ref{prop:higherdiffs} that Figure~\ref{fig:tmfss} 
is complete.

\begin{figure}
\centering
\includegraphics[angle=90, height=0.95\textheight]{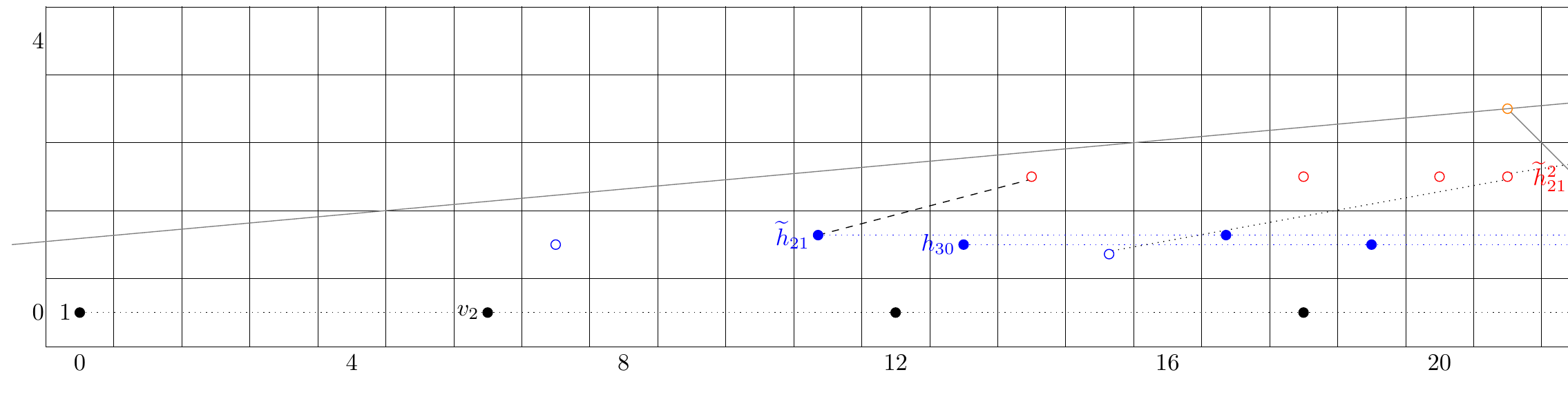}
\includegraphics[angle=90, height=0.95\textheight]{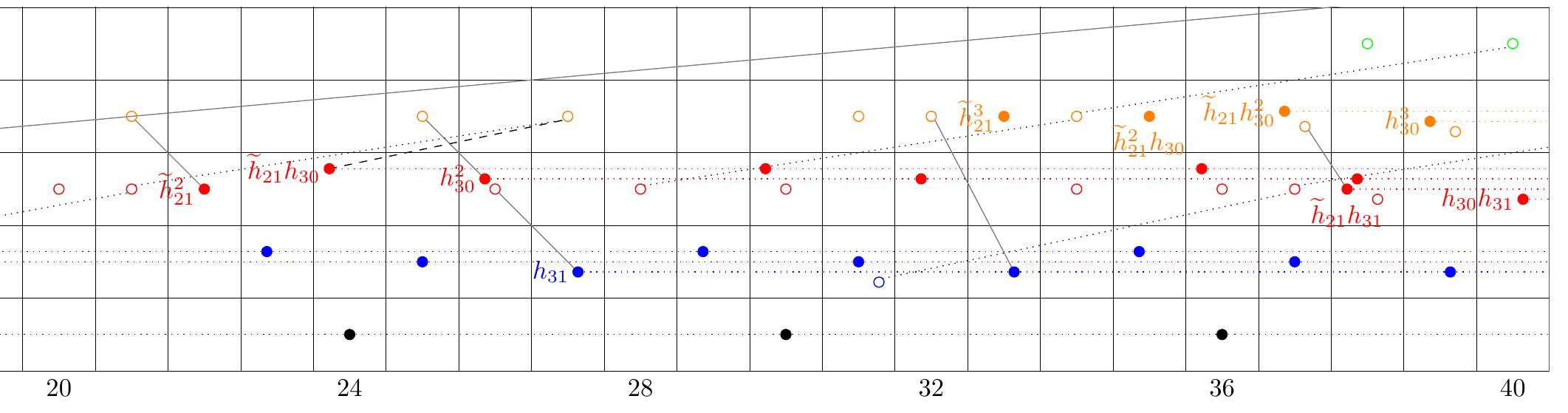}
\caption{The topological AKSS computing $\pi_{t-n} (Z)$ drawn in grading $(x,y) = (t-n, n)$, starting at the $E_{1+\epsilon}$-page. Gray lines are differentials. They are dashed, if our method is inconclusive. 
Dotted lines are known $v_2$-multiplications. Dashed line are known $\nu$-multiplications. The gray line of slope $1/11$ is the line of Proposition~\ref{prop:v2line}.}
\label{fig:tmfss}
\end{figure}

\subsection{The $E(2)$-localization of $Z$}

We end this section with one of the main goals of this paper, which is to determine the homotopy groups of $\pi_*Z_{E(2)}$.

\begin{thm}\label{thm:k2locss}
The Adams Novikov spectral sequence for $Z_{E(2)}$
collapses at the $E_2$-term.
\end{thm}
\begin{proof}
This spectral sequence is isomorphic ($E_2$ onwards) to the $v_2$-localized $\tmf$-ASS
\[  v_2^{-1}\E{\tmf}{}_1^{n,t}(Z) \Longrightarrow \pi_{t-n}Z_{E(2)}.\]
Inverting $v_2$ in the short exact sequence
$$ 0 \to V^{*,*}(Z) \to \E{\tmf}{}^{*,*}_1(Z) \to \mc{C}^{*,*}(Z) \to 0 $$
gives an isomorphism
$$ v_2^{-1}\E{\tmf}{}_1^{*,*}(Z) \cong v_2^{-1}\mc{C}^{*,*}(Z), $$
and hence an isomorphism
\begin{equation}\label{eq:v2locMR}
 v_2^{-1}\E{\tmf}{}_2(Z) \cong v_2^{-1}H^{*,*}(\mc{C}(Z)). 
 \end{equation}
Consider the $v_2$-localized May-Ravenel spectral sequence
$$
v_2^{-1}\E{MR}{}_1(\sitd(2)) \Rightarrow v_2^{-1}H^{*,*}(\mc{C}(Z)). 
$$ 
The $E_1$-term is given by inverting $v_2$ in Theorem~\ref{thm:MRE1}, and so is isomorphic to
\begin{equation}\label{eq:v2locgood}
\FF_2[v_2^{\pm 1}] \otimes E[ h_{3,0} , \td{h}_{2,1}, h_{3,1} , \td{h}_{4,1}]. 
\end{equation}
Since the $E_2$-term of ANSS for $Z_{E(2)}$ was computed in \cite{bhateggerK2Z} to be isomorphic to (\ref{eq:v2locgood}), we deduce from (\ref{eq:v2locMR}) that the $v_2$-localized May-Ravenel spectral sequence must collapse at $E_1$.
The $v_2$-localized $\tmf$-ASS for $Z$ is displayed in   Figure~\ref{fig:k2locss}. 
All differentials are $v_2$-linear since $Z_{E(2)}$ has a $v_2^1$-self map.
Furthermore, there is a horizontal vanishing line at $E_2$. Indeed, $E_2^{n,t}=0$ for $n\geq 5$. The class labeled by $1$ is the image of $\pi_0S^0 \to \pi_0 Z_{E(2)}$ so is a permanent cycle.
For degree reasons, the only possible non-trivial differentials are $d_3$'s with sources $v_2^kh_{31}$. However, since $d_3(v_2^2h_{3,1})$ in the $\tmf$-based ASS is zero, $v_2^2h_{3,1}$ maps to a $d_3$-cycle in $v_2^{-1}\E{\tmf}{}_1^{n,t}$. 
\end{proof}

\begin{figure}
\centering
\includegraphics[width=\textwidth]{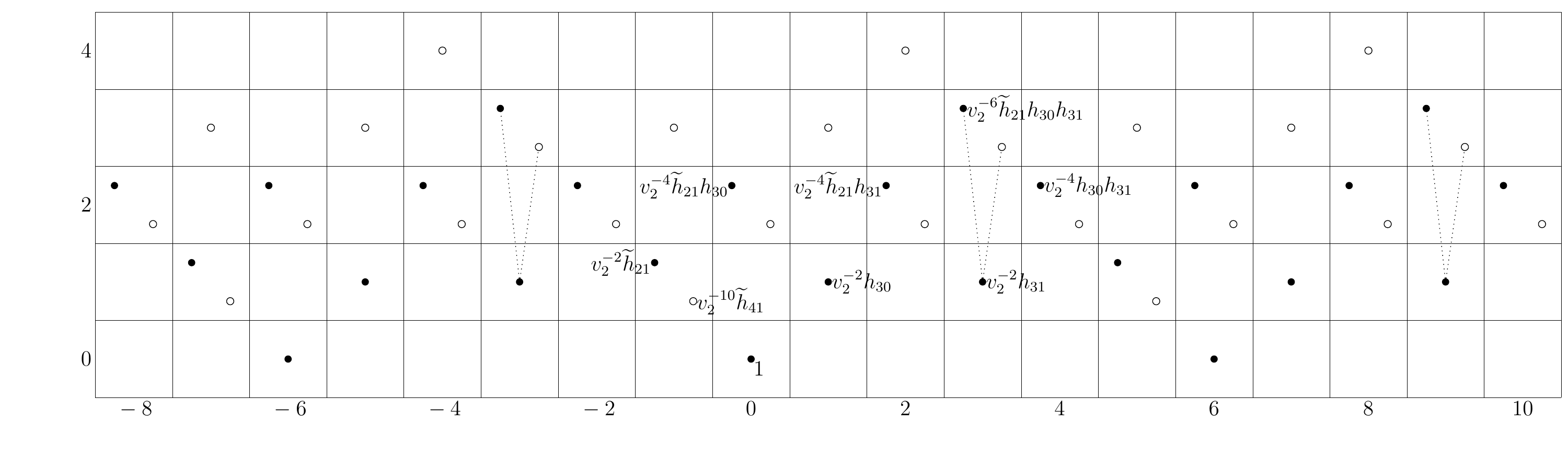}
\caption{The $E_{\infty}$-page of the Adams Novikov spectral sequence for $Z_{E(2)} = Z_{K(2)}$. The only possible non-trivial multiplication by $2$ extensions are dotted. Classes denoted by $\circ$ are multiples of $\zeta_2 \in \pi_{-1}S^0_{K(2)}$.}
\label{fig:k2locss}
\end{figure}

Next, we solve all but one exotic extension:
\begin{thm}\label{thm:k2locext}
For $k\not\equiv 3 \mod 6$, the groups $\pi_{k}Z_{E(2)}$ are annihilated by multiplication by $2$.
\end{thm}

\begin{proof}
The class detected by $\widetilde{h}_{2,1}$ in $\pi_{11}Z$ and $h_{3,0}$ in $\pi_{13}Z$ have order $2$ since there is no room in the $\tmf$-based ASS for exotic extensions in these degrees.  Therefore, their images in $\pi_*Z_{E(2)} = \pi_*Z_{K(2)}$ have order $2$, and so do all their multiples. The class detected by $v_2^{-10}\widetilde{h}_{4,1}$ is in the image of the bottom cell, $S^0_{K(2)} \to Z_{K(2)}$. Indeed, it is the image of the element $\zeta_2 \in \pi_{-1}S^0_{K(2)}$ discussed in \cite[Proposition 2.2.1]{DH}.\footnote{Our notation differs from \cite[(3.4) Theorem]{Ravenel_coh}. In this reference, our class $v_2^{-10}\widetilde{h}_{4,1}$ is closely related to $\rho_2$ and Ravenel's $\zeta_2$ is closely related to $v_2^{-2}\widetilde{h}_{2,1}$.} So, any multiple of $v_2^{-10}\widetilde{h}_{4,1}$ has order $2$.
\end{proof}

\begin{rmk}
In \cite{bhateggerK2Z}, the authors study the Adams Novikov spectral sequence for $Z_{K(2)}$, where $K(2)$ is the Morava $K$-theory whose formal group law is the Honda formal group law. Since the homotopy type of $Z_{K(2)}$ is independent of the choice of $K(2)$, Theorem~\ref{thm:k2locss} and Theorem~\ref{thm:k2locext} settle Conjecture 1 of \cite{bhateggerK2Z} for our particular choice of $Z \in \td{\mathcal{Z}}$, except for the group structure of $\pi_{3+6n}Z_{K(2)}$. 
\end{rmk}


\section{Discussion of the telescope conjecture for $Z$.}\label{sec:telescope}

While the telescope conjecture was initially proposed by Ravenel \cite{1984}, Ravenel was also the first to propose that it should be false for chromatic levels $\ge 2$ \cite{Raveneltelescope}.  The method of disproof proposed in \cite{Raveneltelescope} (the \emph{parameterized Adams spectral sequence}) turned out to not be sufficient to provide a counterexample to the telescope conjecture, but it laid out the blueprint for what could go wrong.  

A more detailed account of this story is laid out by 
Mahowald-Ravenel-Shick \cite{MahowaldRavenelShick}, who studied a family of Thom spectra $y(n)$ (defined for all primes $p$ and all $n \ge 1$) and some conjectures about their localized Adams spectral sequences, which, if true, would provide counterexamples to the telescope conjecture for all primes $p$ and all $n \ge 2$.  These conjectures lay the groundwork for a concrete counter-conjecture for the homotopy of the telescopes proposed by Ravenel in \cite{Ravenelparabola}, which we shall call the \emph{parabola conjecture}. 

In this section we outline the analog of this conjectural story for $Z$, and explain how the structure of the $\tmf$-ASS for $Z$ described in this paper is consistent with the parabola conjecture.
Specifically, let $\widehat{Z}$ denote the telescope of the $v_2$-self map on $Z$.
The telescope conjecture predicts that the map
\begin{equation}\label{eq:telmap}
 \widehat{Z} \rightarrow Z_{E(2)}
 \end{equation}
is an equivalence.  In Theorem~\ref{thm:k2locss}, we have already verified (up to a potential additive extension) that
$$ \pi_*Z_{E(2)} \cong \FF_2[v_2^{\pm 1}]\otimes E[\td{h}_{2,1}, h_{3,0}, h_{3,1}, \td{h}_{4,1}]. $$
The parabola conjecture predicts the structure of $\pi_*\widehat{Z}$, and 
in particular predicts that the map (\ref{eq:telmap}) is neither injective nor surjective in homotopy.

\subsection{The localized Adams spectral sequence for $Z$}

Consider the localized Adams spectral sequence
\begin{equation}\label{eq:LASSZ}
 v_2^{-1}\E{ass}{}_2^{*,*}(Z) \Rightarrow \pi_{*} \widehat{Z}.
 \end{equation}
The $E_2$-term of this spectral sequence was computed in Proposition~\ref{prop:v2Ext}:
$$ v_2^{-1}\E{ass}{}_2^{*,*}(Z) \cong \FF_2[v_2^{\pm}, \td{h}_{2,1}, h_{3,0}, h_{3,1}, h_{4,0}, h_{4,1}, \ldots]. $$
The analog of Mahowald-Ravenel-Shick's \emph{differentials conjecture} \cite[Conj.~3.16]{MahowaldRavenelShick} is the following.

\begin{conjecture}(Differentials Conjecture)\label{conj:diffs}
In the localized Adams spectral sequence (\ref{eq:LASSZ}) we have
\begin{align*}
 d_2(h_{4,0}) & = v_2 \td{h}_{2,1}^2, \\
 d_2(h_{i,0}) & = v_2 h_{i-2,1}^2, \\
 d_4(h_{i,1}) & = v_2 h_{i-1,0}^4.
 \end{align*}
\end{conjecture}

The idea is that the $d_2$ differentials in the above conjecture are lifted from the analogous differentials in the May-Ravenel spectral sequence (Theorem~\ref{thm:diffsquotient}), and that the $d_4$ differentials arise from these through an extended power argument \cite{Raveneltelescope}.

Note that $Z$ is \emph{not} a ring spectrum, as we have already seen in the topological AKSS, where $\td{h}_{2,1}$ is a permanent cycle but $\td{h}_{2,1}^2$ supports a non-trivial differential.
However, assuming these are the only $d_r$ differentials for $r \le 4$, and that they satisfy the Leibniz rule, we would have
$$ v_2^{-1}\E{ass}{}_5^{*,*}(Z) \cong \FF_2[v_2^{\pm}] \otimes E[\td{h}_{2,1}, h_{3,0}, h_{3,1}, x_3, x_4, x_5, \cdots]
$$
where
$$ x_{i} := h_{i,0}^2. $$
In particular, we have $h_{3,0}^2 = x_3$ rather than $h_{3,0}^2 = 0$, but this is somewhat irrelevant given that $Z$ is not a ring spectrum.  Our choice to present $v_2^{-1}\E{}{}_5$ in this manner leads to a more uniform discussion.

In the discussion after Conjecture~5.12 of \cite{MahowaldRavenelShick} (see also \cite{Raveneltelescope}), Mahowald-Ravenel-Shick predict the collapse of the localized ASS for $y(n)$ at a finite stage.  The analog of their conjecture in our context is the following.

\begin{conjecture}[Parabola Conjecture]\label{conj:parabola}
The localized ASS for $Z$ collapses at $E_5$, and therefore
$$ \pi_*\widehat{Z} \cong \FF_2[v_2^{\pm}] \otimes E[\td{h}_{2,1}, h_{3,0}, h_{3,1}, x_3, x_4, x_5, \cdots]. $$
Moreover, the telescope conjecture is false, and the kernel of (\ref{eq:telmap}) is the ideal 
$$ (x_3, x_4, \cdots) \subset \pi_*\widehat{Z} $$
and the ideal
$$ (\td{h}_{4,1}) \subset \pi_*Z_{E(2)} $$
maps isomorphically onto the cokernel of (\ref{eq:telmap}).
\end{conjecture} 

\begin{rmk}\label{rmk:zeta}
Note that the element $v_2^{-10}\td{h}_{4,1}$ is the image of the element $\zeta_2 \in \pi_{-1}S_{K(2)}$ (see the proof of Theorem~\ref{thm:k2locss}), so the second part of the parabola conjecture predicts that $\zeta_2$ is not in the image of the telescopic homotopy.  Note that this was the basis of Ravenel's initial attempt to disprove the telescope conjecture \cite{Raveneltelescope}.
\end{rmk}

We will now explain why we call Conjecture~\ref{conj:parabola} the ``parabola conjecture.''

\subsection{Unbounded $v_2$-torsion in the $\tmf$-ASS for $Z$}

The key to Mahowald's proof of the telescope conjecture at chromatic level $1$ was his \emph{bounded torsion theorem} \cite{Mahowaldbo}, which states that the $E_2$-page of the $\bo$-ASS for the sphere decomposes into a direct sum of $v_1$-periodic classes, and $v_1^2$-torsion classes.  We will explain how the analogous phenomenon likely fails in the context of the $\tmf$-ASS for $Z$.

We have already seen (Theorem~\ref{thm:MRE1}) that the May-Ravenel $E_1$-page has unbounded $v_2$-torsion.  But we must run some more differentials in the $\tmf$-ASS to relate this unbounded $v_2$-torsion to the kernel of the map (\ref{eq:telmap}).

We will assume the following optimistic conjecture in order to simplify our discussion.

\begin{conjecture}[Torsion Conjecture]\label{conj:btc}
The May-Ravenel spectral sequence collapses at $E_1$ with no hidden $v_2$-extensions.
\end{conjecture}

Then $H^{*,*}(\mc{C})$ has basis: 
\begin{align*}
\mr{(I')} & \quad v_2^m h^{\bar{\epsilon}_3}_{3,0}\td{h}_{2,1}^{\epsilon_2} h^{\epsilon_3}_{3,1} \td{h}^{\epsilon_4}_{4,1}, 
\shortintertext{\begin{flushright}
$m \ge 0; \: \epsilon_j, \bar{\epsilon}_j \in \{0,1\},$
\end{flushright}}
\mr{(I'')} & \quad v_2^{<2^{i+1}}h^{\bar{\epsilon}_3}_{3,0} x_i^{k_i+1} x_{i+1}^{k_{i+1}} x_{i+2}^{k_{i+2}} \cdots \td{h}_{2,1}^{\epsilon_2} h^{\epsilon_3}_{3,1} \td{h}^{\epsilon_4}_{4,1} h^{\epsilon_{i+3}}_{i+3,1} \cdots,  
\shortintertext{\begin{flushright}
$i \ge 3; \: k_j \ge 0; \: \epsilon_j, \bar{\epsilon}_j \in \{0,1\},$
\end{flushright}}
\mr{(II)} & \quad h^{\bar{\epsilon}_3}_{3,0}h^{\bar{\epsilon}_{i+3}}_{i+3,0}h^{\bar{\epsilon}_{i+4}}_{i+4,0} \cdots x_3^{k_3}x_4^{k_4} \cdots   \td{h}^{\epsilon_2}_{2,1} \cdots h_{i-1,1}^{\epsilon_{i-1}}h_{i,1}^{l_i+2} h^{l_{i+1}}_{i+1,1}\cdots,
\shortintertext{\begin{flushright}
$i \ge 2; \: k_j, l_j \ge 0; \: \epsilon_j, \bar{\epsilon}_j \in \{0,1\}$.
\end{flushright}}
\end{align*}
The long exact sequence (\ref{eq:LES}) implies that the unbounded $v_2$-torsion in $\E{\tmf}{}^{*,*}_2(Z)$ arises from the terms $\mr{(I')}$ and $\mr{(I'')}$ above.  Since the terms $\mr{(II)}$ above, as well as $H^{*,*}(V)$ are $v_2^1$-torsion, the elements of $\E{\tmf}{}^{*,*}_2(Z)$ not mapping to terms of the form $\mr{(I')}$ or $\mr{(I'')}$ are at most $v_2^2$-torsion.

The $d_4$-differentials of the Differentials Conjecture~\ref{conj:diffs} suggest the following analogous conjecture for the $\tmf$-ASS.

\begin{conjecture}[Differentials Conjecture, v2, part 1]\label{conj:diffsv2}
In the $\tmf$-ASS for $Z$ there are differentials
\begin{multline*}
d_3(v_2^{m}h^{\bar{\epsilon}_3}_{3,0} x_i^{k_i+1} x_{i+1}^{k_{i+1}} x_{i+2}^{\br{\epsilon}_{i+2}} x_{i+3}^{\br{\epsilon}_{i+3}} \cdots 
x_{l-1}^{k_{l-1}}x_{l}^{k_{l}} \cdots
\td{h}_{2,1}^{\epsilon_2} h^{\epsilon_3}_{3,1} \td{h}^{\epsilon_4}_{4,1} h_{l,1} h^{\epsilon_{l+1}}_{l+1,1} h^{\epsilon_{l+2}}_{l+2,1} \cdots) \\
= v_2^{m+1}h^{\bar{\epsilon}_3}_{3,0} x_i^{k_i+1} x_{i+1}^{k_{i+1}} x_{i+2}^{\br{\epsilon}_{i+2}} x_{i+3}^{\br{\epsilon}_{i+3}} \cdots 
x_{l-1}^{k_{l-1}+2}x_{l}^{k_{l}} \cdots
\td{h}_{2,1}^{\epsilon_2} h^{\epsilon_3}_{3,1} \td{h}^{\epsilon_4}_{4,1}  h^{\epsilon_{l+1}}_{l+1,1} h^{\epsilon_{l+2}}_{l+2,1} \cdots \\
+ \cdots
\end{multline*}
for $i \ge 3$, $l \ge i+3$, $m < 2^{i+1}-1$, $k_j \ge 0$, and $\epsilon_j, \bar{\epsilon}_j \in \{0,1\}$.  
\end{conjecture}

After running these $d_3$-differentials the only remaining classes in the $\tmf$-ASS for $Z$ are either $v_2^2$-torsion, or of the form
 \begin{align*}
 \mr{(I')} & \quad v_2^m h^{\bar{\epsilon}_3}_{3,0}\td{h}_{2,1}^{\epsilon_2} h^{\epsilon_3}_{3,1} \td{h}^{\epsilon_4}_{4,1}, 
 \shortintertext{\begin{flushright}
 $m \ge 0; \: \epsilon_j, \bar{\epsilon}_j \in \{0,1\},$
 \end{flushright}}
 \mr{(I''')} & \quad v_2^{<2^{i+1}} \td{h}_{2,1}^{\epsilon_1} h^{\epsilon_2}_{3,0} h^{\epsilon_3}_{3,1} \td{h}^{\epsilon_4}_{4,1} x_i^{k_i+1} x_{i+1}^{k_{i+1}} x_{i+2}^{\br{\epsilon}_{i+2}} x_{i+3}^{\br{\epsilon}_{i+3}} \cdots,  
 \shortintertext{\begin{flushright}
 $i \ge 3; \: k_j \ge 0; \: \epsilon_j, \bar{\epsilon}_j \in \{0,1\}.$
 \end{flushright}}
\end{align*}

\subsection{Parabolas}

In the $\tmf$-ASS for $Z$, we have differentials (Theorem~\ref{thm:diffsquotient})
$$ d_1(h_{i+2,1}) = v^{2^{i+1}}_2x_{i} $$
whereas in the Adams spectral sequence there are conjecturally differentials (Conjecture~\ref{conj:diffs})
$$ d_4(h_{i+2,1}) = v_2x^2_{i+1}. $$
This suggests the following.

\begin{conjecture}[Extension Conjecture]
In the $\tmf$-ASS there are hidden extensions
$$ v_2^{2^{i+1}}x_i = v_2x^2_{i+1}. $$
\end{conjecture}

This conjecture predicts that the $v_2$-torsion families of type $\mr{(I''')}$ of Conjecture~\ref{conj:diffsv2} in the $\tmf$-ASS for $Z$ actually assemble via an infinite sequence of hidden extensions to form $v_2$-periodic families in $\pi_*(Z)$.  Figure~\ref{fig:parabola} shows one such $v_2$-periodic family.

\begin{figure}
\centering
\includegraphics[width=1\linewidth]{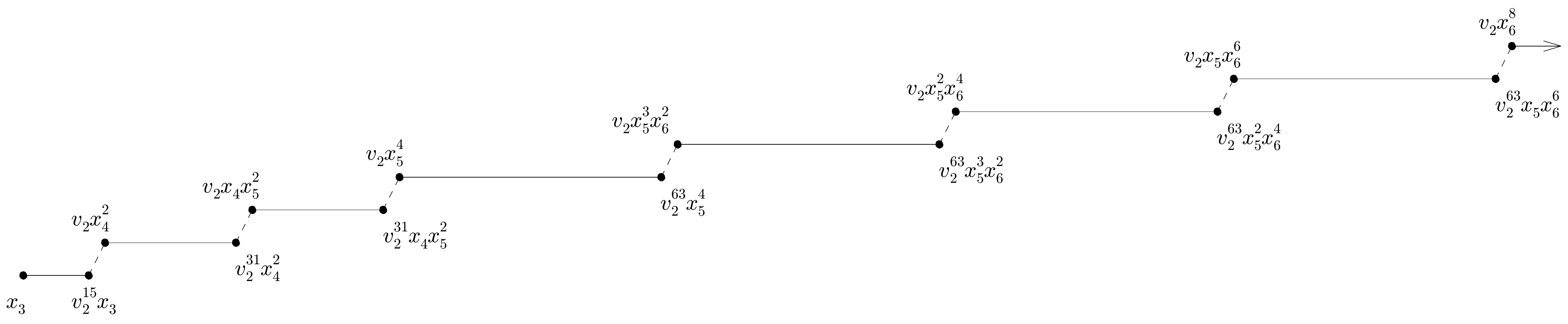}
\caption{The $v_2$-periodic family supported by $x_3$}
\label{fig:parabola}
\end{figure}
 
If we assign a mass to $x_i$ via
$$ M(x_i) := \frac{1}{2^{i-2}} $$
and, more generally for monomials
$$ M(v_2^m x_3^{k_3}x_4^{k_4} \cdots) := k_3M(x_3) + k_4M(x_4) + \cdots $$
then one finds that all of the terms of the form
$$ v_2^{m} x_i^{k_i+1}x_{i+1}^{k_{i+1}}x_{i+2}^{\bar{\epsilon}_{i+2}} x_{i+3}^{\bar{\epsilon}_{i+3}} \cdots $$
(for $i \ge 3$, $0 < m < 2^{i+1}$, $k_j \ge 0$, and $\bar{\epsilon}_j \in \{0,1\}$) lie in the same $v_2$-periodic family if and only if they have the same mass.  

Each of these $v_2$-periodic families begins with a term of the form
$$ x_3^{k_3} x_4^{\bar{\epsilon}_4} x_5^{\bar{\epsilon}_5} \cdots $$
(with $k_3 \ge 0$ and $\bar{\epsilon}_j \in \{0,1\}$) with corresponding mass
$$ M = \frac{k_3}{2}+\frac{\bar{\epsilon}_4}{4}+\frac{\bar{\epsilon}_5}{8} + \cdots. $$
Thus for each monomial
$$ \td{h}^{\epsilon_1}_{2,1}h^{\epsilon_2}_{3,0} h^{\epsilon_3}_{3,1}\td{h}_{4,1}^{\epsilon_4} \in E[\td{h}_{2,1}, h_{3,0}, h_{3,1}, \td{h}_{4,1}] $$
and each mass
$M \in \ZZ[1/2]^{>0}$ there is 
a corresponding non-trivial monomial
$$ x_3^{k_3} x_4^{\bar{\epsilon}_4} x_5^{\bar{\epsilon}_5} \cdots \in \FF_2[x_3]\otimes E[x_4, x_5, x_6, \cdots] $$
such that 
$$
\td{h}^{\epsilon_1}_{2,1}h^{\epsilon_2}_{3,0} h^{\epsilon_3}_{3,1}\td{h}_{4,1}^{\epsilon_4} x_3^{k_3} x_4^{\bar{\epsilon}_4} x_5^{\bar{\epsilon}_5} \cdots
$$
supports a $v_2$-family with mass $M$. 
For each of these $v_2$-families, the elements
$$ \td{h}^{\epsilon_1}_{2,1}h^{\epsilon_2}_{3,0} h^{\epsilon_3}_{3,1}\td{h}_{4,1}^{\epsilon_4} v_2 x_i^{2^{i-2}M} $$
represent a cofinal collection of elements which lie in the family. 
The elements $v_2 x_i^{2^{i-2}M}$
lie on the (sideways) parabola
\begin{equation}\label{eq:parabola}
 t-n = \frac{4}{M}n^2 - 3n+6
 \end{equation}
in the $(t-n,n)$-plane.  As such, we will refer to these $v_2$-families as $v_2$-\emph{parabolas}.

\subsection{The vanishing line}

Theorem~\ref{thm:MRE1} and Proposition~\ref{prop:evilvanishing} imply the following.

\begin{thm}\label{thm:vanishingline}
In the $\tmf$-ASS for $Z$, we have $\E{\tmf}{}_2^{n,t}(Z) = 0$ for
$$ n > \frac{t-n+12}{11}. $$
\end{thm}

Unfortunately, Conjecture~\ref{conj:btc} only predicts the bounded $v_2$-torsion in this $E_2$-term is $v_2^2$-torsion.  This means that the $v_2^2$-torsion could in principle assemble (via infinite sequences of hidden extensions) to detect non-trivial $v_2$-periodic families in $\pi_*Z$ which lie along curves with derivatives $\ge 1/12$ in the $(t-n,n)$-plane. Thus Theorem~\ref{thm:vanishingline} is not strong enough to preclude the bounded $v_2^2$-torsion contributing to the homotopy of $\widehat{Z}$.

This $1/11$ vanishing line essentially arises from the element $\td{h}_{2,1} \in H^{*,*}(\mc{C})$.\footnote{If one replaces $Z$ with the Thom spectrum $y(2)$ of \cite{MahowaldRavenelShick}, a similar analysis to Theorem~\ref{thm:vanishingline} easily yields a vanishing line of slope $1/13$.}  However, the results of \cite{AndersonDavis} imply that $\E{ass}{}^{*,*}_2(A_2)$ has a vanishing line of slope $1/13$.
Moreover, the element $h_{2,2}^4$ in the May spectral sequence (corresponding to $\td{h}^4_{2,1} \in H^{*,*}(\mc{C}(Z))$) detects the element $g_2 \in \E{ass}{}_2(S)$.  
The element $g_2$ is not nilpotent \cite{Isaksenstems}, but it detects the element $\bar{\kappa}_2 \in \pi_{44}(S)$ which necessarily is nilpotent by the Nishida nilpotence theorem.  It seems likely this can be used to prove the following, which would imply that the bounded $v_2^2$-torsion cannot contribute to the homotopy of $\widehat{Z}$.

\begin{conjecture}[Vanishing Line Conjecture]\label{conj:vanishingline}
There is an $r$ so that $\E{\tmf}{}^{n,t}_r(Z)$ has a $1/13$ vanishing line.
\end{conjecture} 

\subsection{The parabola conjecture}

Assuming all of the conjectures so far are true, the homotopy of $\widehat{Z}$ can only be detected by the $v_2$-periodic elements or the $v_2$-parabolas in $\E{\tmf}{}_4(Z)$.  We therefore are left to consider the possibility of differentials between these families.  The only possibilities are:
\begin{enumerate}
\item differentials between $v_2$-periodic elements,
\item differentials from $v_2$-periodic elements to $v_2$-parabolas,
\item differentials from a $v_2$-parabola of mass $M$ to a $v_2$-parabola of mass $M'$ with $M' > M$.
\end{enumerate}
Differentials of type (1) are ruled out by Theorem~\ref{thm:k2locss}.  Proposition~\ref{prop:higherdiffs} establishes that $\td{h}_{2,1}$, $h_{3,0}$, and $v_2^2h_{3,1}$ are permanent cycles in the $\tmf$-ASS.  While $Z$ is not a ring spectrum, one might nevertheless suspect that the $v_2$-families
$$  v_2^m \td{h}^{\epsilon_1}_{2,1}h^{\epsilon_2}_{3,0} h^{\epsilon_3}_{3,1} $$
cannot support differentials of type (2), and presumably this could be easily established be extending our low dimensional calculations a little further.  

We therefore turn to considering differentials of type (2) involving the element $\td{h}_{4,1}$.  Note that since $v_2^{-10}\td{h}_{4,1}$ detects $\zeta_2 \in \pi_{-1}Z_{E(2)}$, this is equivalent to the question of whether the element $\zeta_2 \in \pi_{-1}Z_{E(2)}$ lifts to $\pi_{-1}\widehat{Z}$ (compare with Remark~\ref{rmk:zeta}).    

We first note that Conjecture~\ref{conj:diffsv2} does not include the differential
$$ d_4 (h_{4,1}) = v_2x_3^2 $$
of Conjecture~\ref{conj:diffs}.  We therefore offer this second installment to Conjecture~\ref{conj:diffs} which does include this differential, and its consequences.

\begin{conjecture}[Differentials conjecture, v2, part 2]\label{conj:diffsv2p2}
In the $\tmf$-ASS, for $m \gg 0$, the $v_2$-families 
$$ v_2^m \td{h}^{\epsilon_1}_{2,1}h^{\epsilon_2}_{3,0} h^{\epsilon_3}_{3,1}\td{h}_{4,1} $$
support differentials which hit the $v_2$-parabolas supported by
$$ \td{h}^{\epsilon_1}_{2,1}h^{\epsilon_2}_{3,0} h^{\epsilon_3}_{3,1} x_3^2 $$
and the $v_2$-parabolas supported by
$$ 
\td{h}^{\epsilon_1}_{2,1}h^{\epsilon_2}_{3,0} h^{\epsilon_3}_{3,1}\td{h}_{4,1} x_3^{k_3} x_4^{\bar{\epsilon}_4} x_5^{\bar{\epsilon}_5} \cdots
$$
support differentials which hit the $v_2$-parabolas
$$ 
\td{h}^{\epsilon_1}_{2,1}h^{\epsilon_2}_{3,0} h^{\epsilon_3}_{3,1} x_3^{k_3+2} x_4^{\bar{\epsilon}_4} x_5^{\bar{\epsilon}_5} \cdots.
$$
\end{conjecture}

\begin{figure}
\centering
\includegraphics[width=1\linewidth]{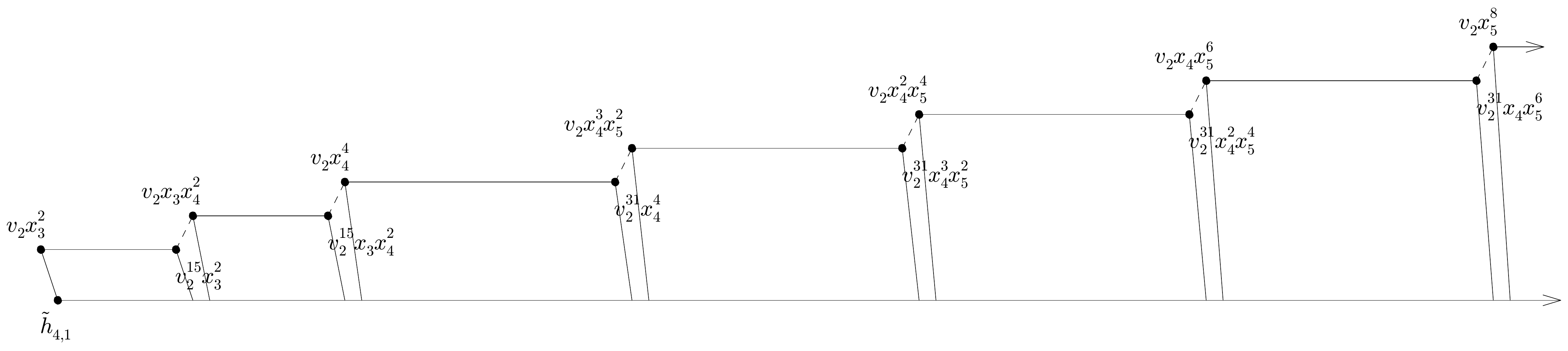}
\caption{The conjectural differentials on $v_2^m \td{h}_{4,1}$.}
\label{fig:zetadiffs}
\end{figure}

Figure~\ref{fig:zetadiffs} shows an example of such a family of differentials.  Note that the lengths of each of the families of differentials predicted by Conjecture~\ref{conj:diffsv2p2} are unbounded.  However, it could be that far enough out in the family, the differentials are all zero.  This could occur, for instance, if another parabola supporting shorter differentials kills the $v_2$-family which is the putative target.  Such a phenomenon would be a means for $\zeta_2$ to exist in $\pi_* \widehat{Z}$ without violating Conjecture~\ref{conj:diffsv2p2}.  

The following version of the parabola conjecture offers a maximally anti-telescope point of view, and is consistent with Conjecture~\ref{conj:parabola}.

\begin{conjecture}[Parabola Conjecture, v2]
The differentials of Conjecture~\ref{conj:diffsv2p2} are non-trivial, and all of the remaining $v_2$-parabolas have elements which are permanent cycles.  Thus the $v_2$-periodic homotopy of $Z$ is generated by the $v_2$-families
$$ v_2^m h^{\bar{\epsilon}_3}_{3,0}\td{h}_{2,1}^{\epsilon_2} h^{\epsilon_3}_{3,1}, \quad m \ge 0, \: {\epsilon}_j \in \{0,1\}, $$
and the $v_2$-parabolas are supported by
$$ h^{\bar{\epsilon}_3}_{3,0}\td{h}_{2,1}^{\epsilon_2} h^{\epsilon_3}_{3,1}
x_3^{\bar{\epsilon}_3}x_4^{\bar{\epsilon}_4} \cdots, \quad \epsilon_j, \bar{\epsilon}_j \in \{0,1\}. $$
\end{conjecture}

\begin{rmk}
Recent work of Carmeli-Schlank-Yanovski \cite{CSY} gives some circumstantial evidence that it could be the case that $\zeta_2 \in \pi_*Z_{E(2)}$ lifts to an element of $\pi_*\widehat{Z}$.  If this turns out to be true, then it flies in the face of the conventional wisdom on the subject, but it does not seem to necessarily force the telescope conjecture to be true.  Rather, it is totally possible that a weak form of the parabola conjecture is true, where the map
$$ \pi_*\widehat{Z} \rightarrow \pi_*Z_{E(2)} $$
is surjective with non-trivial kernel generated by a portion of the $v_2$-parabolas.
\end{rmk}


\appendix
\section{$A(2)$ as a module over the Steenrod algebra}\label{apx:data}

Here, we describe the $A$-module structure on $A(2)$ resulting from \cite[p. 30, Chapter III]{Roth} and present it as a definition file for Bruner's program
\cite{Bruner}. The definition file is a text file, where the first line is an integer $n$ which records the dimension of the $A$-module as an $\FF_2$-vector space. We should then interpret that we are given an ordered basis  $g_0, \ldots, g_{n-1}$. The second line of the text file is an ordered list of integers $d_0, \ldots, d_{n-1}$, where $d_i$ is the internal degree of $g_i$. For $A(2)$, the first two lines of Bruner's definition file reads as:
\begin{verbatim}
64 

0 1 2 3 3 4 4 5 5 6 6 6 7 7 7 7 8 8 8 9 9 9 9 10 10 10 10 10 11 11 11 
11 12 12 12 12 13 13 13 13 13 14 14 14 14 15 15 15 16 16 16 16 17 17 
17 18 18 19 19 20 20 21 22 23
\end{verbatim}

Every subsequent line in the text file describes a nontrivial action of some $Sq^k$ on some generator $g_i$. For example, if
\[Sq^k (g_i) = g_{j_1} + \ldots + g_{j_l} \]
we would encode this fact by writing the line
\[ i \ k \ l \ j_1 \  \ldots \ j_l \]
followed by a line break. Actions which are not indicated by such data are assumed to be trivial.

\newpage
\begin{multicols}{3}
\begin{small}
\begin{verbatim}


0 1 2 3 
3 4 4 5 
5 6 6 6 
7 7 7 7 
8 8 8 9
9 9 9 10 
10 10 10 10 
11 11 11 11 
12 12 12 12 
13 13 13 13 
13 14 14 14 
14 15 15 15 
16 16 16 16 
17 17 17 18 
18 19 19 20 
20 21 22 23

0 1 1 1
0 2 1 2
0 3 1 3
0 4 1 5
0 5 1 7
0 6 1 9
0 7 1 12
0 10 1 23
0 12 1 32
0 13 1 36
0 14 1 41
0 20 2 59 60
0 21 1 61

1 2 2 3 4
1 3 1 6 
1 4 2 7 8
1 5 1 10
1 6 2 12 13
1 7 1 16
1 8 1 19
1 9 1 23
1 12 1 36
1 14 1 45
1 15 1 48
1 20 1 61
1 22 1 63

2 1 1 3
2 2 1 6
2 4 3 9 10 11
2 5 2 12 14
2 6 2 16 17
2 7 1 20
2 8 1 24
2 9 1 28
2 10 1 32
2 11 1 36
2 12 2 41 42
2 14 1 49
2 15 1 52
2 16 1 55
2 18 2 59 60
2 19 1 61

3 2 1 8
3 3 1 10
3 4 2 12 14
3 6 3 19 20 21
3 7 2 23 25
3 8 2 28 29
3 9 1 33
3 10 2 36 37
3 11 1 41
3 12 1 45
3 13 1 48
3 20 1 63

4 1 1 6
4 2 1 8
4 3 1 10
4 4 2 13 15
4 5 2 16 18
4 6 2 19 22
4 7 2 23 26
4 8 1 30
4 9 1 34
4 10 1 38
4 11 1 42
4 12 1 45
4 13 1 48
4 14 1 52
4 16 1 57
4 17 1 59
4 18 1 61
4 20 1 63

5 1 1 7
5 2 2 9 10
5 3 1 12
5 4 1 17
5 5 1 20
5 6 2 23 25
5 8 2 34 35
5 9 1 39
5 10 2 42 43
5 11 1 46
5 12 2 48 49
5 13 1 52
5 16 1 59
5 18 1 62
5 19 1 63

6 2 1 10
6 4 2 16 18
6 6 3 23 26 27
6 7 1 31
6 8 2 34 35
6 9 1 39
6 10 2 42 43
6 11 1 46
6 12 2 48 49
6 13 1 52
6 16 1 59
6 18 1 62
6 19 1 63

7 2 2 12 13
7 3 1 16
7 4 2 19 20
7 5 1 23
7 6 1 29
7 7 1 33
7 8 1 39

8 1 1 10
8 4 3 19 21 22
8 5 3 23 25 26
8 6 2 29 31
8 7 1 33
8 8 4 37 38 39 40
8 9 3 41 42 44
8 10 3 45 46 47
8 11 2 48 50
8 12 2 52 53
8 13 1 55
8 14 1 57
8 15 1 59
8 18 1 63

9 1 1 12
9 2 1 16
9 4 2 23 24
9 5 1 28
9 6 2 32 33
9 7 1 36
9 8 3 41 42 43
9 9 1 46

10 4 3 23 25 26
10 6 2 33 35
10 7 1 39
10 8 3 41 42 44
10 10 4 48 49 50 51
10 11 2 52 54
10 12 2 55 56
10 13 1 58
10 14 2 59 60
10 15 1 61
10 16 1 62
10 17 1 63

11 1 1 14
11 2 1 17
11 3 1 20
11 4 2 24 27
11 5 2 28 31
11 6 2 32 35
11 7 2 36 39
11 8 2 42 43
11 9 1 46
11 10 1 48
11 12 1 55
11 14 2 59 60
11 15 1 61
11 16 1 62
11 17 1 63

12 2 1 19
12 3 1 23
12 4 1 28
12 6 2 36 37
12 7 1 41
12 8 2 45 46
12 9 1 48
12 10 1 52

13 1 1 16
13 2 1 19
13 3 1 23
13 4 2 29 30
13 5 2 33 34
13 6 2 38 39
13 7 1 42
13 8 2 46 47
13 9 1 50
13 10 1 54

14 2 2 20 21
14 3 1 25
14 4 3 28 29 31
14 5 1 33
14 6 3 36 37 39
14 7 1 41
14 8 2 45 46
14 9 1 48
14 10 1 52

15 1 1 18
15 2 1 22
15 3 1 26
15 4 1 30
15 5 1 34
15 6 1 38
15 7 1 42
15 10 1 52
15 12 1 57
15 13 1 59
15 14 1 61

16 2 1 23
16 4 2 33 34
16 6 2 42 43
16 7 1 46
16 8 2 49 50
16 9 1 52
16 10 1 56
16 11 1 58

17 1 1 20
17 2 1 25
17 4 3 32 33 35
17 5 2 36 39
17 6 1 41
17 8 1 51
17 9 1 54
17 10 1 56
17 11 1 58
17 12 1 59
17 14 1 62
17 15 1 63

18 2 2 26 27
18 3 1 31
18 4 2 34 35
18 5 1 39
18 6 2 42 43
18 7 1 46
18 8 1 49
18 9 1 52
18 12 1 59
18 14 1 62
18 15 1 63

19 1 1 23
19 4 2 37 38
19 5 2 41 42
19 6 2 45 46
19 7 1 48
19 8 1 53
19 9 1 55
19 10 2 57 58
19 11 1 59
19 12 1 61

20 2 1 29
20 3 1 33
20 4 2 36 39
20 6 1 45
20 7 1 48
20 8 1 54
20 10 1 58

21 1 1 25
21 2 1 29
21 3 1 33
21 4 2 37 40
21 5 2 41 44
21 6 2 45 47
21 7 2 48 50
21 8 2 52 53
21 9 1 55
21 10 1 57
21 11 1 59
21 12 1 61

22 1 1 26
22 2 1 31
22 4 3 38 39 40
22 5 2 42 44
22 6 2 46 47
22 7 1 50
22 8 1 53
22 9 1 55
22 10 1 57
22 11 1 59
22 12 1 61

23 4 2 41 42
23 6 2 48 49
23 7 1 52
23 8 1 55
23 10 2 59 60
23 11 1 61
23 12 1 62
23 13 1 63

24 1 1 28
24 2 2 32 33
24 3 1 36
24 4 1 43
24 5 1 46
24 6 2 48 49
24 7 1 52
24 8 1 56
24 9 1 58
24 12 1 62
24 13 1 63

25 2 1 33
25 4 2 41 44
25 6 3 48 50 51
25 7 1 54
25 8 2 55 56
25 9 1 58
25 10 2 59 60
25 11 1 61
25 12 1 62
25 13 1 63

26 2 1 35
26 3 1 39
26 4 2 42 44
26 6 3 49 50 51
26 7 2 52 54
26 8 2 55 56
26 9 1 58
26 10 2 59 60
26 11 1 61
26 12 1 62
26 13 1 63

27 1 1 31
27 2 1 35
27 3 1 39
27 4 1 43
27 5 1 46
27 6 1 49
27 7 1 52
27 12 1 62
27 13 1 63

28 2 2 36 37
28 3 1 41
28 4 2 45 46
28 5 1 48
28 6 1 52
28 8 1 58
28 12 1 63

29 1 1 33
29 4 2 45 47
29 5 2 48 50
29 6 1 54
29 8 2 57 58
29 9 1 59
29 10 1 61
29 12 1 63

30 1 1 34
30 2 2 38 39
30 3 1 42
30 4 1 47
30 5 1 50
30 6 2 52 54
30 12 1 63

31 2 1 39
31 4 1 46
31 6 1 52
31 12 1 63

32 1 1 36
32 2 1 41
32 4 2 48 49
32 5 1 52
32 8 1 60
32 9 1 61
32 10 1 62
32 11 1 63

33 4 2 48 50
33 6 1 56
33 7 1 58
33 8 1 59
33 10 1 62
33 11 1 63

34 2 2 42 43
34 3 1 46
34 4 2 49 50
34 5 1 52
34 6 1 56
34 7 1 58

35 1 1 39
35 4 2 49 51
35 5 2 52 54
35 6 1 56
35 7 1 58
35 8 1 60
35 9 1 61
35 10 1 62
35 11 1 63

36 2 1 45
36 3 1 48
36 4 1 52
36 8 1 61
36 10 1 63

37 1 1 41
37 2 1 45
37 3 1 48
37 4 1 53
37 5 1 55
37 6 2 57 58
37 7 1 59
37 10 1 63

38 1 1 42
38 2 1 46
38 4 2 52 53
38 5 1 55
38 6 2 57 58
38 7 1 59
38 8 1 61

39 4 2 52 54
39 6 1 58
39 8 1 61
39 10 1 63

40 1 1 44
40 2 1 47
40 3 1 50
40 4 1 53
40 5 1 55
40 6 1 57
40 7 1 59
40 10 1 63

41 2 1 48
41 4 1 55
41 6 2 59 60
41 7 1 61
41 8 1 62
41 9 1 63

42 2 1 49
42 3 1 52
42 4 1 55
42 6 2 59 60
42 7 1 61
42 8 1 62
42 9 1 63

43 1 1 46
43 2 1 49
43 3 1 52
43 4 1 56
43 5 1 58

44 2 2 50 51
44 3 1 54
44 4 2 55 56
44 5 1 58
44 6 2 59 60
44 7 1 61
44 8 1 62
44 9 1 63

45 1 1 48
45 4 1 57
45 5 1 59
45 6 1 61
45 8 1 63

46 2 1 52
46 4 1 58

47 1 1 50
47 2 1 54
47 4 2 57 58
47 5 1 59
47 6 1 61

48 4 1 59
48 6 1 62
48 7 1 63

49 1 1 52
49 4 1 60
49 5 1 61
49 6 1 62
49 7 1 63

50 2 1 56
50 3 1 58
50 4 1 59
50 6 1 62
50 7 1 63

51 1 1 54
51 2 1 56
51 3 1 58
51 4 1 60
51 5 1 61
51 6 1 62
51 7 1 63

52 4 1 61
52 6 1 63

53 1 1 55
53 2 2 57 58
53 3 1 59
53 6 1 63

54 2 1 58
54 4 1 61
54 6 1 63

55 2 2 59 60
55 3 1 61
55 4 1 62
55 5 1 63

56 1 1 58
56 4 1 62
56 5 1 63

57 1 1 59
57 2 1 61
57 4 1 63

58 4 1 63

59 2 1 62
59 3 1 63

60 1 1 61
60 2 1 62
60 3 1 63

61 2 1 63

62 1 1 63

\end{verbatim}
\end{small}

%
\end{multicols}

\bibliographystyle{amsalpha}
\bibliography{tmfASSZ2}
\end{document}